\documentclass{amsart}

\usepackage[utf8]{inputenc}
\usepackage[margin=1.2in]{geometry}
\usepackage[english]{babel}
\usepackage{graphicx}
\usepackage{amsmath}
\usepackage{amsfonts}
\usepackage{amssymb}
\usepackage{amsthm}
\usepackage{amsfonts}
\usepackage{mathtools}
\usepackage{tabu}
\usepackage{verbatim}
\usepackage{float}
\usepackage{tikz-cd}
\usetikzlibrary{positioning}
\usetikzlibrary{shapes.geometric}
\usetikzlibrary{calc}
\usepackage{subcaption}
\usepackage{bbm}
\usepackage{graphicx}
\usepackage{bbm}
\usepackage{enumitem}
\usepackage{todonotes}

\usepackage[unicode=true]{hyperref}
\usepackage{changepage}
\usepackage[numbers]{natbib}
\hypersetup{
  colorlinks   = true, %Colours links instead of ugly boxes
  urlcolor     = blue, %Colour for external hyperlinks
  linkcolor    = blue, %Colour of internal links
  citecolor   = red %Colour of citations
}

\theoremstyle{plain}
\newtheorem{theorem}{Theorem}[section]
\newtheorem{lemma}[theorem]{Lemma}
\newtheorem{proposition}[theorem]{Proposition}
\newtheorem{corollary}[theorem]{Corollary}

\theoremstyle{definition}
\newtheorem{definition}[theorem]{Definition}
\newtheorem{remark}[theorem]{Remark}
\newtheorem{example}[theorem]{Example}

\newtheorem{question}[theorem]{Question}
\newtheorem{convention}[theorem]{Convention}

\makeatletter
\newcommand{\bigoperp}{%
  \mathop{\mathpalette\bigp@rp\relax}%
  \displaylimits
}
\newcommand{\bigp@rp}[2]{%
  \vcenter{
    \m@th\hbox{$#1\raisebox{2.4pt}{\begin{tikzpicture}[baseline,scale={\ifx#1\displaystyle1.4\else1\fi}]%
\draw[line width=0.66pt] (0,0) circle (.165);%
\draw[line width=0.66pt] (0,.165) -- (0,-.07);%
\draw[line width=0.66pt] (-.15,-.07) -- (.15,-.07);%
\end{tikzpicture}}$}}%
}
\newcommand{\operp}{%
  \mathop{\mathpalette\op@rp\relax}%
}
\newcommand{\op@rp}[2]{
#1\raisebox{2.4pt}{\begin{tikzpicture}[baseline,scale={\ifx#1\displaystyle.69\else.69\fi}]%
\draw[line width=0.27pt] (0,0) circle (.165);%
\draw[line width=0.27pt] (0,.165) -- (0,-.07);%
\draw[line width=0.27pt] (-.15,-.07) -- (.15,-.07);%
\end{tikzpicture}}
}
\makeatother

\newcommand{\Z}		{\ensuremath{\mathbb{Z}}}
\newcommand{\Q}		{\ensuremath{\mathbb{Q}}}
\newcommand{\R}		{\ensuremath{\mathbb{R}}}
\renewcommand{\C}		{\ensuremath{\mathbb{C}}}
\newcommand{\F}		{\ensuremath{\mathbb{F}}}

\newcommand{\p}		{\ensuremath{\mathfrak{p}}}

\renewcommand{\i}	{\textup{i}}
\newcommand{\ee}	{\textup{e}}

\DeclareMathOperator{\Hom}{Hom}

\DeclareMathOperator{\vol}{vol}
\DeclareMathOperator{\Tr}{Tr}
\DeclareMathOperator{\rk}{rk}
\DeclareMathOperator{\Hdim}{orth\,dim}
\DeclareMathOperator{\X}{X}
\DeclareMathOperator{\vor}{Vor}

\newcommand{\vorc}{\overline{\vor}}

\DeclarePairedDelimiter\Kab{\lvert}{\rvert_K}

\newcommand{\dif}{\mathop{}\!\mathrm{d}}
\newcommand{\tensor}{\otimes}
\newcommand{\emptyinner}{\langle\,\cdot\,,\,\cdot\,\rangle}

\date{\today}
\title[Indecomposable algebraic integers]{Indecomposable algebraic integers}
\author[D.\ M.\ H. van Gent]{D.\ M.\ H. van Gent}
\address{Mathematisch Instituut, Universiteit Leiden, The Netherlands}
\email{d.m.h.van.gent@math.leidenuniv.nl}

\begin{document}

\begin{comment}
\begin{abstract} 
[ABSTRACT]
\end{abstract}
\end{comment}

\maketitle

\tableofcontents{}

\section{Introduction}

In number theory, in particular the theory of the geometry of numbers, one equips an order \(R\), and by extension its field of fractions, with an inner product, turning \(R\) into a lattice in \(\R\tensor_\Z R\). 
After normalizing this inner product, we may define it on an algebraic closure \(\overline{\Q}\) of \(\Q\) as
\[\langle\alpha,\beta\rangle=\frac{1}{[\Q(\alpha,\beta):\Q]}\sum_{\sigma:\Q(\alpha,\beta)\to\C}\sigma(\alpha)\cdot\overline{\sigma(\beta)},\]
where the sum ranges over all field embeddings of \(\Q(\alpha,\beta)\) in \(\C\).
We write \(\overline{\Z}\) for the ring of integers of \(\overline\Q\) and we call its elements the algebraic integers.
Although \(\overline{\Z}\) is not of finite rank, we may still meaningfully call it a lattice in a more general sense: We show it is a discrete subgroup of a Hilbert space (Theorem~\ref{thm:Zbar_is_lattice}).
In this document we prove some basic facts about such `Hilbert lattices', but we are mainly concerned with \(\overline{\Z}\).

A {\em decomposition} of an element \(\alpha\) of a Hilbert lattice \(\Lambda\) is a pair \((\beta,\gamma)\in\Lambda^2\) such that \(\alpha=\beta+\gamma\) and \(\langle \beta,\gamma\rangle\geq0\).
If \(\alpha\) is non-zero and has only the trivial decompositions \((\alpha,0)\) and \((0,\alpha)\), then \(\alpha\) is called {\em indecomposable}, or {\em Voronoi-relevant}. 
Indecomposable elements are of interest because they determine the Voronoi polyhedron (Theorem~\ref{thm:minimal_indec_set}), and can be used to detect whether \(\Lambda\) is {\em indecomposable} (Theorem~\ref{thm:eichler}), meaning \(\Lambda\) is non-zero and cannot be written as a non-trivial orthogonal sum of two sublattices. 
Moreover, indecomposable elements can be used in the context of gradings of orders \cite{Lenstra2018}.
We prove the following theorem in Section~\ref{sec:Zbar_indecomposable}.

\begin{theorem}\label{thm:Zbar_indecomposable}
The Hilbert lattice of algebraic integers is indecomposable.
\end{theorem}

Our main goal is to determine the indecomposable algebraic integers.
Formally we ask the question: Does there exist an algorithm that, given a non-zero algebraic integer, decides whether it is indecomposable and if not produces a non-trivial decomposition? We give a partial answer.
The following is a result derived from classical capacity theory by T. Chinburg, for which we give a proof in Section~\ref{sec:szego}.
\begin{theorem}\label{thm:large}
If \(\alpha\in\overline{\Z}\) satisfies \(4 < \langle\alpha,\alpha\rangle\), then \(\alpha\) has infinitely many decompositions.
\end{theorem}
\noindent The proof of this theorem is constructive and we are able to derive an algorithm to compute arbitrarily many non-trivial decompositions in this case.
A consequence of Theorem~\ref{thm:large} is that for every \(n\in\Z_{>0}\) there exist only finitely many indecomposable algebraic integers of degree \(n\). 
We have effective upper and lower bounds, which are proven in Section~\ref{sec:analysis}.
\begin{theorem}\label{thm:counting_asymptotic}
There are least \(\exp( \frac{1}{4}(\log 2) n^2 + O(n\log n))\) and at most \(\exp(\frac{1}{2}(1+\log 2)n^2+O(n\log n))\) indecomposable algebraic integers of degree up to \(n\).
\end{theorem}

Our main result is the following, which we prove in Section~\ref{sec:fekete}.

\begin{theorem}\label{thm:small}
If \(\alpha\in\overline{\Z}\) satisfies \( \langle\alpha,\alpha\rangle < 2\sqrt{\ee} \approx 3.297\), then \(\alpha\) has only finitely many decompositions.
\end{theorem}
\noindent The proof is inspired by capacity theory. Although the proof is not constructive, we are able to construct an algorithm to compute all decompositions which is provably correct for \(\langle \alpha,\alpha\rangle < 2\sqrt{\ee}\).
We use it in Section~\ref{sec:enum3} to compute decompositions for some algebraic integers of degree 3. We derive the following numerical results.

\begin{theorem}\label{thm:counting}
There are exactly \(2\) indecomposable algebraic integers of degree \(1\), there are exactly \(14\) of degree \(2\), and there are at least \(354\) and at most \(588\) of degree \(3\).
\end{theorem}

Theorem~\ref{thm:large} and Theorem~\ref{thm:small} combined still leave a gap between \(2\sqrt{\ee}\) and \(4\). It is still an open problem what happens in this case. 
We would also like to know what the supremum of \(\langle\alpha,\alpha\rangle\) is, where \(\alpha\) ranges over the indecomposable algebraic integers.
Theorem~\ref{thm:large} gives an upper bound of \(4\).
We prove the following lower bound in Section~\ref{sec:indecomposable_integers}.

\begin{theorem}\label{thm:largest_indec}
It holds that \(2\sqrt{2}\leq\sup\{ \langle\alpha,\alpha\rangle\,|\,\textup{\(\alpha\in\overline{\Z}\) is indecomposable} \}\).
\end{theorem}

A {\em grading} of a ring \(R\) is a pair \((G,\{R_g\}_{g\in G})\) such that \(G\) is an abelian group, \(R_g\) is a subgroup of \(R\) for all \(g\in G\) such that the natural map \(\bigoplus_{g\in G} R_g\to R\) is an isomorphism of groups, and such that for all \(g,g'\in G\) we have \(R_g R_{g'} \subseteq R_{gg'}\). 
We say a grading \((G,\{R_g\}_g)\) of a ring \(R\) is {\em universal} if for each grading \((H,\{S_h\}_h)\) of \(R\) there exists a unique group homomorphism \(f:G\to H\) such that \(S_h = \sum_{g\in f^{-1}h} R_g\) for all \(h\in H\).
Using the theory of Hilbert lattices we generalize two results of Lenstra and Silverberg \citep{Lenstra2018} in Section~\ref{sec:graded_rings}.

\begin{theorem}\label{thm:universal_grading}
Every subring of \(\overline{\Z}\) has a universal grading.
\end{theorem}

The universal grading of \(\overline{\Z}\) is the trivial grading, which is a consequence of Theorem~\ref{thm:Zbar_indecomposable}.

\begin{theorem}\label{thm:integral_universal_grading}
Every integrally closed subring of \(\overline{\Z}\) has a universal grading with a subgroup of \(\Q/\Z\).
\end{theorem}

\noindent We also show that every subgroup of \(\Q/\Z\) occurs for some integrally closed subring of \(\overline{\Z}\).

The structure of this document is as follows. 
In sections \ref{sec:begin-hilbert} through \ref{sec:end-hilbert} we develop general theory on Hilbert lattices.
We specialize to the Hilbert lattice \(\overline{\Z}\) in sections \ref{sec:begin-integers} through \ref{sec:end-integers}.
In sections \ref{sec:begin-application} through \ref{sec:end-application} we treat some applications of the theory.
In section \ref{sec:enum2} and \ref{sec:enum3} we do some explicit computations on indecomposable integers of degree 2 respectively 3.
Sections \ref{sec:begin-capacity-large} through \ref{sec:end-capacity-large} we dedicate to the proof of Theorem~\ref{thm:large} and its algorithmic counterpart.
In sections \ref{sec:begin-capacity-small} through \ref{sec:end-capacity-small} we prove Theorem~\ref{thm:small} and its algorithmic counterpart.

\section{Inner products and Hilbert spaces}\label{sec:begin-hilbert}

\begin{definition}
Let \(R\subseteq\C\) be a subring.
An {\em \(R\)-norm} on an \(R\)-module \(M\) is a map \(\|\cdot\|:M\to\R_{\geq 0}\) that satisfies:
\begin{itemize}[align=right]
\setlength{\itemindent}{7.2em}
\item[(\textit{Absolute homogeneity})] For all \(x\in M\) and \(a\in R\) we have \(\|ax\|=|a|\cdot\|x\|\);
\item[(\textit{Triangle inequality})] For all \(x,y\in M\) we have \(\|x+y\|\leq\|x\|+\|y\|\);
\item[(\textit{Positive-definiteness})] For all non-zero \(x\in M\) we have \(\|x\|\in \R_{>0}\).
\end{itemize}
A {\em normed \(R\)-module} is an \(R\)-module \(M\) together with an \(R\)-norm on \(M\).
For normed \(R\)-modules \(M\) and \(N\) an \(R\)-module homomorphism \(f:M\to N\) is called an {\em isometric map} if \(\|x\|=\|f(x)\|\) for all \(x\in M\).
The isometric maps are the morphisms in the category of normed \(R\)-modules.
\end{definition}

Note that an isometric map is injective, but not necessarily surjective.

\begin{definition}
Let \(R\subseteq\C\) be a subring and \(M\) be an \(R\)-module.
An {\em \(R\)-inner product} on \(M\) is a map \(\emptyinner:M^2\to\C\) that satisfies:
\begin{itemize}[align=right]
\setlength{\itemindent}{7.2em}
\item[(\textit{Conjugate symmetry})] For all \(x,y\in M\) we have \(\langle x,y\rangle = \overline{\langle y,x\rangle}\);
\item[(\textit{Left linearity})] For all \(x,y,z\in M\) and \(a\in R\) we have \(\langle x+ay,z\rangle=\langle x,z\rangle+a\langle y,z\rangle\);
\item[(\textit{Positive-definiteness})] For all non-zero \(x\in M\) we have \(\langle x,x\rangle \in \R_{>0}\).
\end{itemize}
We say it is a {\em real} inner product if \(\langle M,M\rangle\subseteq \R\), which implies \(R\subseteq \R\) when \(M\neq 0\).
An {\em \(R\)-inner product space} is an \(R\)-module together with an \(R\)-inner product.
\end{definition}

\begin{remark}\label{rem:norm}
An \(R\)-inner product space \(M\) comes with an \(R\)-norm given by \(\|x\|=\sqrt{\langle x,x\rangle}\), which in turn induces a metric \(d(x,y)=\|x-y\|\) and a topology. One can then speak about the completeness of \(M\) with respect to this metric.
\end{remark}

\begin{lemma}\label{lem:parallelogram_law}
Suppose \(R\subseteq \C\) is a subring and \(M\) is a real \(R\)-inner product space. Then the induced norm satisfies the {\em parallelogram law}: For all \(x,y\in M\) we have 
\begin{align*}
\|x+y\|^2+\|x-y\|^2=2\|x\|^2+2\|y\|^2. \hfill\qed
\end{align*}
\end{lemma}

The following is an exercise in many standard texts.

\begin{theorem}[Jordan--von Neumann \citep{parallelogram}]\label{thm:parallelogram}
Let \(R\subseteq \Q\) be a subring, \(M\) an \(R\)-module and suppose a map \(\|\cdot\|:M\to \R_{\geq 0}\) satisfies positive-definiteness and the parallelogram law.
Then \(\|\cdot\|\) is an \(R\)-norm on \(M\) induced by a real \(R\)-inner product \(\emptyinner:M^2\to\R\) given by
\[ \langle x, y\rangle = \tfrac{1}{2}( \|x+y\|^2 - \|x\|^2 - \|y\|^2 ). \]
\end{theorem}
\begin{proof}
Note that taking \(x=y=0\) in the parallelogram law shows \(2\|0\|^2=4\|0\|^2\), hence \(\|0\|=0\).
For all \(x\in M\) we have \(\|x+x\|^2=2\|x\|^2+2\|x\|^2-\|x-x\|^2=4\|x\|^2\), hence \(\langle x,x\rangle = \tfrac{1}{2} ( \|2x\|^2-2\|x\|^2 ) = \|x\|^2 \).
It now suffices to show that \(\emptyinner\) is an inner product, as \(\|\cdot\|\) is then the associated norm as in Remark~\ref{rem:norm}.
Clearly \(\emptyinner\) satisfies conjugate symmetry and positive definiteness, so it remains to prove left linearity.
It suffices to show for all \(x\in M\) that \(x\mapsto\langle x,z\rangle\) is \(\Z\)-linear: 
Since \(R\) is in the field of fractions of \(\Z\), any \(\Z\)-linear map to \(\R\) is also \(R\)-linear.
Let \(x,y,z\in M\) and note that \(\langle x,y\rangle = \tfrac{1}{4}(\|x+y\|^2-\|x-y\|^2)\).
By the parallelogram law we have
\begin{align*}
2\|y+z\|^2+2\|x\|^2-\|-x+y+z\|^2=\|x+y+z\|^2=2\|x+z\|^2+2\|y\|^2-\|x-y+z\|^2.
\end{align*}
so
\begin{align}\label{eq:jordan}
2\|x+y+z\|^2 = 2\|x+z\|^2+2\|y+z\|^2+2\|x\|^2+2\|y\|^2-\|-x+y+z\|^2-\|x-y+z\|^2.
\end{align}
Applying (\ref{eq:jordan}) also with \(z\) replaced by \(-z\), we obtain
\begin{align*}
8 \langle x+y,z \rangle &=2\|x+y+z\|^2-2\|x+y-z\|^2 \\
&= 2\|x+z\|^2+2\|y+z\|^2-2\|x-z\|^2-2\|y-z\|^2 \\
&= 8\langle x,z\rangle+8\langle y,z \rangle,
\end{align*}
as was to be shown.
We conclude that \(\emptyinner\) is in fact an \(R\)-inner product.
\end{proof}

Inner product spaces over \(\Z\) or \(\Q\) can be extended to \(\R\) in a `canonical' way.
This can best be expressed in a categorical sense in terms of universal morphisms.
We proceed as in Chapter III of \cite{categories}.

\begin{definition}
Let \(\mathcal{C}\) be a category. An object \(U\) of \(\mathcal{C}\) is called {\em universal} if for each object \(X\) of \(\mathcal{C}\) there exists a unique morphism \(U\to X\) in \(\mathcal{C}\).
\end{definition}

\begin{definition}
Let \(\mathcal{C}\) and \(\mathcal{D}\) be categories. 
Let \(F:\mathcal{C}\to\mathcal{D}\) be a functor and \(Z\) an object of \(\mathcal{D}\). 
For \(X\) an object of \(\mathcal{C}\) we say \(\eta\in\Hom_\mathcal{D}(Z,F(X))\) is a {\em universal morphism} from \(Z\) to \(F\) if for all objects \(Y\) of \(\mathcal{C}\) and every \(g\in\Hom_{\mathcal{D}}(Z,F(Y))\) there exists a unique \(f\in \Hom_{\mathcal{D}}(X,Y)\) such that \(F(f) \circ \eta = g\).
Equivalently, for all objects \(Y\) of \(\mathcal{C}\) and morphisms \(g:Z\to F(Y)\) we have the following diagram:
\end{definition}
\begin{center}
\begin{tikzcd}
Z \arrow{r}{\eta} \arrow[swap]{dr}{g} & F(X) \arrow[dashed]{d}{F(f)} & X \arrow[dashed]{d}{!f} \\ & F(Y) & Y
\end{tikzcd}
\vspace{1em}
\end{center}

For a reader familiar with category theory we remark that, if for a functor \(F:\mathcal{C}\to\mathcal{D}\) every object \(Z\) of \(\mathcal{D}\) has a universal morphism, then \(F\) is a {\em right adjoint} functor.
For a reader less familiar with category theory we leave it as an exercise to construct, given a functor \(F:\mathcal{C}\to\mathcal{D}\) and an object \(Z\) of \(\mathcal{D}\), a category for which the universal objects naturally correspond to the universal morphisms from \(Z\) to \(F\), justifying the name universal morphism.

\begin{example}\label{ex:universal_morphism}
We will give a concrete example of a universal morphism. 

(i) Consider the forgetful functor \(F\) from the category of \(\Q\)-vector spaces to the category of abelian groups, i.e. the functor that sends a \(\Q\)-vector space to its underlying abelian group.
Now let \(A\) be an abelian group.
We take \(V=\Q\tensor_\Z A\), which is a \(\Q\)-vector space, and \(\eta:A\to F(V)\) the map \(z\mapsto 1\tensor z\).
Because \(F\) is a forgetful functor, as will always be the case in our applications, we may omit it in the notation for simplicity and state that \(\eta\) is a morphism \(A\to V\) of abelian groups.
One can verify that \(\eta\) is a universal morphism.

Since \(\eta\) is universal the vector space \(V\) corresponding to \(A\) is `uniquely unique', meaning that any other other vector space \(V'\) with a universal morphism \(\eta':A\to V'\) gives rise to a unique isomorphism \(\varphi:V\to V'\) such that \(\varphi\circ\eta=\eta'\). 

Note that \(\eta\) need not be injective, as this is only the case when \(A\) is torsion-free. Then \(\eta\) can be thought of as a canonical embedding.

(ii) Similarly, we can consider a forgetful functor \(F\) from the category of \(\Q\)-inner product spaces to the category of \(\Z\)-inner product spaces. 
The underlying \(\Q\)-vector space \(V\) and map \(\eta\) are the same as before, and we equip \(V\) with the inner product we extend \(\Q\)-linearly from \(A\).
To show that this inner product is positive definite we use that \(A\) is torsion-free, being a \(\Z\)-inner product space.
\end{example}

\begin{definition}
A {\em Hilbert space} is an \(\R\)-module \(\mathcal{H}\) equipped with a real \(\R\)-inner product such that \(\mathcal{H}\) is complete with respect to the induced metric. The morphisms of Hilbert spaces are the isometric maps.
\end{definition}

\begin{theorem}[Theorem~3.2-3 in \citep{completions}]\label{thm:completion_inner_product}
Let \(F\) be the forgetful functor from the category of Hilbert spaces to the category of \(\Q\)-inner product spaces.
Then every \(\Q\)-inner product space \(V\) has an injective universal morphism to \(F\), and a morphism \(f:V\to \mathcal{H}\) for some Hilbert space \(\mathcal{H}\) is universal precisely when the image of \(f\) is dense in \(\mathcal{H}\). \qed
\end{theorem}

The Hilbert space constructed for \(V\) in Theorem~\ref{thm:completion_inner_product} can be obtained as the topological completion of \(V\), similar to how \(\R\) is constructed from \(\Q\) using Cauchy sequences.

\begin{definition}\label{def:pnorm}
For a set \(B\) and \(p\in\R_{>0}\) we define the \(\R\)-vector space
\[\ell^p(B) = \Big\{ (x_b)_{b} \in \R^B \,\Big|\, x_b = 0 \text{ for all but countably many \(b\in B\) and } \sum_{b\in B} |x_b|^p < \infty \Big\} \vspace{-.7em}\]
and \(\|x\|_p = ( \sum_{b\in B} |x_b|^p )^{1/p}\) for all \(x=(x_b)_b\in\ell^p(B)\).
\end{definition}

\begin{theorem}[Minkowski's inequality, Theorem~1.2-3 in \citep{completions}]
For any set \(B\) and \(p\in\R_{\geq 1}\) the map \(\|\cdot\|_p\) is an \(\R\)-norm on \(\ell^p(B)\). \qed
\end{theorem}

\begin{lemma}[Example 3.1-6 in \citep{completions}]
For any set \(B\) the space \(\ell^2(B)\) is a Hilbert space with inner product given by \(\langle x,y \rangle = \sum_{b\in B} x_b \cdot y_b \) for \(x=(x_b)_b,y=(y_b)_b\in\ell^2(B)\), such that \(\langle x,x\rangle=\|x\|_2^2\).\qed
\end{lemma}

\begin{lemma}\label{lem:holder}
Let \(n\in\Z_{\geq 1}\), \(x\in\R^n\) and let \(0<p\leq q\) be real. Then we have  
\[\|x\|_q\leq \|x\|_p \quad\text{and}\quad n^{-1/p}\cdot \|x\|_p \leq n^{-1/q}\cdot\|x\|_q.\]
\end{lemma}
\begin{proof}
Clearly we may assume \(x\neq 0\).
For the first inequality, consider \(y=x / \|x\|_p\). Then \(|y_i|\leq 1\) for all \(i\), from which \(|y_i|^q\leq |y_i|^p\) follows.
Now 
\[ \|y\|_q^q = \sum_{i=1}^n |y_i|^q \leq \sum_{i=1}^n |y_i|^p = \|y\|_p^p = 1. \]
Hence \(\|x\|_q/ \|x\|_p = \|y\|_q\leq 1\), as was to be shown.
For the second inequality, note that \(x\mapsto x^{q/p}\) is a convex function on \(\R_{\geq 0}\).
We have by Jensen's inequality (Theorem~7.3 in \citep{Cvetkovski2012}) that
\[\|x\|_q^q = \sum_{i=1}^n |x_i|^q = \sum_{i=1}^n |x_i^p|^{q/p} \geq n \left( \frac{1}{n}\sum_{i=1}^n |x_i|^p \right)^{q/p} = n^{1-q/p}  \|x\|_p^q,\]
so \(n^{-1/p}\cdot\|x\|_p\leq n^{-1/q} \cdot\|x\|_q\).
\end{proof}

\begin{definition}
Let \(\mathcal{H}\) be a Hilbert space.
A subset \(S\subseteq \mathcal{H}\) is called {\em orthogonal} if \(0\not\in S\) and \(\langle x,y\rangle =0\) for all distinct \(x,y\in S\).
The {\em orthogonal dimension} of \(\mathcal{H}\), written \(\Hdim \mathcal{H}\), is the cardinality of a maximal orthogonal subset of \(\mathcal{H}\).
\end{definition}

That the orthogonal dimension is well-defined, i.e.\ that maximal orthogonal subsets of a given Hilbert space have the same cardinality follows from Proposition~4.14 in \citep{hilbertspaces}. 

\begin{theorem}[Theorem~5.4 in \citep{hilbertspaces}] \label{thm:is_ell2}
Let \(\mathcal{H}\) be a Hilbert space and \(B\) a set.
Then the Hilbert spaces \(\mathcal{H}\) and \(\ell^2(B)\) are isomorphic if and only if the cardinality of \(B\) equals \(\Hdim\mathcal{H}\).
\end{theorem}

\section{Hilbert lattices}

\begin{definition}\label{def:hilbert_lattice}
A {\em Hilbert lattice} is an abelian group \(\Lambda\) together with a map \(q:\Lambda\to\R\), which we then call the \textit{square-norm} of \(\Lambda\), that satisfies:
\begin{itemize}[align=right]
\setlength{\itemindent}{8em}
\item[(\textit{Parallelogram law})] For all \(x,y\in\Lambda\) we have \(q(x+y)+q(x-y)=2q(x)+2q(y)\);
\item[(\textit{Positive packing radius})] There exists a \(r\in\R_{>0}\) such that \(q(x)\geq r\) for all non-zero \(x\in\Lambda\).
\end{itemize}
We write \(\textup{P}(\Lambda)=\inf\{ q(x)\,|\, x\in\Lambda\setminus\{0\} \}\).
\end{definition}

The following lemma gives an equivalent definition of a Hilbert lattice.

\begin{lemma}\label{lem:trivial_hilbert_eq}
A Hilbert lattice \(\Lambda\) with square-norm \(q\) is a discrete \(\Z\)-inner product space with inner product given by \((x,y) \mapsto \frac{1}{2}(q(x+y)-q(x)-q(y))\). Conversely, every discrete \(\Z\)-inner product space \(M\) is a Hilbert lattice with square norm given by \(x \mapsto \langle x ,x \rangle\). 
\end{lemma}
\begin{proof}
The first statement is Theorem~\ref{thm:parallelogram} with the observation that the positive packing radius implies non-degeneracy and discreteness.
The second statement is Lemma~\ref{lem:parallelogram_law} with the observation that discreteness implies a positive packing radius.
\end{proof}

\begin{example}
Consider for some \(n\in\Z_{\geq 0}\) the vector space \(\R^n\) with the standard inner product.
If \(\Lambda\subseteq \R^n\) is a discrete subgroup, then \(\Lambda\) is a Hilbert lattice when \(q\) is given by \(x\mapsto\|x\|^2\).
\end{example}

\begin{example}\label{ex:free_Hilbert_lattice}
Let \(B\) be a set. Then \(\Z^{(B)}=\{(x_b)_b \in \Z^B \,|\, x_b=0 \text{ for all but finitely many \(b\)}\}\) is a Hilbert lattice in \(\ell^2(B)\) when \(q\) is given by \(x\mapsto\|x\|^2\). In fact, any discrete subgroup of \(\ell^2(B)\) is a Hilbert lattice.
\end{example}

\begin{example}\label{ex:pre_general_lattice}
The infimum defining \(\textup{P}(\Lambda)\) of a Hilbert lattice \(\Lambda\) need not be attained. 
Certainly if \(\Lambda = 0\) we have that \(\textup{P}(\Lambda)=\infty\) is not attained.
For an example of a non-degenerate \(\Lambda\) consider the following.
For a set \(I\) and a map \(f:I\to \R_{\geq 0}\) we write \(\smash{\Lambda^f}\) for the group \(\Z^{(I)}\) together with the map \(q((x_i)_i)=\sum_{i\in I} f(i)^2 x_i^2\).
Note that \(\inf\{q(x)\,|\,x\in\smash{\Lambda^f}\setminus\{0\}\}=\inf\{f(i)^2\,|\, i\in I\}\), so \(\smash{\Lambda^f}\) is a Hilbert lattice if and only if \(\inf\{f(i)\,|\,i\in I\}>0\).
We now simply take \(f:\Z_{>0}\to\R_{\geq 0}\) given by \(n\mapsto 1 + 1/n\).
\end{example}

\begin{lemma}\label{lem:sublattice}
Let \(\Lambda\) be a Hilbert lattice with square-norm \(q\). 
Then any subgroup \(\Lambda'\subseteq\Lambda\) is a Hilbert lattice when equipped with the square-norm \(q|_{\Lambda'}\). \qed
\end{lemma}

\begin{theorem}\label{thm:lattice_eq}
Let \(F\) be the forgetful functor from the category of Hilbert spaces to the category of \(\Z\)-inner product spaces.
Then every \(\Z\)-inner product space \(L\) has an injective universal morphism \(\eta\) to \(F\). 
For every \(\Z\)-inner product space \(L\), Hilbert space \(\mathcal{H}\) and injective morphism \(f: L\to\mathcal{H}\) we have that \(f\) is universal if and only if \(\Q \cdot f(L)\) is dense in \(\mathcal{H}\), and \(L\) is a Hilbert lattice if and only if \(f(L)\) is discrete in \(\mathcal{H}\).
\end{theorem}

It follows from this theorem that the Hilbert lattices are, up to isomorphism, precisely the discrete subgroups of Hilbert spaces.
Hence Theorem~\ref{thm:lattice_eq} allows us to assume without loss of generality that a Hilbert lattice is a discrete subgroup of a Hilbert space.

\begin{proof}
The first and second statement are just a combination of Example~\ref{ex:universal_morphism}.ii and Theorem~\ref{thm:completion_inner_product}, while the third is trivial when taking the equivalent definition of Lemma~\ref{lem:trivial_hilbert_eq}.
\end{proof}

\begin{remark}\label{rem:fingen_lattice}
Let \(\Lambda\) be a Hilbert lattice in a Hilbert space \(\mathcal{H}\) and suppose that \(\Lambda\) is finitely generated.
Then \(\R\Lambda\) is a finite dimensional \(\R\)-inner product space and thus complete.
It follows that \(\Lambda\to \R\Lambda\) is a universal morphism because \(\Q\Lambda\) is dense in \(\R\Lambda\).
Since \(\R\Lambda\) is finite dimensional, \(\Lambda\) is a lattice in the classical sense: a discrete subgroup of a Euclidean vector space.
\end{remark}

\begin{lemma}\label{lem:injective_tensor_hilbert}
Let \(\Lambda\) be a Hilbert lattice in a Hilbert space \(\mathcal{H}\). Then the natural map \(\R\tensor_\Z\Lambda\to\mathcal{H}\) is injective. 
\end{lemma}
\begin{proof}
To show \(\R\tensor_\Z\Lambda\to\mathcal{H}\) is injective we may assume by Lemma~\ref{lem:sublattice} without loss of generality that \(\Lambda\) is finitely generated, as any element in the kernel is also in \(\R\tensor_\Z\Lambda'\) for some finitely generated sublattice \(\Lambda'\subseteq\Lambda\).
Write \(V=\R\Lambda\subseteq\mathcal{H}\).
We may choose an \(\R\)-basis for \(V\) in \(\Lambda\), and let \(\Lambda'\) be the group generated by this basis.
Then \(\R\tensor_\Z\Lambda'\to V\) is an isomorphism.
As \(\Lambda\) is discrete in \(V\), also \(\Lambda/\Lambda'\) is discrete in \(V/\Lambda'\).
Now \(V/\Lambda'\) is compact, so the quotient \(\Lambda/\Lambda'\) is finite. Then \(\Lambda \subseteq \tfrac{1}{n} \Lambda' \), where \(n=\#(\Lambda/\Lambda')\).
Now the natural map \(\R\tensor_\Z\Lambda\to\mathcal{H}\) is injective because it is the composition of the map \(\R\tensor_\Z\Lambda\to\R\tensor_\Z(\tfrac{1}{n}\Lambda')=\R\tensor_\Z\Lambda'\), which is injective since \(\R\) is flat over \(\Z\), and the map \(\R\tensor_\Z\Lambda'\to V\), which is injective by construction. 
\end{proof}

\begin{proposition}\label{prop:fingen_is_free}
Let \(\Lambda\) be a Hilbert lattice in a Hilbert space \(\mathcal{H}\) and suppose \(\Lambda\) is finitely generated as \(\Z\)-module. 
Then \(\Lambda\) has a \(\Z\)-basis and any \(\Z\)-basis is \(\R\)-linearly independent.
\end{proposition}
\begin{proof}
Since \(\Lambda\) is finitely generated and torsion free, it is clear that \(\Lambda\) is free.
By Lemma~\ref{lem:injective_tensor_hilbert}, any \(\Z\)-linearly independent subset of \(\Lambda\) is \(\R\)-linearly independent.
\end{proof}

\begin{proposition}\label{prop:projection_lattice}
Suppose \(\Lambda\) is a Hilbert lattice in a Hilbert space \(\mathcal{H}\) and let \(\Lambda'\subseteq\Lambda\) be a finitely generated subgroup.
Let \(\pi:\mathcal{H}\to\mathcal{H}\) be the orthogonal projection onto the orthogonal complement of \(\Lambda'\).
Then for each \(0\leq t < \tfrac{1}{4}\textup{P}(\Lambda)\) there are only finitely many \(z\in\pi\Lambda\) such that \(q(z)\leq t\), and \(\pi\Lambda\) is a Hilbert lattice.
\end{proposition}
\begin{proof}
Suppose that \(\pi\Lambda\) contains infinitely many points \(z\) with \(q(z)\leq t\), or equivalently there exists some infinite set \(S\subseteq\Lambda\) such that \(\pi|_S\) is injective and \(q(\pi(x))\leq t\) for all \(x\in S\).
Consider the map \(\tau:\mathcal{H} \to \R\Lambda'\), the complementary projection to \(\pi\).
As \((\R\Lambda')/\Lambda'\) is compact, there must exist distinct \(x,y\in S\) such that \(q(\tau(x)-\tau(y)+w)<\textup{P}(\Lambda)-4t\) for some \(w\in\Lambda'\).
Then 
\[0<q(x-y+w)=q(\pi(x-y))+q(\tau(x-y)+w) < 2\big(q(\pi(x))+q(\pi(y))\big)+\textup{P}(\Lambda)-4t \leq \textup{P}(\Lambda),\]
a contradiction. Hence there are only finitely many \(z\in\pi\Lambda\) such that \(q(z)\leq t\).
To verify that \(\pi\Lambda\) is a Hilbert lattice it suffices to show that it is discrete in \(\mathcal{H}\), which follows from the previous by taking any non-zero value for \(t\).
\end{proof}

\begin{lemma}\label{lem:basis_ext}
Let \(\Lambda\) be a Hilbert lattice which is finitely generated as \(\Z\)-module and let \(S\subseteq\Lambda\) be a set of vectors that forms a basis for \(\Lambda\cap(\R S)\). Then there exists a basis \(B\supseteq S\) of \(\Lambda\).
\end{lemma}
\begin{proof}
Let \(\pi\) be the projection onto the orthogonal complement of \(S\). Then \(\pi\Lambda\) is a Hilbert lattice by Proposition~\ref{prop:projection_lattice} with a basis \(B_\pi\) by Proposition~\ref{prop:fingen_is_free}. Now choose for every \(b_\pi\in B_\pi\) a lift \(b\in\Lambda\) and let \(T\) be the set of those elements.
It is easy to show that \(B=S\cup T\) is a basis of \(\Lambda\). 
\end{proof}

We say an abelian group is {\em almost free} if all its countable subgroups are free.

\begin{theorem}\label{thm:almost_free}
If \(\Lambda\) is a Hilbert lattice, then \(\Lambda\) is almost free as an abelian group.
\end{theorem}
\begin{proof}
By Lemma~\ref{lem:sublattice} suffices to show that if \(\Lambda\) is countable then \(\Lambda\) is free.
We may write \(\Lambda=\{x_1,x_2,\dotsc\}\) and let \(V_i=\sum_{j=1}^i \R x_j\) and \(\Lambda_i=V_i\cap\Lambda\).
We claim that there exist bases \(B_i\) for \(\Lambda_i\) such that \(B_i\subseteq B_j\) for all \(i\leq j\).
Indeed, take \(B_0=\emptyset\) and inductively for \(\Lambda_{i+1}\) note that \(B_i\) is a basis for \(\Lambda_i=\Lambda_{i+1}\cap V_i\), so that by Lemma~\ref{lem:basis_ext} there exists some basis \(B_{i+1}\) for \(\Lambda_{i+1}\) containing \(B_i\).
Then \(B=\bigcup_{i=0}^\infty B_i\) is a basis for \(\Lambda\), so \(\Lambda\) is free.
\end{proof}

\begin{question}
We have by Example~\ref{ex:free_Hilbert_lattice} and Theorem~\ref{thm:almost_free} two inclusions
\[ \{ \textup{free abelian groups} \} \subseteq \{\textup{underlying groups of Hilbert lattices} \} \subseteq\{\textup{almost free abelian groups}\}. \]
Is one of these inclusions an equality, and if so, which?
\end{question}

\begin{example}
There are abelian groups which are almost free but not free.
Let \(X\) be a countably infinite set and consider the Baer--Specker group \(B=\Z^X\).
Then by Theorem~21 in \citep{kaplansky}, we have that \(B\) is not free.
Since \(B\) is a torsion-free \(\Z\)-module, so is any countable subgroup, which is then free by Theorem~16 in \citep{kaplansky}, i.e.\ \(B\) is almost free.
\end{example}

\begin{definition}
For a Hilbert lattice \(\Lambda\) we define its {\em rank} as \(\rk \Lambda=\dim_\Q(\Lambda\tensor_\Z\Q)\).
We will say a Hilbert lattice \(\Lambda\) is of {\em full rank} in an ambient Hilbert space \(\mathcal{H}\) if \(\Q\Lambda\) is dense in \(\mathcal{H}\).
\end{definition}

For free Hilbert lattices \(\Lambda\) we have \(\Lambda\cong\Z^{(\rk\Lambda)}\) as abelian group.
By Theorem~\ref{thm:lattice_eq} every Hilbert lattice has a uniquely unique Hilbert space in which it is contained and of full rank.

\begin{lemma}\label{lem:aux_orthdim}
Let \(\mathcal{H}\) be a Hilbert space and let \(S,T\subseteq\mathcal{H}\) be subsets such that \(S\) is infinite, the \(\Q\)-vector space generated by \(S\) is dense in \(\mathcal{H}\) and \(\inf\{ \|x-y\| \,|\, x,y\in T,\,x\neq y\} >0\). Then \(\# S \geq \#T\).
\end{lemma}
\begin{proof}
Because \(S\) is infinite, the set \(S\) and the \(\Q\)-vector space \(V\) generated by \(S\) have the same cardinality.
Let \(\rho=\inf\{ \|x-y\| \,|\, x,y\in T,\,x\neq y\}\).
Since \(V\) is dense in \(\mathcal{H}\) we may for each \(x\in T\) choose \(f(x)\in V\) such that \(\|x-f(x)\|<\rho/2\).
If \(f(x)=f(y)\), then \(\|x-y\|=\|(x-f(x))-(y-f(y))\|\leq \|x-f(x)\|+\|y-f(y)\|<\rho\), so \(x=y\).
Hence \(f\) is injective and we have \(\# S = \# V \geq \# T\).
\end{proof}

The following theorem was independently proven by O. Berrevoets and independently by B. Kadets.

\begin{theorem}\label{thm:onno}
If \(\Lambda\) is a Hilbert lattice in a Hilbert space \(\mathcal{H}\), then \(\rk \Lambda \leq \Hdim \mathcal{H}\) with equality if \(\Lambda\) is of full rank in \(\mathcal{H}\). %\(\Q\Lambda\) is dense in \(\mathcal{H}\).
\end{theorem}
\begin{proof}
First suppose \(\rk\Lambda\) is finite.
It follows from Lemma~\ref{lem:injective_tensor_hilbert} that \(\rk \Lambda = \dim_\R (\R\tensor_\Z\Lambda) = \dim_\R(\R\Lambda) \leq \dim_\R \mathcal{H}\).
If \(\Lambda\) is of full rank, then \(\R\Lambda\) is dense in \(\mathcal{H}\), but \(\R\Lambda\) is complete as it is finite-dimensional, so \(\R\Lambda=\mathcal{H}\) and \(\rk \Lambda = \dim_\R\mathcal{H}\).
Lastly, it follows from Theorem~\ref{thm:is_ell2} that \(\dim_\R\mathcal{H}=\Hdim\mathcal{H}\) when \(\dim_\R\mathcal{H}\) is finite.

Now suppose \(\rk \Lambda\) is infinite and thus \(\# \Lambda = \rk \Lambda\).
By Theorem~\ref{thm:is_ell2} we may assume without loss of generality that \(\mathcal{H}=\ell^2(B)\) for some set \(B\) of cardinality \(\Hdim \mathcal{H}\), which must be infinite.
Observe that the \(\Q\)-vector space generated by \(B\) is dense in \(\mathcal{H}\).
We may apply Lemma~\ref{lem:aux_orthdim} by discreteness of \(\Lambda\) to obtain \(\Hdim \mathcal{H} = \#B \geq \# \Lambda = \rk \Lambda\), as was to be shown.
If \(\Q\Lambda\) is dense in \(\mathcal{H}\), then we may apply Lemma~\ref{lem:aux_orthdim} since \(\|b-c\|^2=\|b\|^2+\|c\|^2=2\) for all distinct \(b,c\in B\) to conclude that \(\rk \Lambda = \#\Lambda \geq \# B = \Hdim \mathcal{H}\), and thus we have equality.
\end{proof}

\section{Decompositions}

\begin{definition}
Let \(\Lambda\) be a Hilbert lattice.
A {\em decomposition} of an element \(z\in \Lambda\) is a pair \((x,y)\in \Lambda^2\) such that \(z=x+y\) and \(\langle x, y \rangle \geq 0\).
A decomposition \((x,y)\) of \(z\in\Lambda\) is {\em trivial} if \(x=0\) or \(y=0\).
We say \(z\in \Lambda\) is {\em indecomposable} if it has exactly two decompositions, i.e.\ \(z\neq 0\) and the only decompositions of \(z\) are trivial.
Write \(\text{dec}(z)\) for the set of decompositions of \(z\in\Lambda\) and \(\text{indec}(\Lambda)\) for the set of indecomposable elements of \(\Lambda\).
\end{definition}

Indecomposable elements are in the compute science literature often called Voronoi-relevant vectors, for example in \cite{compsci}.
This name is clearly inspired by Theorem~\ref{thm:minimal_indec_set}.

\begin{example}\label{ex:compute_indecomposables}
Let \(f:I\to\R_{\geq 0}\) be such that \(\smash{\Lambda^f}\) as in Example~\ref{ex:pre_general_lattice} is a Hilbert lattice.
We will compute the indecomposables of \(\Lambda^f\).
Let \(x=(x_i)_i\in \textup{indec}(\smash{\Lambda^f})\) and write \(e_i\) for the \(i\)-th standard basis vector. 
Note that \(x\) must be {\em primitive}, i.e.\ not be of the form \(n y\) for any \(y\in\smash{\Lambda^f}\) and \(n\in\Z_{>1}\), because otherwise \(\langle y,(n-1)y\rangle = (n-1)\langle y,y\rangle > 0\) shows \((y,(n-1)y)\) is a non-trivial decomposition.
If \(x_i\) and \(x_j\) are non-zero for distinct \(i,j \in I\), then \(q(x-e_ix_i) + q(e_ix_i) = q(x)\) and we have a non-trivial decomposition of \(x\).
Hence \(x=\pm e_i\) for some \(i\in I\).
Note that \(e_i\) is indeed indecomposable for all \(i\in I\): Any decomposition \((x,y)\in\textup{dec}(e_i)\) must have \(|x_i|+|y_i|=1\) and \(x_j=y_j=0\) for \(i\neq j\), so \(x=0\) or \(y=0\).
As \(z\mapsto -z\) is an isometry of \(\smash{\Lambda^f}\), we have that \(-e_i\) is indecomposable as well.
Hence \(\textup{indec}(\smash{\Lambda^f})=\{\pm e_i\,|\,i\in I\}\).
\end{example}

\begin{lemma}\label{lem:eq_dec_def}
Let \(\Lambda\) be a Hilbert lattice and let \(x,y,z\in\Lambda\). Then the following are equivalent:
\begin{enumerate}[label=\textup{(\roman*)}]
\item The pair \((x,y)\) is a decomposition of \(z\).
\item We have \(x+y=z\) and \(q(x)+q(y)\leq q(z)\).
\item We have \(x+y=z\) and \(q(x-z/2) \leq q(z/2)\).
\item We have \(x+y=z\) and \(q(z-2y)\leq q(z)\).
\end{enumerate}
\end{lemma}
For a visual aid to this lemma see Figure~\ref{fig:decomposition}.
\begin{proof}
(i \(\Leftrightarrow\) ii) By bilinearity we have
\[ q(z) = \langle x+y,x+y\rangle = q(x)+q(y) + 2\langle x, y\rangle. \]

(ii \(\Leftrightarrow\) iii) By the parallelogram law we have
\[ q(x)+q(y)=2q\Big(\frac{x+y}{2}\Big)+2q\Big(\frac{x-y}{2}\Big)=2q\Big(\frac{z}{2}\Big)+2q\Big(x-\frac{z}{2}\Big), \]
so \(q(x)+q(y)-q(z)=2\cdot[ q(x-z/2)-q(z/2) ]\). The claim then follows trivially.

(iii \(\Leftrightarrow\) iv) Note that \(z-2y = x-y = 2(x/2-y/2) = 2(x-z/2)\).
By the parallelogram law we have \(q(2w)=4q(w)\) for all \(w\in\Q\Lambda\), from which this equivalence trivially follows.
\end{proof}

By Lemma~\ref{lem:eq_dec_def}, finding decompositions of \(z\in\Lambda\) amounts to finding \(x\in\Lambda\) sufficiently close to \(z/2\).

\begin{lemma}\label{lem:small_indec}
Let \(z\in\Lambda\) such that \(0<q(z)\leq2\textup{P}(\Lambda)\). Suppose that the latter inequality is strict or \(\textup{P}(\Lambda)\) is not attained by any vector in \(\Lambda\). Then \(z\in\textup{indec}(\Lambda)\). 
\end{lemma}
\begin{proof}
If \((x,y)\in\textup{dec}(z)\) is non-trivial, then by Lemma~\ref{lem:eq_dec_def} we have \(2\textup{P}(\Lambda) \geq q(z) \geq q(x)+q(y) \geq \textup{P}(\Lambda)+\textup{P}(\Lambda)\) with either the first or last inequality strict, which is a contradiction. Since \(z\neq 0\) it follows that \(z\) is indecomposable.
\end{proof}

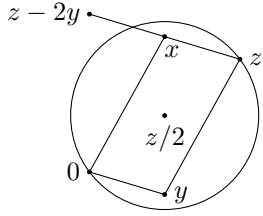
\begin{figure}
\centering
\begin{tikzpicture}[scale=1]
\draw[fill=black] (0,0) circle (0.025);
\node[left] (O) at (0,0) {$0$};

\draw[fill=black] (1,1.8) circle (0.025);
\node (X) at (1.1,1.6) {$x$};

\draw[fill=black] (2,1.5) circle (0.025);
\node[right] (Z) at (2,1.5) {$z$};

\draw[fill=black] (1,-.3) circle (0.025);
\node[right] (Y) at (1,-.3) {$y$};

\draw[fill=black] (1,.75) circle (0.025);
\node[below] (Z2) at (1,.75) {$z/2$};

\draw[fill=black] (0,2.1) circle (0.025);
\node[left] (Zm2Y) at (0,2.1) {$z-2y$};

\draw (1,.75) circle (1.25);

\draw (1,1.8) -- (0,0) -- (1,-.3) -- (2,1.5) -- (0,2.1);
\end{tikzpicture}
\caption{A decomposition \(z=x+y\)}
\label{fig:decomposition}
\end{figure}

\begin{proposition}\label{prop:indec_sum}
If \(\Lambda\) is a non-zero Hilbert lattice, then every \(z\in\Lambda\) can be written as a sum of at most \(q(z)/\textup{P}(\Lambda)\) indecomposables \(z_1,\dotsc,z_n\in\Lambda\) such that \(\sum_i q(z_i) \leq q(z)\).
\end{proposition}
\begin{proof}
By scaling \(q\) we may assume without loss of generality that \(\textup{P}(\Lambda)=1\).
We apply induction to \(\lfloor q(z)\rfloor\).
When this equals \(0\) we have \(q(z)<\textup{P}(\Lambda)\), so \(z=0\), which we can write as a sum of zero indecomposables.
Now suppose \(\lfloor q(z)\rfloor\geq 1\) and thus \(z\neq 0\).
If \(z\) is indecomposable then indeed it is the sum of \(1\leq \lfloor q(z) \rfloor\) indecomposable, so suppose there is a non-trivial \((x,y)\in\text{dec}(z)\). 
Then \(q(x)+q(y)\leq q(z)\) by Lemma~\ref{lem:eq_dec_def} and since \(y\neq 0\) also \(q(y)\geq \textup{P}(\Lambda) = 1\).
Hence \(\lfloor q(x)\rfloor \leq \lfloor q(z)-q(y)\rfloor < \lfloor q(z) \rfloor\), so by the induction hypothesis we may write \(x=\sum_i x_i\) with \(x_1,\dotsc,x_a\in\text{indec}(\Lambda)\) and \(a\leq q(x)\) such that \(\sum_i q(x_i)\leq q(x)\).
By symmetry we may similarly write \(y\) as a sum of at most \(q(y)\) indecomposables \(y_1,\dotsc,y_b\).
Hence we can write \(z=\sum_i x_i + \sum_i y_i\) as a sum of \(a+b\leq q(x)+q(y)\leq q(z)\) indecomposables such that \(\sum_i q(x_i) + \sum_i q(y_i) \leq q(x)+q(y) \leq q(z)\).
The proposition follows by induction.
\end{proof}

\begin{lemma}\label{lem:reverse_indec_sum}
Suppose \(\Lambda\) is a Hilbert lattice and \(z\in\Lambda\) is the sum of some non-zero \(z_1,\dotsc,z_n\in\Lambda\) and \(n\in\Z_{\geq 2}\).
If \(\sum_i q(z_i)\leq q(z)\), then \(z\) has a non-trivial decomposition.
\end{lemma}
\begin{proof}
We have
\[\sum_{i=1}^n \langle z_i,z-z_i\rangle = \sum_{i=1}^n \langle z_i,z\rangle - \sum_{i=1}^n \langle z_i,z_i\rangle = q(z)-\sum_{i=1}^n q(z_i) \geq 0,\]
so \(\langle z_i,z-z_i\rangle \geq 0\) for some \(i\). As neither \(z_i\) nor \(z-z_i\) are \(0\), we conclude that \((z_i,z-z_i)\) is a non-trivial decomposition of \(z\).
\end{proof}

\begin{proposition}\label{prop:hilbert_eq}
Let \(\Lambda\) be a Hilbert lattice. The group \(\{\pm 1\}\) acts on \(\textup{indec}(\Lambda)\) by multiplication, and the natural map \(\textup{indec}(\Lambda) / \{\pm 1\} \to \Lambda/2\Lambda\) is injective.
\end{proposition}
\begin{proof}
Note that \(\{\pm 1\}\) acts on \(\Lambda\) and thus also on \(\textup{indec}(\Lambda)\).
By (i \(\Leftrightarrow\) iv) of Lemma~\ref{lem:eq_dec_def} we have
\[ \text{dec}(z)=\{(\tfrac{z-a}{2},\tfrac{z+a}{2}) \,|\, a\in z+2\Lambda,\, q(a)\leq q(z) \}. \]
Let \(z\in \textup{indec}(\Lambda)\). 
Then \(\{(0,z),(z,0)\}=\text{dec}(z)\), so the only \(a\in z+2\Lambda\) such that \(q(a)\leq q(z)\) are \(a=z\) and \(a=-z\). Thus \(z\) is a \(q\)-minimal element of its coset in \(\Lambda/2\Lambda\) and this minimal element is unique up to sign. Consequently, the map \(\textup{indec}(\Lambda)/\{\pm 1\}\to\Lambda/2\Lambda\) is injective.
\end{proof}

\begin{corollary}\label{cor:finite_indec}
Let \(\Lambda\) be a Hilbert lattice. Then \(\textup{indec}(\Lambda)\) is finite if and only if \(\rk \Lambda<\infty\).
\end{corollary}
\begin{proof}
Recall that \(\textup{indec}(\Lambda)\) generates \(\Lambda\) by Proposition~\ref{prop:indec_sum}.
Hence if \(\textup{indec}(\Lambda)\) is finite, then \(\rk \Lambda <\infty\). 
If \(\rk \Lambda<\infty\), then Proposition~\ref{prop:hilbert_eq} implies \(\#\textup{indec}(\Lambda)\leq 2 \cdot \#(\Lambda/2\Lambda) = 2^{1+\rk \Lambda}<\infty\).
\end{proof}

The zero coset of \(\Lambda/2\Lambda\) is never in the image of the map of Proposition~\ref{prop:hilbert_eq}, as any non-zero element of the form \(2x\) with \(x\in\Lambda\) has a non-trivial decomposition \((x,x)\).
A non-zero coset \(C\) of \(\Lambda/2\Lambda\) can fail to be in the image for two reasons:
Either \(C\) has no minimal element or a minimal element exists but is not unique up to sign.
In the latter case, with \(z\in C\) minimal, there exists a \((x,y)\in\text{dec}(z)\) with \(x,y\neq 0\) and \(q(x)+q(y)=q(z)\) and thus \(\langle x,y\rangle =0\), i.e.\ \(z\) has an orthogonal decomposition. This is exhibited, for example, by the lattice \(\Z^2\subseteq\R^2\) with the standard inner product and \(z=(1,1)\), where \((-1,1)\in z+2\Z^2\) gives rise to the orthogonal decomposition \((1,0)+(0,1)=z\).
In the former case, \(\rk \Lambda\) has to be infinite: 
If \(\rk \Lambda \) is finite, then for any \(x \in \Lambda \) there are only finitely many \(y\in\Lambda\) with \(q(y)\leq q(x)\), so \(q\) assumes a minimum on any non-empty subset of \(\Lambda\), in particular \(C\).
An example is the following.

\begin{example}\label{ex:minimum_on_coset} \label{ex:general_lattice}
We will exhibit a Hilbert lattice \(\Lambda\) and a coset of \(2\Lambda\) on which \(q\) does not attain a minimum.
Let \(f:I\to\R_{\geq 0}\) be such that \(\smash{\Lambda^f}\) as in Example~\ref{ex:pre_general_lattice} is a Hilbert lattice.
We define the \(\smash{\Lambda_2^f}\) to be the sublattice
\[\Lambda_2^f = \ker( \Lambda^f \xrightarrow{\Sigma} (\Z/2\Z) ) = \Big\{ (x_i)_{i} \in \Z^{(I)} \,\Big|\, \sum_{i\in I} x_i \equiv 0 \mod 2 \Big\}.\]

Consider \(f:\Z_{\geq 0}\to \R_{\geq 0}\) strictly decreasing, write \(f(\infty)\) for its limit, assume \(f(\infty)>0\), and let \(\Lambda=\smash{\Lambda_2^f}\).
Let \(z=2e_k\in\Lambda\) for any \(k\in\Z_{\geq 0}\).
Then for all \(y=(y_i)_i\in\Lambda\) we have
\[ q(z-2y) = 4\Big( (1-y_k)^2 f(k)^2 + \sum_{i \neq k} y_i^2 f(i)^2 \Big) > 4 f(\infty)^2, \]
since either \(y_i^2\geq 1\) for some \(i \neq k\) or \(y_k\) is even.
However, we also have for \(y=e_k+e_i\) that \(q(z-2y)=q(2e_i)=4f(i)^2\to 4 f(\infty)^2\) as \(i\to\infty\).
Hence \(\{q(z-2y)\,|\,y\in\Lambda\}\) does not contain a minimum.
\end{example}

If we eliminate the zero coset in the proof of Corollary~\ref{cor:finite_indec} the upper bound on the number of indecomposables becomes \(2(2^{\rk(\Lambda)}-1)\) and this bound is tight. 
If one took the effort to define a sensible probability measure on the space of all Hilbert lattices of given finite rank, then this upper bound will in fact be an equality with probability \(1\).

\section{Hilbert lattice decompositions}

The main motivation for considering indecomposable elements is found in the study of decompositions of lattices.
In this section we generalize a result of Eichler \citep{eichler} on the existence of a universal decomposition in lattices to Hilbert lattices.

\begin{definition}\label{def:orth_sum}
Let \(\Lambda\) be a Hilbert lattice. 
We call a family \(\{\Lambda_i\}_{i\in I}\) of sub-Hilbert lattices \(\Lambda_i\subseteq\Lambda\) for some index set \(I\) a {\em decomposition} of \(\Lambda\), and if so write \(\bigoperp_{i\in I}\Lambda_i = \Lambda\), if \(\langle \Lambda_i,\Lambda_j\rangle = 0\) for all \(i\neq j\) and the natural map of abelian groups \(\bigoplus_{i\in I}\Lambda_i\to\Lambda\) is an isomorphism.
We make the set of decompositions of \(\Lambda\) into a category, where the morphisms from \(\{\Lambda_i\}_{i\in I}\) to \(\{\textup{M}_{j}\}_{j\in J}\) are the maps \(f:I\to J\) such that \(\textup{M}_j = \bigoperp_{i \in f^{-1} \{j\}} \Lambda_i\) for all \(j\in J\).
We say \(\Lambda\) is {\em indecomposable} if \(\Lambda \neq 0\) and for all \(\Lambda_1,\Lambda_2\subseteq\Lambda\) such that \(\Lambda_1\operp\Lambda_2=\Lambda\) we have \(\Lambda_1=0\) or \(\Lambda_2=0\).
\end{definition}

\begin{definition}
A {\em graph} is a pair \((V,E)\) where \(V\) is a (potentially infinite) set and \(E\) is a set of size-2 subsets of \(V\).
For a graph \((V,E)\) we call the elements of \(V\) its {\em vertices} and those of \(E\) its {\em edges}.
A {\em connected component} of \((V,E)\) is a non-empty set \(S\subseteq V\) such that there is no \(\{u,v\}\in E\) such that \(u\in S\) and \(v\not\in S\) and which is minimal with respect to inclusion given these properties.
We say a graph is {\em connected} if it has precisely one connected component.
\end{definition}

\begin{lemma}\label{lem:graph_facts}
Let \(G=(V,E)\) be a graph.
Then the connected components of \(G\) are pairwise disjoint, and if for \(S\subseteq V\) there exist no \(\{u,v\}\in E\) such that \(u\in S\) and \(v\not\in S\), then \(S\) is a union of connected components.
\end{lemma}
\begin{proof}
Let \(\mathcal{C}\) be the set of \(S\subseteq V\) such that there are no \(\{u,v\}\in E\) such that \(u\in S\) and \(v\not\in S\), so that the connected components of \(G\) become the minimal non-empty elements of \(\mathcal{C}\) with respect to inclusion.
Note that \(\mathcal{C}\) is closed under taking complements, arbitrary unions and arbitrary intersections, i.e.\ \(\mathcal{C}\) is a clopen topology on \(V\).
Suppose \(S,T\in\mathcal{C}\) are connected components that intersection non-trivially, then \(S\cap T\in\mathcal{C}\) is non-empty, so by minimality \(S=T\).
Hence the connected components are pairwise disjoint.

Now let \(S\in\mathcal{C}\) and for all \(s\in S\) let \(A_s=\{T\in\mathcal{C}\,|\,s\in T\}\) and \(C_s=\bigcap_{T\in A_s} T\), which is an element of \(A_s\). For all \(T\in\mathcal{C}\) we either have \(T\in A_s\) or \(V\setminus T \in A_s\), and thus either \(C_s\subseteq T\) or \(C_s\cap T = \emptyset\). It follows that no non-empty \(T\in\mathcal{C}\) is strictly contained in \(C_s\), i.e.\ \(C_s\) is a connected component of \(G\).
As \(S\in A_s\) we have \(s\in C_s\subseteq S\) and thus \(S = \bigcup_{s\in S} C_s\) is a union of connected components.
\end{proof}

The following is a generalization of a theorem due to Eichler \citep{eichler}, although the proof more closely resembles that of Theorem 6.4 in \citep{milnor}.

\begin{theorem}[Eichler] \label{thm:eichler}
Let \(\Lambda\) be a Hilbert lattice and \(V\subseteq\textup{indec}(\Lambda)\) such that \(V\) generates \(\Lambda\) as a group.
Let \(G\) be the graph with vertex set \(V\) and with an edge between \(x\) and \(y\) if and only if \(\langle x,y\rangle \neq 0\). 
Let \(U\) be the set of connected components of \(G\) and for \(u\in U\) let \(\textup{Y}_u\subseteq \Lambda\) be the subgroup generated by the elements in \(u\).
Then \(\{\textup{Y}_u\}_{u\in U}\) is a universal decomposition of \(\Lambda\).
\end{theorem}

As corollary to this theorem we have that \(\Lambda\) is indecomposable if and only if \(G\) is connected.
Note that \(V=\textup{indec}(\Lambda)\) generates \(\Lambda\) as a group by Proposition~\ref{prop:indec_sum}.
It then follows from the theorem that every Hilbert lattice has a universal decomposition.
This universal decomposition is unique and thus independent of choice of \(V\).

\begin{proof}
We have \(V\subseteq\bigcup_{u\in U}\textup{Y}_u\) by Lemma~\ref{lem:graph_facts}, so \(\sum_{u\in U} \textup{Y}_u=\Lambda\) by assumption on \(V\). 
For \(u,v\in U\) distinct we have \(\langle u,v\rangle =\{0\}\) by definition of \(G\), so \(\langle \textup{Y}_u, \textup{Y}_v\rangle = \{0\}\).
We conclude that \(\Lambda=\bigoperp_{u\in U}\textup{Y}_u\) is a decomposition.

To show it is universal, let \(\{\Lambda_i\}_{i\in I}\) be a family of sublattices of \(\Lambda\) such that \(\bigoperp_{i\in I}\Lambda_i=\Lambda\).
Let \(x\in\textup{indec}(\Lambda)\) and write \(x=\sum_{i\in I} \lambda_i\) with \(\lambda_i\in\Lambda_i\) for all \(i\in I\).
If \(j\in I\) is such that \(\lambda_j\neq 0\), then \(\langle \lambda_j,\sum_{i\neq j} \lambda_i \rangle = 0\) and thus \(\lambda_j=x\), because otherwise we obtain a non-trivial decomposition of \(x\).
Therefore every indecomposable of \(\Lambda\) is in precisely one of the \(\Lambda_i\). 
We conclude that the \(S_i=\Lambda_i\cap V\) for \(i\in I\) are pairwise disjoint and have \(V\) as their union.
Then by Lemma~\ref{lem:graph_facts} every connected component \(u\in U\) is contained in precisely one of the \(S_i\), say in \(S_{f(u)}\).
By definition of the map \(f:U\to I\) and the \(\textup{Y}_u\) we have \(\bigoperp_{u\in f^{-1}\{i\}} \textup{Y}_u \subseteq \Lambda_i\) for all \(i\), and since both the \(\textup{Y}_u\) and the \(\Lambda_i\) sum to \(\Lambda\) we must have equality for all \(i\).
Hence a map \(f\) as in the theorem exists and it follows trivially from the construction that it is unique.
We conclude that \(\{\textup{Y}_u\}_{u\in U}\) is a universal decomposition of \(\Lambda\).
\end{proof}

We will use this theorem in Section~\ref{sec:graded_rings} to generalize some theorems from \citep{Lenstra2018}.

\section{Voronoi cells} 

We will generalize the Voronoi cell as defined for classical lattices to Hilbert lattices and extend some known definitions and properties.
Some of these definitions relate to the ambient Hilbert space of the Hilbert lattice, which exists and is uniquely unique by Theorem~\ref{thm:lattice_eq} when we require the \(\Q\)-vector space generated by the lattice to lie dense in the Hilbert space. 
In this section we will write \(\mathcal{H}_\Lambda\) for this ambient Hilbert space of a Hilbert lattice \(\Lambda\).

\begin{definition}\label{def:radius}
Let \(\Lambda\) be a Hilbert lattice. The {\em packing radius} of \(\Lambda\) is 
\[\rho(\Lambda)=\inf\{ \tfrac{1}{2}\|x-y\|\,|\, x,y\in\Lambda,\,x\neq y\} = \tfrac{1}{2} \sqrt{\textup{P}(\Lambda)},\]
the {\em covering radius} of \(\Lambda\) is
\[r(\Lambda) = \inf\{ b\in \R_{>0} \,|\, (\forall z\in\mathcal{H}_\Lambda)\,(\exists\, x\in\Lambda)\, \|z-x\|\leq b \}\]
and the {\em Voronoi cell} of \(\Lambda\) in \(\mathcal{H}\) is the set 
\[\vor(\Lambda)=\{ z\in\mathcal{H}_\Lambda \,|\, (\forall x\in\Lambda\setminus\{0\})\ \|z\| < \|z-x\| \}.\]
\end{definition}

We call \(\rho(\Lambda)\) the packing radius because it is the radius of the largest sphere \(B\subseteq\mathcal{H}_\Lambda\) such that the spheres \(x+B\) for \(x\in\Lambda\) are pairwise disjoint.
Similarly \(r(\Lambda)\) is the radius of the smallest sphere \(B\subseteq\mathcal{H}_\Lambda\) for which \(\bigcup_{x\in \Lambda} (x+B)=\mathcal{H}_\Lambda\).
Note that \(r(\Lambda)=0\) only for \(\Lambda=0\) by discreteness.

\begin{example}
The covering radius of a Hilbert lattice need not be finite.
Take \(\Lambda^f\) as in Example~\ref{ex:pre_general_lattice} but with \(f:\Z_{\geq 0}\to\R_{\geq 0}\) diverging to infinity.
The lattice point closest to \(\frac{1}{3}e_i\) is \(0\) for all \(i\in \Z_{\geq 0}\), so it has distance \(\frac{1}{3}f(i)\) to the lattice.
Hence \(r(\smash{\Lambda^f})\geq \sup\{ \frac{1}{3} f(i)  \,|\, i\in \Z_{\geq 0} \} = \infty\).
\end{example}

\begin{example}
The Voronoi cell does not need to be an open set.
Consider the lattice \(\Lambda=\Lambda^f_2\) as in Example~\ref{ex:minimum_on_coset} with \(f:\Z_{\geq 0}\to\R_{>0}\) strictly decreasing.
Let \(i\in\Z_{\geq 0}\) and \(\textup{A}=(1+f(\infty)^2 f(i)^{-2})/2\) and write \(e_i\) for the \(i\)-th standard basis vector.
We claim that \(\alpha e_i \in \vor(\Lambda)\) for \(\alpha\in\R\) precisely when \(|\alpha|\leq \textup{A}\), which proves the Voronoi cell is not open.

Let \(x=\sum_{j} x_j e_j \in\Lambda\) such that \(x_i\neq 0\). 
Then \(|x_i|=1\) or \(q(x)/|x_i| \geq |x_i| f(i)^2 > f(i)^2 + f(i+1)^2\).
It follows that 
\[\inf\Big\{ \frac{q(x)}{2 |x_i| f(i)^2} \,\Big|\, x\in\Lambda,\, x_i\neq 0 \Big\} = \frac{f(i)^2 + f(\infty)^2}{2f(i)^2} = \textup{A} \]
and that the infimum is not attained.
By definition \(\alpha e_i\in \vor(\Lambda)\) if and only if \(g(x) := q(\alpha e_i-x)-q(\alpha e_i) > 0\) for all \(x\in\Lambda\setminus\{0\}\).
Note that \(g(x) = q(x) - 2 \alpha x_i f(i)^2\).

Suppose \(|\alpha|\leq \textup{A}\) and let \(x\in\Lambda \setminus\{0\}\).
If \(\alpha x_i \leq 0\), then \(g(x) \geq q(x) > 0\), so suppose \(\alpha x_i > 0\).
Then
\[ \frac{g(x)}{2|x_i| f(i)^2} = \frac{q(x)}{2|x_i|f(i)^2} - |\alpha| > \textup{A} - |\alpha| \geq 0, \]
so \(g(x)>0\).
We conclude that \(\alpha e_i\in \vor(\Lambda)\). 
Conversely, if \(|\alpha|>\textup{A}\), then 
\[\inf\Big\{ \frac{g(x)}{2 |x_i| f(i)^2} \,\Big|\, x\in\Lambda,\, x_i\neq 0 \Big\} = \textup{A}-|\alpha| < 0,\]
hence there exists some \(x\in \Lambda\) such that \(g(x)<0\) and thus \(\alpha e_i \not\in\vor(\Lambda)\).
\end{example}

\begin{definition}
Let \(\mathcal{H}\) be a Hilbert space and let \(S\subseteq\mathcal{H}\) be a subset.
We say \(S\) is {\em symmetric} if for all \(x\in S\) also \(-x\in S\).
We say \(S\) is {\em convex} if for all \(x,y\in S\) and \(t\in[0,1]\) also \((1-t)x+ty\in S\).
\end{definition}

\begin{lemma}\label{lem:convex_facts}
Let \(\mathcal{H}\) be a Hilbert space. For \(X\subseteq\mathcal{H}\) write \(\overline{X}\) for the topological closure of \(X\). Then
\begin{enumerate}[nosep,label=\textup{(\roman*)}]
\item The intersection \(\bigcap_i S_i\) of convex sets \((S_i)_{i\in I}\) in \(\mathcal{H}\) is convex;
\item The topological closure \(\overline{S}\) of a convex set \(S\) in \(\mathcal{H}\) is convex;
\item For all \(S\) in \(\mathcal{H}\) open convex, \(x\in S\), \(y\in \overline{S}\) and \(t\in[0,1)\) we have \((1-t)x+ty\in S\).
\item For convex open sets \((S_i)_{i\in I}\) in \(\mathcal{H}\) with non-empty intersection we have \(\overline{\bigcap_i S_i}=\bigcap_i \overline{S_i}\).
\end{enumerate}
\end{lemma}
\begin{proof}
(i) Trivial. (ii) Let \(x,y\in \overline{S}\) and let \((x_n)_n\) and \((y_n)_n\) be sequences in \(S\) with limit \(x\) respectively \(y\).
For all \(t\in[0,1]\) we have \((1-t)x_n+t y_n \in S\), and since addition and scalar multiplication are continuous also \((1-t)x+ty=\lim_{n\to\infty} [ (1-t)x_n+ty_n] \in \overline{S}\). 
(iii) By translating \(S\) we may assume without loss of generality that \((1-t)x+ty=0\). Since \(x\in S\) and \(S\) is open there exists some \(r_x>0\) such that the open ball \(B_x\) of radius \(r_x\) around \(x\) is contained in \(S\). 
For \(r_y>0\) sufficiently small (in fact \(r_y=r_x \cdot (1-t)/t\) suffices, see Figure~\ref{fig:triangles}) it holds that for any \(z\) in the open ball \(B_y\) of radius \(r_y\) around \(y\) the line through \(0\) and \(z\) intersects \(B_x\). 
Taking \(z\in B_y\cap S\), which exists because \(y\) is in the closure of \(S\), there exists some \(w\in B_x\subseteq S\) such that \(0\) lies on the line segment between \(w\) and \(z\). 
By convexity \(0\in S\) follows, as was to be shown. 
(iv) Since \(\bigcap_i \overline{S_i}\) is closed and contains \(\bigcap_i S_i\), clearly \(\overline{\bigcap_i S_i}\subseteq\bigcap_i \overline{S_i}\). 
By (i) the set \(\bigcap_i S_i\) is convex and by assumption it contains some \(x\). 
For \(t\in[0,1)\) and \(y\in \bigcap_i \overline{S_i}\) we have \(z_t=(1-t)x+ty\in S_i\) by (iii). 
Thus \(z_t\in \bigcap_i S_i\) for all \(t\in[0,1)\), so \(y=\lim_{t\to 1} z_t\in \overline{\bigcap_i S_i}\), proving the reverse inclusion.
\end{proof}

\begin{figure}
\def\tikza{-11}
\def\tikzll{3}
\def\tikzlr{4}
\begin{tikzpicture}
\draw (0,0) -- ({\tikzll*cos(\tikza+180)},{\tikzll*sin(\tikza+180)}) -- ({\tikzll/cos(\tikza+180)},0) -- (0,0);
\draw (0,0) -- ({\tikzlr*cos(\tikza)},{\tikzlr*sin(\tikza)}) -- ({\tikzlr/cos(\tikza)},0) -- (0,0);
\draw ({\tikzll/cos(\tikza+180)},0) circle ({\tikzll*tan(\tikza)});
\draw ({\tikzlr/cos(\tikza)},0) circle ({\tikzlr*tan(\tikza)});
\draw[fill=black] (0,0) circle (0.025);
\node[above] at (0,0) {\small $0$};
\draw[fill=black] ({\tikzll/cos(\tikza+180)},0) circle (0.025);
\node[below left] at ({\tikzll/cos(\tikza+180)},0) {\small $x$};
\draw[fill=black] ({\tikzlr/cos(\tikza)},0) circle (0.025);
\node[above right] at ({\tikzlr/cos(\tikza)},0) {\small $y$};
\node[above] at ({\tikzlr/cos(\tikza)/2},0) {\small $(1-t)\|x-y\|\phantom{abc}$};
\node[below] at ({\tikzll/cos(\tikza+180)/2},0) {\small $\phantom{ab}t\|x-y\|$};
\end{tikzpicture}
\caption{\label{fig:triangles}Computation of \(r_y\) from \(r_x\) via similar triangles.}
\end{figure}
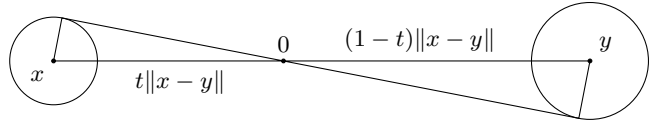

\begin{lemma}\label{lem:eq_def_voronoi}
Let \(\mathcal{H}\) be a Hilbert space. 
Then for \(x,y\in\mathcal{H}\) we have \(\|y\|\leq\|y-x\|\) if and only if \(2\langle x,y\rangle \leq \langle x,x\rangle\), and similarly with \(\leq\) replaced by \(<\).
\end{lemma}
\begin{proof}
We have \(\|y-x\|^2-\|y\|^2=\langle x,x\rangle - 2\langle x,y\rangle\), from which the lemma trivially follows.
\end{proof}

\begin{proposition}\label{prop:vor_closure}
For all Hilbert lattices \(\Lambda\) the set \(\vor(\Lambda)\) is symmetric, convex and has topological closure \(\vorc(\Lambda) \vcentcolon= \{z\in\mathcal{H}_\Lambda\,|\, (\forall x\in\Lambda)\ \|z\|\leq \|z-x\|\}\).
\end{proposition}
\begin{proof}
It follows readily from the definition that \(\vor(\Lambda)\) is symmetric.
Now for \(x\in\Lambda\) consider \(H_x=\{z\in\mathcal{H}_\Lambda\,|\,2\langle x,z\rangle < \langle x,x\rangle\}\).
It is easy to show for all \(x\in\Lambda\) that \(H_x\) is convex: For \(a,b\in H_x\) and \(t\in[0,1]\) we have 
\[2\langle (1-t)a+tb,x\rangle = (1-t)\cdot2\langle a,x\rangle+t\cdot 2\langle b ,x\rangle < (1-t)\langle x,x\rangle + t\langle x,x\rangle = \langle x,x\rangle, \]
so \((1-t)a+tb\in H_x\). 
As \(\vor(\Lambda)\) is the intersection of all \(H_x\) with \(x\in\Lambda\) non-zero by Lemma~\ref{lem:eq_def_voronoi}, it follows from Lemma~\ref{lem:convex_facts}.i that \(\vor(\Lambda)\) is convex.
The \(H_x\) are all open, and for \(x\) non-zero we have \(0\in H_x\).
Hence the topological closure of \(\vor(\Lambda)\) equals \(\{z\in\mathcal{H}_\Lambda\,|\, (\forall x\in\Lambda\setminus\{0\})\ \|z\|\leq \|z-x\|\}\) by Lemma~\ref{lem:convex_facts}.iv, from which the proposition follows.
\end{proof}

\begin{example}
We do not have in general that \(\mathcal{H}_\Lambda=\Lambda+\vorc(\Lambda)\) for all Hilbert lattices \(\Lambda\), as in the finite-dimensional case.
Note that \(z\in\mathcal{H}_\Lambda\) is in \(\Lambda+\vorc(\Lambda)\) if and only if the infimum \(\inf\{\|z-x\|\,|\,x\in\Lambda\}\) is attained for some \(x\in\Lambda\).
Consider Example~\ref{ex:minimum_on_coset}, where we exhibit a lattice \(\Lambda\) and a coset \(z+2\Lambda\) of \(\Lambda/2\Lambda\) where \(\inf\{q(z+2x)\,|\, x\in\Lambda \}\) is not attained.
Equivalently, \(\inf\{\|\frac{1}{2}z-x\|\,|\,x\in\Lambda\}\) is not attained, so \(\frac{1}{2}z\not\in\Lambda+\vorc(\Lambda)\).
\end{example}

\begin{proposition}\label{prop:almost_fundamental}
Let \(\Lambda\) be a Hilbert lattice. Then for all \(\epsilon > 0\) we have \(\Lambda + (1+\epsilon)\vor(\Lambda) = \mathcal{H}_\Lambda\).
\end{proposition}
\begin{proof}
Without loss of generality we assume \(\Lambda \neq 0\) and \(\textup{P}(\Lambda)=2\).
Let \(z\in\mathcal{H}_\Lambda\).
Choose \(y\in\Lambda\) such that \(q(z-y)\leq \epsilon + \inf\{q(z-w)\,|\, w\in\Lambda\}\).
Suppose \(x\in\Lambda\setminus\{0\}\).
Then 
\begin{align*}
(1+\epsilon)\langle x,x \rangle -2 \langle x,z-y\rangle = \epsilon q(x) +( q(z-y-x)-q(z-y) ) \geq \epsilon q(x) - \epsilon \geq \epsilon (\textup{P}(\Lambda)-1)=\epsilon>0.
\end{align*}
It follows that \(2\langle x,(z-y)/(1+\epsilon)\rangle < \langle x,x\rangle\) for all \(x\in\Lambda\setminus\{0\}\), so \((z-y)/(1+\epsilon)\in\vor(\Lambda)\) by Lemma~\ref{lem:eq_def_voronoi}.
Thus \(z\in\Lambda+(1+\epsilon)\vor(\Lambda)\), as was to be shown.
\end{proof}

\begin{proposition}\label{prop:vor_incl}
For all Hilbert lattices \(\Lambda\) the set \(\vor(\Lambda)\) contains the open sphere of radius \(\rho(\Lambda)\) around \(0\in\Lambda\) and \(\vorc(\Lambda)\) is contained in the closed sphere of radius \(r(\Lambda)\) around \(0\).
\end{proposition}
\begin{proof}
Let \(z\in\mathcal{H}_\Lambda\) be such that \(\|z\|<\rho(\Lambda)\) and let \(x\in\Lambda\setminus\{0\}\).
By Cauchy--Schwarz we have \(\langle x,z\rangle \leq \|x\|\cdot\|z\| < \|x\| \cdot \frac{1}{2}\|x\| = \frac{1}{2} \langle x,x\rangle\), so \(z\in\vor(\Lambda)\) by Lemma~\ref{lem:eq_def_voronoi}.

Let \(z\in\vorc(\Lambda)\). 
For each \(r>r(\Lambda)\) there exists \(x\in \Lambda\) such that \(\|z-x\|\leq r\) by definition of \(r(\Lambda)\).
Then by Proposition~\ref{prop:vor_closure} we have \(\|z\| \leq \|z-x\|\leq r\). Taking the limit of \(r\) down to \(r(\Lambda)\) proves the second inclusion.
\end{proof}

\begin{theorem}\label{thm:minimal_indec_set}
Let \(\Lambda\) be a Hilbert lattice.
Write \(C\) for the collection of subsets \(S\subseteq\Lambda\) for which \(\vor(\Lambda) = \{ z\in\mathcal{H}_\Lambda \,|\, (\forall x\in S)\ \|z\| < \|z-x\| \}\).
Then \(\textup{indec}(\Lambda)\in C\), and it is the unique minimum with respect to inclusion.
\end{theorem}
\begin{proof}
For \(S\subseteq\Lambda\) write \(V(S)=\{z\in\mathcal{H}_\Lambda\,|\, (\forall x\in S)\ \|z\| < \|z-x\|\}\). 

First suppose \(V(S)=\vor(\Lambda)\) for some \(S\subseteq\Lambda\). 
Let \(z\in\text{indec}(\Lambda)\) and note that \(\frac{1}{2}z\not\in\vor(\Lambda)\) since \(\|\frac{1}{2}z\|\geq \|\frac{1}{2}z-z\|\).
As \(V(S)=\vor(\Lambda)\) there must be some \(x\in S\) such that \(\|\frac{1}{2}z\|\geq\|\frac{1}{2}z-x\|\).
Hence \((x,z-x)\) is a decomposition of \(z\) by Lemma~\ref{lem:eq_dec_def}, so \(0=x\) or \(z=x\) since \(z\) is indecomposable.
If \(0=x\in S\), then \(V(S)=\emptyset\neq\vor(\Lambda)\), hence \(z=x\in S\).
We conclude that \(\text{indec}(\Lambda)\subseteq S\).

It remains to show that \(V(\text{indec}(\Lambda))=\vor(\Lambda)\).
We clearly have that \(\vor(\Lambda)\subseteq V(\text{indec}(\Lambda))\).
Suppose \(z\in V(\text{indec}(\Lambda))\) and let \(x\in\Lambda\setminus\{0\}\).
By Proposition~\ref{prop:indec_sum} we may write \(x=\sum_{i=1}^n x_i\) for some \(n\in\Z_{\geq 1}\) and \(x_i\in\textup{indec}(\Lambda)\) such that \(\sum_{i=1}^n \langle x_i,x_i\rangle\leq \langle x,x\rangle\). Then by Lemma~\ref{lem:eq_def_voronoi} we have
\[ 2\langle x,z\rangle = \sum_{i=1}^n 2\langle x_i,z\rangle < \sum_{i=1}^n \langle x_i,x_i\rangle \leq \langle x,x\rangle  \]
and thus \(z\in\vor(\Lambda)\). We conclude that \(\vor(\Lambda)=V(\text{indec}(\Lambda))\).
\end{proof}

\begin{corollary}\label{cor:vorc}
Let \(\Lambda\) be a Hilbert lattice. Then \(\vorc(\Lambda) = \{z\in\mathcal{H}_\Lambda\,|\,(\forall x\in\textup{indec}(\Lambda))\ \|z\|\leq \|z-x\|\}\).
\end{corollary}
\begin{proof}
By Lemma~\ref{lem:convex_facts}.iv we have for \(S\subseteq\Lambda\) not containing \(0\) that \(\overline{V}(S)=\{z\in\mathcal{H}_\Lambda\,|\,(\forall x\in S)\ \|z\|\leq \|z-x\|\}\) is the topological closure of \(V(S)\) as defined in the proof of Theorem~\ref{thm:minimal_indec_set}.
The corollary then follows from Proposition~\ref{prop:vor_closure} and Theorem~\ref{thm:minimal_indec_set}.
\end{proof}

\begin{example}
For a Hilbert lattice \(\Lambda\) the set \(\text{indec}(\Lambda)\) can fail to be the minimum among all sets \(S\subseteq\Lambda\) such that \(\overline{V}(S)=\{z\in\mathcal{H}_\Lambda\,|\,(\forall x\in S)\ \|z\|\leq \|z-x\|\}\) equals \(\vorc(\Lambda)\). 
We will give a counterexample. 

Let \(I=\Z_{\geq 0}\cup\{\infty\}\) and let \(f:I\to\R_{\geq 0}\) such that \(f|_{\Z_{\geq 0}}\) is strictly decreasing with limit \(f(\infty)>0\).
Consider the lattice \(\Lambda=\smash{\Lambda_2^f}\) as in Example~\ref{ex:general_lattice}.
Note that \(\textup{P}(\Lambda)=2f(\infty)^2\) and that \(\textup{P}(\Lambda)\) is not attained by any vector.
Hence \(2e_\infty\in\Lambda\) is indecomposable by Lemma~\ref{lem:small_indec}.
Now let \(S=\textup{indec}(\Lambda)\setminus\{\pm 2e_\infty\}\). 
We claim that \(\overline{V}(S)=\overline{V}(\text{indec}(\Lambda))\), the latter being equal to \(\vorc(\Lambda)\) by Corollary~\ref{cor:vorc}.
It remains to show for all \(z=(z_i)_i\in\overline{V}(S)\) that \(\|z\|\leq\|z-2e_\infty\|\) by symmetry.

For all \(i\) let \(s_i\in\{\pm1\}\) such that \(s_i z_i = |z_i|\).
Since \(f\) is strictly decreasing there exists some \(N\in\Z_{\geq 0}\) such that \(f(n)<\sqrt{3}f(\infty)\) for all integers \(n\geq N\).
For all integers \(n\geq N\) we have \(s_ne_n + e_\infty\in S\) by Lemma~\ref{lem:small_indec} as \(q(s_ne_n+e_\infty)=f(n)^2+f(\infty)^2< 4 f(\infty)^2 = 2\textup{P}(\Lambda)\).
Then
\[0 \leq \|z-(s_ne_n+e_\infty)\|^2-\|z\|^2=(1-2|z_n|)f(n)^2+(1-2z_\infty)f(\infty)^2.\]
As \(z\in\mathcal{H}_\Lambda\) we must have \(\lim_{n\to\infty}|z_n|=0\), so taking the limit over the above inequality we get \(0\leq 2(1-z_\infty)f(\infty)^2\) and thus \(z_\infty\leq 1\). 
But then \(\|z-2e_\infty\|^2-\|z\|^2=4f(\infty)^2(1-z_\infty)\geq 0\) and we are done.
\end{example}

\label{sec:end-hilbert}
\section{The lattice of algebraic integers}\label{sec:algebraic_integers}\label{sec:begin-integers}

We will write \(\overline{\Q}\) for an algebraic closure of \(\Q\). 
An {\em algebraic integer} is an element \(\alpha\in\overline{\Q}\) for which there exists a monic \(f\in\Z[X]\) such that \(f(\alpha)=0\).
The algebraic integers form a subring of \(\overline{\Q}\), which we denote \(\overline{\Z}\).
In this section we will prove that \(\overline{\Z}\) together with a natural choice of square-norm is a Hilbert lattice.

\begin{definition}\label{def:fundamental_set}
For a ring \(K\) we define the {\em fundamental set} \,to be the set \(\X(K)\) of ring homomorphisms from \(K\) to \(\C\). 
For a ring \(L\) with subring \(K\) and \(\sigma\in\X(K)\) we define \(\X_\sigma(L)=\{\rho\in\X(L)\,|\,\rho|_K = \sigma\}\).
\end{definition}

\begin{lemma}\label{lem:extension_dont_care}
Let \(\alpha\in\overline{\Q}\) and \(\Q(\alpha)\subseteq L\subseteq\overline{\Q}\) subfields with \([L:\Q]<\infty\).
Then the quantities
\[ \prod_{\sigma\in \X(L)} |\sigma(\alpha)|^{1/[L:\Q]} \quad\text{and}\quad \frac{1}{[L:\Q]} \sum_{\sigma\in \X(L)} |\sigma(\alpha)|^2 \]
are in \(\R_{\geq 0}\), equal to zero if and only if \(\alpha=0\), and do not depend on the choice of \(L\). \qed
\end{lemma}

\begin{definition}\label{def:Nq}
We define the maps \(N,q:\overline{\Q}\to\R_{\geq 0}\) by
\[ N(\alpha) = \prod_{\sigma\in \X(\Q(\alpha))} |\sigma(\alpha)|^{1/[\Q(\alpha):\Q]} \quad\text{and}\quad q(\alpha) = \frac{1}{[\Q(\alpha):\Q]} \sum_{\sigma\in \X(\Q(\alpha))} |\sigma(\alpha)|^2. \]
\end{definition}

\begin{lemma}\label{lem:parallogram_law}
For \(\alpha,\beta\in\overline{\Q}\) we have \(q(\alpha+\beta)+q(\alpha-\beta)=2q(\alpha)+2q(\beta)\).
\end{lemma}
\begin{proof}
By Lemma~\ref{lem:extension_dont_care} the restriction of \(q\) to \(L=\Q(\alpha,\beta)\) is given by \(q(\gamma)=\frac{1}{[L:\Q]}\sum_{\sigma\in \X(L)} |\sigma(\gamma)|^2\).
The norm \(|\cdot|\) on \(\C\) satisfies the parallelogram law and we may apply this term-wise to the sum defining \(q\) to obtain the lemma.
\end{proof}

\begin{lemma}[AM-GM inequality, Theorem~5.1 in \citep{Cvetkovski2012}]\label{lem:amgm}
Let \(n\in\Z_{\geq 1}\) and \(x_1,\dotsc,x_n\in\R_{\geq 0}\). Then
\[ \sqrt[n]{x_1\dotsm x_n} \leq \frac{x_1+\dotsm+x_n}{n}, \]
with equality if and only if \(x_1=x_2=\dotsm=x_n\). \qed
\end{lemma}

\begin{definition}\label{def:uniform}
An element \(\delta\in\overline{\Q}\) is called {\em uniform} if \(|\sigma(\delta)|=|\tau(\delta)|\) for all \(\sigma,\tau\in \X(\overline{\Q})\).
\end{definition}

\begin{lemma}\label{lem:Nq_ineq}
For all \(\alpha\in\overline{\Q}\) we have \(N(\alpha)^2\leq q(\alpha)\) with equality if and only if \(\alpha\) is uniform.
\end{lemma}
\begin{proof}
This lemma follows from a straightforward application of Lemma~\ref{lem:amgm}:
\[ q(\alpha) = \frac{1}{[\Q(\alpha):\Q]} \sum_{\sigma\in \X(\Q(\alpha))} |\sigma(\alpha)|^2 \geq \Big( \prod_{\sigma\in \X(\Q(\alpha))} |\sigma(\alpha)|^2 \Big)^{1/[\Q(\alpha):\Q]} = N(\alpha)^2, \]
with equality if and only if \(|\sigma(\alpha)|^2=|\rho(\alpha)|^2\) for all \(\sigma,\rho\in\X(\Q(\alpha))\).
\end{proof}

\begin{proposition}\label{prop:root_of_unity}
If \(\alpha\in\overline{\Z}\), then \(q(\alpha)\leq 1\) if and only if \(\alpha=0\) or \(\alpha\) is a root of unity.
If \(\alpha\) is a root of unity, then \(q(\alpha)=1\).
\end{proposition}
\begin{proof}
Let \(\alpha\in\overline{\Z}\).
The `if' part of the implication follows directly from the definition.
For the `only if' part, suppose \(\alpha\) is non-zero. 
If \(q(\alpha)\leq 1\), then \(N(\alpha)^2\leq 1\) by Lemma~\ref{lem:Nq_ineq}.
Then \(N(\alpha)^{[\Q(\alpha):\Q]}=|N_{\Q(\alpha)/\Q}(\alpha)|\in\Z_{\geq 1}\), so \(N(\alpha)^2=q(\alpha)=1\).
By Lemma~\ref{lem:Nq_ineq} we have \(|\sigma(\alpha)|=1\) for all \(\sigma\in \X(\Q(\alpha))\), so \(\alpha\) is a root of unity by Kronecker's theorem \citep{kronecker}.
\end{proof}

Recall for a Hilbert lattice \(\Lambda\) with square-norm \(q\) the definition \(\textup{P}(\Lambda)=\{q(x) \,|\, x\in\Lambda\setminus\{0\}\}\) from Definition~\ref{def:hilbert_lattice}.

\begin{theorem}\label{thm:Zbar_is_lattice}
The group \(\overline{\Z}\) with square-norm \(q\) is a Hilbert lattice and \(\textup{P}(\overline{\Z})=1\).
\end{theorem}
\begin{proof}
By Lemma~\ref{lem:parallogram_law} and Proposition~\ref{prop:root_of_unity} respectively the group \(\overline\Z\) together with \(q\) satisfies the parallelogram law and has \(\inf\{q(\alpha)\,|\,\alpha\in\overline{\Z}\setminus\{0\}\}=1\).
\end{proof}

\begin{remark}
Since \(\overline{\Z}\) is a Hilbert lattice, it is a discrete subgroup of a Hilbert space by Theorem~\ref{thm:lattice_eq}, which has an inner product.
For \(\alpha,\beta\in\overline{\Q}\) and \(\Q(\alpha,\beta)\subseteq L \subseteq\overline{\Q}\) with \([L:\Q]<\infty\) it is given by
\[ \langle \alpha,\beta \rangle = \frac{1}{[L:\Q]} \sum_{\sigma\in \X(L)} \sigma(\alpha)\overline{\sigma(\beta)}. \]
\end{remark}

\begin{lemma}\label{lem:uniform}
Suppose \(\alpha,\delta\in\overline{\Q}\) and \(\delta\) is uniform. 
Then \(q(\alpha\delta)=q(\alpha)q(\delta)\). 
If also \(\alpha,\delta\in\overline{\Z}\) and \(\alpha\delta\) is indecomposable, then \(\alpha\) is indecomposable.
\end{lemma}
\begin{proof}
Let \(L\supseteq\Q(\alpha,\delta)\). 
Then for all \(\sigma\in \X(L)\) we have \(q(\delta)=|\sigma(\delta)|^2\). 
Moreover,
\[ q(\alpha\delta) = \frac{1}{[L:\Q]}\sum_{\sigma\in\X(L)} |\sigma(\alpha\delta)|^2 = \frac{1}{[L:\Q]}\sum_{\sigma\in\X(L)} |\sigma(\alpha)|^2 \cdot q(\delta) = q(\alpha)q(\delta).  \]
Now suppose \(\alpha,\delta\in\overline{\Z}\) and let \((\beta,\gamma)\in\textup{dec}(\alpha)\).
Then \(\alpha\delta=\beta\delta+\gamma\delta\) and 
\[q(\alpha\delta)=q(\alpha)q(\delta)\geq (q(\beta)+q(\gamma))q(\delta)=q(\beta\delta)+q(\gamma\delta),\] 
so \((\beta\delta,\gamma\delta)\in\textup{dec}(\alpha\delta)\). 
If \(\alpha\delta\) is indecomposable, then \(\delta\neq 0\) and \(0\in\{\beta\delta,\gamma\delta\}\), so \(0\in\{\beta,\gamma\}\) and \((\beta,\gamma)\) must be a trivial decomposition.
Hence \(\alpha\) is indecomposable.
\end{proof}

\begin{definition}\label{def:isometries}\label{def:roots_of_unity}
Write \(\mu_\infty\) for the group of roots of unity in \(\overline{\Z}\) and \(\textup{Gal}(\overline{\Q})\) for the group of ring automorphisms of \(\overline{\Q}\).
Note that \(\textup{Gal}(\overline{\Q})\) naturally acts on \(\mu_\infty\), and write \(\mu_\infty \rtimes \textup{Gal}(\overline{\Q})\) for their semi-direct product with respect to this action.
\end{definition}

\begin{lemma}\label{lem:isometries} 
The group \(\mu_\infty \rtimes \textup{Gal}(\overline{\Q})\) acts faithfully on the Hilbert lattice \(\overline{\Z}\), where \(\mu_\infty\) acts by multiplication and \(\textup{Gal}(\overline{\Q})\) by application.
\end{lemma}
\begin{proof}
Let \(\alpha\in\overline{\Z}\), \(\zeta\in\mu_\infty\) and \(\rho\in \textup{Gal}(\overline{\Q})\).
Let \(K\) be the normal closure of \(\Q(\zeta,\alpha)\) and \(n=[K:\Q]\).

First we show that the individual group actions on \(\overline{\Z}\) are well-defined. % and faithful. 
Clearly \(\zeta \alpha \in\overline\Z\) and note that \(\zeta\) is uniform with \(q(\zeta)=1\).
Hence multiplication by \(\zeta\) map is an isometry, i.e. preserves length, by Lemma~\ref{lem:uniform}.
Recall that automorphisms preserve integrality and thus \(\rho(\alpha)\in\overline{\Z}\). 
Since \(K\) is normal over \(\Q\) we have \(\rho K = K\) and thus \(\X(K)\circ \rho = \X(K)\).
Hence applying \(\rho\) to \(\alpha\) simply results in a reordering of the terms in the sum defining \(q\) with respect to \(K\), and thus \(\rho\) is an isometry.

Note that for \((\chi,\sigma),(\xi,\tau)\in \mu_\infty \rtimes \textup{Gal}(\overline{\Q})\) we have 
\[(\chi,\sigma)\big((\xi,\tau)\,\alpha\big) = \chi\cdot\sigma(\xi\cdot\tau(\alpha)) = (\chi\sigma(\xi))((\sigma\tau)(\alpha)) = \big((\chi,\sigma)\cdot(\xi,\tau)\big)\,\alpha,\]
so the semi-direct product acts on \(\overline{\Z}\) as well.
Finally, suppose \((\zeta, \rho)\) acts as the identity.
Note that \(\textup{Gal}(\overline{\Q})\) fixes \(1\), so letting \((\zeta,\rho)\) act on \(1\) shows that \(\zeta=1\), and thus \(\rho=\textup{id}\).
Hence the action is faithful.
\end{proof}

\begin{question}
Is \(\mu_\infty \rtimes \textup{Gal}(\overline{\Q})\) the entire isometry group of \(\overline{\Z}\)?
\end{question}

\begin{proposition}\label{prop:holder_q}
Let \(\alpha\in\overline{\Z}\), \(r\in\Z_{\geq 0}\) and \(s\in\Z_{>0}\) such that \(r/s\leq 1\).
Then any root \(\beta\) of \(X^s-\alpha^r\) satisfies \(q(\beta)\leq q(\alpha)^{r/s}\).
\end{proposition}
\begin{proof}
Let \(\beta\) be a root of \(X^s-\alpha^r\), let \(K=\Q(\alpha,\beta)\) and \(n=[K:\Q]\).
The case \(r=0\) follows from Proposition~\ref{prop:root_of_unity}, so suppose \(r>0\).
Then
\begin{align*}
q(\beta) &= \frac{1}{n}\sum_{\sigma\in\X(K)} |\sigma(\beta)|^2 = \frac{1}{n}\sum_{\sigma\in\X(K)} |\sigma(\beta^s)|^{2/s} = \frac{1}{n}\sum_{\sigma\in\X(K)} |\sigma(\alpha^r)|^{2/s} \\
&= \frac{1}{n}\sum_{\sigma\in\X(K)} \big(|\sigma(\alpha)|^2\big)^{r/s}  \leq \Big(\frac{1}{n}\sum_{\sigma\in\X(K)} |\sigma(\alpha)|^2\Big)^{r/s} = q(\alpha)^{r/s},
\end{align*}
where the inequality is the inequality \(n^{-1/r}\|x\|_r \leq n^{-1/s} \|x\|_s\) from Lemma~\ref{lem:holder} applied to the vector \(x=(|\sigma(\alpha)|^{2/s})_{\sigma\in\X(K)}\), using that \(0<r\leq s\).
\end{proof}

\section{Indecomposable algebraic integers}\label{sec:indecomposable_integers}

We will now focus on the indecomposables of the lattice \(\overline{\Z}\).

\begin{figure}[h]
\centering
\begin{tikzpicture}[scale=1.6]
\coordinate (A) at (0.5,1.322875);
\draw (-0.7,0) -- (1.4,0); 
\draw (0,-0.2) -- (0,1.4); 
\draw[fill=black] (A) circle (.025);
\node[right] (alpha) at ($(A)+(.1,0)$) {$\alpha$};
\node (O) at (-.15,-.15) {0};
\node (one) at (1,-.15) {1};
\draw (1,.05) -- (1,-.05);
\node (i) at (.12,.85) {i};
\draw (.05,1) -- (-.05,1);
\begin{scope}
    \clip (-.65,0) rectangle (1.2,1.2);
    \draw[dashed] (0,0) circle(1);
    \draw[dashed] (A) circle(1);
\end{scope}
\coordinate (zeta) at (0.9114,0.4114);
\coordinate (xi) at (-0.4114,0.9114);
\draw[fill=black] (zeta) circle (0.025);
\draw[fill=black] (xi) circle (0.025);
\draw (0,0) -- (zeta) -- (A) -- (xi) -- (0,0);
\node (z) at ($(zeta)+(0.2,-.05)$) {$\beta$};
\node (x) at ($(xi)+(-.07,-.2)$) {$\gamma$};
\end{tikzpicture}
\caption{Integral \(\alpha=\frac{1}{2}(1+\i\sqrt{7})\) with \(q(\alpha)=2\).}\label{fig:roots_of_unity}
\end{figure}
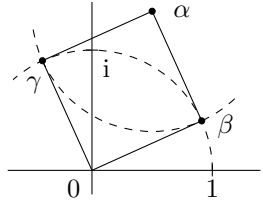

\begin{proposition}\label{prop:q_is_2}
Let \(\alpha\in\overline{\Z}\). If \(0<q(\alpha)<2\), then \(\alpha\) is indecomposable.
If \(q(\alpha)=2\), then \(\alpha\) is decomposable if and only if it is the sum of two roots of unity.
Such roots of unity are necessarily orthogonal, unique up to reordering, and of degree at most $2$ over \(\Q(\alpha)\).
\end{proposition}
\begin{proof}
If \(0<q(\alpha)< 2\), then \(\alpha\) is indecomposable by combining Theorem~\ref{thm:Zbar_is_lattice} and Lemma~\ref{lem:small_indec}.
Suppose \(q(\alpha)=2\). 
If \(\alpha=\zeta+\xi\) for roots of unity \(\zeta,\xi\in\overline{\Q}\), then \(q(\alpha)=2=q(\zeta)+q(\xi)\), so \((\zeta,\xi)\in\textup{dec}(\alpha)\) is non-trivial.
Conversely, suppose \((\beta,\gamma)\in\text{dec}(\alpha)\) is non-trivial. 
By Theorem~\ref{thm:Zbar_is_lattice} we have \(q(\beta),q(\gamma)\geq 1\). 
Then \(0\leq q(\alpha)-q(\beta)-q(\gamma)=2-q(\beta)-q(\gamma)\leq 0\), so we must have \(q(\beta)=q(\gamma)=1\).
It follows that \(\beta\) and \(\gamma\) are orthogonal, and by Proposition~\ref{prop:root_of_unity} they are roots of unity.

Suppose \((\beta,\gamma)\in\text{dec}(\alpha)\) is non-trivial.
For any \(\sigma\in\X(\Q(\alpha))\) and \(\rho\in\X_\sigma(\Q(\alpha,\beta))\) the points \(0\), \(\rho(\alpha)\), \(\rho(\beta)\) and \(\rho(\gamma)\) form the vertices of a rhombus with unit length sides, as can be seen in Figure~\ref{fig:roots_of_unity}. 
It follows that \(\{\rho(\beta),\rho(\gamma)\}\) is uniquely determined by \(\rho(\alpha)=\sigma(\alpha)\).
As \(\rho\) is uniquely determined by \(\rho(\beta)\) there are at most two elements in \(\X_\sigma(\Q(\alpha,\beta))\), in other words \([\Q(\alpha,\beta):\Q(\alpha)]\leq 2\).
\end{proof}

\begin{remark}
Proposition~\ref{prop:q_is_2} gives us a way to decide whether an \(\alpha\in\overline{\Z}\) with \(q(\alpha)=2\) is indecomposable, as it puts an upper bound on the degree of the roots of unity, leaving only finitely many to check.
Knowledge of \(\Q(\alpha)\) can further reduce this number. 
\end{remark}

\begin{example}\label{ex:proof-7}
There exist \(\alpha\in\overline{\Z}\) with \(q(\alpha)=2\) which are indecomposable.
Consider \(f=X^2-X+2\) with root \(\alpha\in\overline{\Z}\), as in Figure~\ref{fig:roots_of_unity}.
Note that the roots of \(f\) in \(\C\) are \(\frac{1}{2}(1\pm \i\sqrt{7})\) with absolute value \(\frac{1}{2}\sqrt{1^2+7}=\sqrt{2}\), so \(q(\alpha)=2\).
Suppose \((\beta,\gamma)\in\textup{dec}(\alpha)\) is a non-trivial decomposition.
Proposition~\ref{prop:q_is_2} shows that \(\beta\) and \(\gamma\) are roots of unity.
Note that \(|\alpha^2|=2\) under all embeddings of \(\alpha\) in \(\C\), and \(\alpha^2=|\alpha|^2\cdot \beta\gamma\).
Hence \(\alpha^2/2=\frac{1}{2}\alpha-1\) is a root of unity.
Either one notes that \(\alpha^2/2\) is not even integral, or that \(\alpha^2/2=\pm 1\), as those are the only roots of unity in \(\Q(\alpha)\), which is clearly absurd.
Hence we have a contradiction and \(\alpha\) is indecomposable.
\end{example}

\begin{lemma}\label{lem:trace}
Let \(\Q\subseteq K\subseteq L \subseteq \overline{\Q}\) be fields with \([L:\Q]<\infty\).
Then for all \(\alpha\in K\) and \(\beta\in L\) we have
\[ [L:K] \cdot \langle \alpha,\beta \rangle = \langle \alpha, \Tr_{L/K}(\beta)\rangle. \] 
\end{lemma}
\begin{proof}
Recall for \(\sigma\in \X(K)\) the definition \(\X_\sigma(L)=\{\rho\in\X(L)\,|\,\rho|_K = \sigma\}\) from Definition~\ref{def:fundamental_set}.
For all \(\sigma\in \X(K)\) and \(\beta\in L\) we have \(\sigma(\Tr_{L/K}(\beta))=\sum_{\rho\in \X_\sigma(L)} \rho(\beta)\). 
Then with \(\alpha\in K\) and \(\beta\in L\) we have
\begin{align*}
[L:K] \cdot \langle \alpha,\beta \rangle 
&= \frac{[L:K]}{[L:\Q]} \sum_{\sigma\in \X(K)} \sum_{\rho\in\X_\sigma(L)} \rho(\alpha) \overline{\rho(\beta)}
= \frac{1}{[K:\Q]} \sum_{\sigma\in \X(K)} \sigma(\alpha) \overline{\sum_{\rho\in\X_\sigma(L)} \rho(\beta) } \\
&= \frac{1}{[K:\Q]} \sum_{\sigma\in \X(K)} \sigma(\alpha) \overline{\sigma( \Tr_{L/K}(\beta) )}
= \langle \alpha, \Tr_{L/K}(\beta) \rangle,
\end{align*}
as was to be shown.
\end{proof}

\begin{proposition}\label{prop:orthogonal_roots_of_unity}
Roots of unity \(\zeta,\xi\in\overline{\Z}\) are orthogonal, i.e. \(\langle \zeta,\xi\rangle = 0\), if and only if \(\zeta^{-1} \xi\) does not have square-free order.
\end{proposition}
\begin{proof}
Let \(K=\Q(\zeta^{-1}\xi)\).
We have \([K:\Q]\cdot\langle\zeta,\xi\rangle= [K:\Q]\cdot\langle 1, \zeta^{-1}\xi\rangle = \Tr_{K/\Q}(\zeta^{-1}\xi) \) by Lemma~\ref{lem:isometries} and Lemma~\ref{lem:trace}. 
Recall that the trace of an \(n\)-th root of unity equals \(\mu(n)\), the M\"obius function, which is zero precisely when \(n\) has a square divisor in \(\Z_{>1}\).
\end{proof}

For \(\alpha,\beta\in\overline{\Z}\) we say \(\beta\) {\em divides} \(\alpha\), and write \(\beta\mid\alpha\), if there exists some \(\gamma\in\overline{\Z}\) such that \(\alpha=\beta\gamma\). 
We write \(\beta\nmid\alpha\) if \(\beta\) does not divide \(\alpha\).
Recall from Definition~\ref{def:uniform} that for \(\delta\in\overline{\Z}\) we say \(\delta\) is uniform if \(|\sigma(\delta)|=|\tau(\delta)|\) for all \(\sigma,\tau\in\X(\overline{\Q})\)

\begin{proposition}
If \(\alpha\in\overline{\Z}\) is such that \(\sqrt{2}\mid\alpha\) or \(\sqrt{3}\mid\alpha\), then \(\alpha\not\in\textup{indec}(\overline{\Z})\).
\end{proposition}
\begin{proof}
Let \(\zeta\in\overline{\Q}\) be a primitive \(8\)-th root of unity, which we may choose such that \(\zeta+\zeta^{-1}=\sqrt{2}\).
Thus \((\zeta,\zeta^{-1})\in\textup{dec}(\sqrt{2})\), because \(\langle\zeta,\zeta^{-1}\rangle=0\) by Proposition~\ref{prop:orthogonal_roots_of_unity}. 
Moreover, \(\sqrt{2}\), \(\zeta\) and \(\zeta^{-1}\) are all uniform. 
For any \(\beta\in\overline{\Z}\) we get from Lemma~\ref{lem:uniform} that 
\[q(\zeta \beta)+q(\zeta^{-1}\beta)=(q(\zeta)+q(\zeta^{-1})) \cdot q(\beta) = q(\sqrt{2})\cdot q(\beta) = q(\sqrt{2}\beta),\]
so \(\sqrt{2}\beta\) has a non-trivial decomposition.

With \(\xi\) a twelfth root of unity we have \(\xi+\xi^{-1}=\sqrt{3}\) with \(\xi\), \(\xi^{-1}\) and \(\sqrt{3}\) uniform. 
We have a decomposition because \(\langle\xi,\xi^{-1}\rangle=\langle1,\xi^{-2}\rangle =\frac{1}{2}\geq 0\), so the argument from before applies.
\end{proof}

We will now begin the proof of Theorem~\ref{thm:largest_indec}.

\begin{lemma}\label{lem:2sqrt2_help}
If \(\alpha\in\overline{\Z}\) is such that \(\sqrt{2}\nmid \alpha \mid 2\) and \(\alpha\) is uniform, then \(\alpha\in\textup{indec}(\overline{\Z})\).
\end{lemma}
\begin{proof}
By assumption we may write \(2=\alpha\gamma\) for some non-zero \(\gamma\in\overline{\Z}\).
Note that \(\gamma\) is not a unit, since otherwise \(\sqrt{2} \mid 2 \mid \alpha \).
Now let \((\beta,\alpha-\beta)\in\textup{dec}(\alpha)\).
Then by Lemma~\ref{lem:eq_dec_def} and Lemma~\ref{lem:uniform} we have \(q(\alpha)\geq q(\alpha-2\beta)=q(\alpha-\alpha\beta\gamma)=q(\alpha) \cdot q(1-\beta\gamma)\), so \(q(1-\beta\gamma)\leq 1\).
As \(\gamma\) is not a unit we have \(\beta\gamma\neq 1\), so \(\beta\gamma=1-\zeta\) for some root of unity \(\zeta\) of order say \(n\) by Proposition~\ref{prop:root_of_unity}.
Suppose \(n\) is not a power of \(2\).
Then \(1-\zeta\) and \(2\) are coprime.
As \(2 \mid 2\beta = \alpha (1-\zeta)\) we have that \(2 \mid \alpha\), which contradicts \(\sqrt{2}\nmid \alpha\).
Hence \(n\) is a power of \(2\).
If \(n>2\), then \(1-\zeta \mid \sqrt{2}\) so \(\sqrt{2} \mid \alpha\), which is again a contradiction.
Therefore \(n=1\) or \(n=2\), which correspond to the trivial decompositions with \(\beta=0\) and \(\beta=\alpha\) respectively.
We conclude that \(\alpha\) is indecomposable.
\end{proof}

\noindent\textbf{Theorem~\ref{thm:largest_indec}. }{\em It holds that \(2\sqrt{2}\leq\sup\{ \langle\alpha,\alpha\rangle\,|\,\textup{\(\alpha\in\overline{\Z}\) is indecomposable} \}\).}

\begin{proof}
We will prove that for each \(r\in\Q\cap[1,3/2)\) there exists \(\alpha\in\textup{indec}(\overline{\Z})\) such that \(q(\alpha)=2^r\).

Consider \(\beta = \frac{1+\sqrt{-7}}{2}\) as in Example~\ref{ex:proof-7} and write \(\overline{\beta}=1-\beta\) for its conjugate. 
Write \(r=\frac{a}{b}\) with integers \(a\geq b>0\) and let \(\gamma\in\overline{\Z}\) be a zero of \(X^{b}-\beta\).
Now take \(\alpha=\overline{\beta}\cdot \gamma^{a-b}\).
We will show \(\alpha\) satisfies the conditions to Lemma~\ref{lem:2sqrt2_help}.
Because \(|\sigma(\alpha)|=|\sigma(\beta)|^r=2^{r/2}\) for all \(\sigma\in\X(\overline{\Q})\), and hence \(\alpha\) is uniform, we then have that \(\alpha\) is indecomposable and \(q(\alpha)=2^r\).

Note that \(\alpha \cdot \gamma^{2b-a} = \overline{\beta} \cdot \beta = 2\), so \(\alpha \mid 2\).
Let \(v:K\to\R\cup\{\infty\}\) be a valuation over \(2\) for some number field \(K\) which is Galois over \(\Q\) containing the relevant elements.
Because \(0=v(1)=v(\overline{\beta}+\beta)\geq \min\{v(\overline{\beta}),v(\beta)\}\), we have \(v(\beta)=0\) or \(v(\overline{\beta})=0\).
By potentially composing \(v\) with an automorphism swapping \(\beta\) and \(\overline{\beta}\) we obtain a valuation \(v'\) such that \(v'(\overline{\beta})=0\). 
We have \(1=v'(2)=v'(\beta\cdot\overline{\beta})=v'(\beta)\) and thus \(v'(\gamma)=1/b\).
Then \(v'(\alpha)=r-1 < 1/2 = v'(\sqrt{2})\), from which we conclude that \(\sqrt{2}\nmid \alpha\).
Thus \(\alpha\) satisfies the conditions to Lemma~\ref{lem:2sqrt2_help}, as was to be shown.
\end{proof}

\label{sec:end-integers}
\section{Decompositions of the lattice of algebraic integers}\label{sec:Zbar_indecomposable} \label{sec:begin-application}

In this section we will show that \(\overline{\Z}\) is indecomposable as a Hilbert lattice.
The following is a standard result from linear algebra.

\begin{lemma}\label{lem:vector_space_union}
Let \(V\) be a vector space over an infinite field and let \(S\) be a finite set of subspaces of \(V\).
If \(\bigcup_{U\in S} U = V\), then \(V\in S\). \qed
\end{lemma}

\begin{proposition}\label{prop:find_non_ortho}
Let \(S\subseteq \overline{\Z}\) with \(S\) finite and \(0\not\in S\). Then there exist \(\alpha\in\textup{indec}(\overline{\Z})\) such that \(\langle \alpha,\beta\rangle \neq 0\) for all \(\beta\in S\).
\end{proposition}
\begin{proof}
Let \(K\) be the field generated by \(S\) and fix \(1<r<\sqrt{2}\).

We will construct an element \(u\in \mathcal{O}_K\) such that \(0\not\in\langle u, S\rangle\) and \(|\sigma(u)|>r\) for all \(\sigma\in \X(K)\).
For \(x\in K\) write \(x^{\bot} =\{ y \in K \,|\, \langle x,y\rangle = 0 \}\), which is a proper \(\Q\)-vector subspace of \(K\) when \(x\neq 0\) because \( x\not\in x^\bot\).
Hence \(\bigcup_{x\in S} x^\bot \neq K\) by Lemma~\ref{lem:vector_space_union}, so there exists some non-zero \(u\in K\) such that \(0\not\in\langle u, S\rangle\).
By scaling \(u\) by some non-zero integer we may assume \(u\in\overline{\Z}\) as well.
By further scaling \(u\) with integers we may assume \(|\sigma(u)|>r\) for all \(\sigma\in\X(K)\), as was to be shown.

As \(|\sigma(u)|>r\) for all \(\sigma\in\X(K)\) we have \(N(u)>r> 1\), where \(N\) is as in Definition~\ref{def:Nq}, so \(u\) is not a unit. 
Let \(\p\subseteq\mathcal{O}_K\) be a prime containing \(u\) and let \(v\in \p\setminus\p^2\).
Let \(f_n=X^n-uX^{n-1}-v\in\mathcal{O}_K[X]\) for \(n\geq 2\) and note that it is Eisenstein at \(\p\) and therefore irreducible.
Let \(\alpha_n\in\overline{\Z}\) be a root of \(f_n\).
It suffices to show that for \(n\) sufficiently large \(\alpha_n\) is indecomposable and satisfies \(0\not\in\langle \alpha_n, S\rangle\). 
By Lemma~\ref{lem:trace} and by construction of \(u\) it holds for any \(n\geq 2\) that
\[ \langle \alpha_n, S\rangle = \frac{\langle \Tr_{K(\alpha_n)/K}(\alpha_n), S \rangle}{[K(\alpha_n):K]}  = \frac{\langle u, S \rangle}{[K(\alpha):K]} \not\ni 0, \]
so it remains to be shown that \(\alpha_n\) is indecomposable for \(n\) sufficiently large.

Let \(D\subseteq\C\) be the closed disk of radius \(r\) around \(0\).
Let \(n\) be sufficiently large such that we have \(|\sigma(v)| \cdot r^{1-n} < |\sigma(u)|-r\) for all \(\sigma\in\X(K)\).
Fix \(\sigma\in\X(K)\). 
For all \(x\) on the boundary of \(D\) we have
\[ |x^n-\sigma(v)|\leq r^n+|\sigma(v)| = r^{n-1}(r+|\sigma(v)|\cdot r^{1-n}) < |\sigma(u)| \cdot r^{n-1} = |\sigma(u)\cdot x^{n-1}|.  \]
Hence by Rouch\'e's Theorem \(\sigma(u) X^{n-1}\) and \(\sigma(f_n)=(X^n-\sigma(v))-\sigma(u)X^{n-1}\) have the same number of zeros in \(D\), counting multiplicities, which for \(\sigma(u)X^{n-1}\) clearly is \(n-1\).
For the remaining zero \(x_{\sigma,n}\in\C\) of \(\sigma(f_n)\) with \(|x_{\sigma,n}|>r\) we have \(x_{\sigma,n}^{n-1}(x_{\sigma,n}-\sigma(u))=\sigma(v)\) and thus 
\[|x_{\sigma,n}-\sigma(u)|=|\sigma(v)| \cdot |x_{\sigma,n}|^{1-n} < |\sigma(v)| \cdot r^{1-n} \to 0 \quad(\text{as }n\to\infty),\] 
i.e.\ \(\lim_{n\to\infty} x_{\sigma,n} = \sigma(u)\).
Now summing over all \(\sigma\in\X(K)\) we get
\begin{align*}
q(\alpha_n) &= \frac{1}{n\cdot [K:\Q]} \sum_{\sigma\in X(K)} \sum_{\rho\in \X_\sigma(K(\alpha_n))} |\rho(\alpha_n)|^2 \leq \frac{1}{n \cdot[K:\Q]} \sum_{\sigma\in X(K)} \big( (n-1)r^2+|x_{\sigma,n}|^2\big) \\
&\leq r^2 + \frac{1}{n\cdot [K:\Q]}\sum_{\sigma\in\X(K)} |x_{\sigma,n}|^2 \to r^2 \quad(\text{as }n\to\infty).
\end{align*}
Because \(r^2<2\) we have for sufficiently large \(n\) that \(q(\alpha_n) < 2\).
From Proposition~\ref{prop:q_is_2} we may then conclude that \(\alpha_n\) is indecomposable, as was to be shown.
\end{proof}

\noindent\textbf{Theorem~\ref{thm:Zbar_indecomposable}. }{\em The Hilbert lattice \(\overline{\Z}\) is indecomposable.}

\begin{proof}
Let \(\beta,\gamma\in\textup{indec}(\overline\Z)\). 
Then there exists some \(\alpha\in\textup{indec}(\overline{\Z})\) such that \(\langle\alpha,\beta\rangle\neq 0 \neq \langle\alpha,\gamma\rangle\) by Proposition~\ref{prop:find_non_ortho}.
Hence \(\alpha\), \(\beta\) and \(\gamma\) must be in the same connected component of the graph of Theorem~\ref{thm:eichler}.
As this holds for all \(\beta\) and \(\gamma\) the graph is connected and hence \(\overline{\Z}\) is indecomposable.
\end{proof}

\section{Universal gradings} \label{sec:graded_rings}

\begin{definition}
Let \(R\) be a commutative ring.
A {\em grading} of \(R\) is a pair \((G,\mathcal{R})\) where \(G\) is an abelian group and \(\mathcal{R}\) is a \(G\)-indexed collection \(\{R_g\}_{g}\) of additive subgroups of \(R\) such that the natural map \(\bigoplus_{g\in G} R_g \to R\) is an isomorphism and for all \(g,h\in\Gamma\) we have \(R_g \cdot R_h \subseteq R_{gh}\).
We equip the set of gradings of \(R\) with a category structure, where the morphisms \((G,\{R_g\}_{g\in G})\) to \((H,\{R_h'\}_{h\in H})\) are the group homomorphisms \(f:G\to H\) such that \(R_{h}'=\bigoplus_{g\in f^{-1}\{h\}} R_g\) for all \(h\in H\).
A {\em universal grading} of \(R\) is a universal object in the category of gradings of \(R\).
\end{definition}

Note that a universal grading, when it exists, is {\em uniquely unique}, i.e. for any two universal gradings of a given commutative ring there exists a unique isomorphism between them.

An {\em order} is a commutative ring for which the additive group is isomorphic to \(\Z^r\) for some \(r\in \Z_{\geq 0}\).
Lenstra and Silverberg \citep{Lenstra2018} showed using a classical version of Theorem~\ref{thm:eichler}, that every grading of a reduced order, in particular every order in \(\overline{\Q}\), has a universal grading.
Theorem~\ref{thm:eichler} is powerful enough to prove Theorem~\ref{thm:universal_grading}, which states that every subring of \(\overline{\Z}\) has a universal grading.
We quickly go through the proof, which generalizes that of Lenstra and Silverberg.

\begin{lemma}\label{lem:subring}
Suppose \(R\) is a commutative ring with a grading \((G,\{R_g\}_{g})\). Then \(R_1\subseteq R\) is a subring.
If \(y\in R_h\) for some \(h\in G\) is invertible in \(R\), then \(y^{-1}\in R_{h^{-1}}\) and \(y R_1=R_h\).
\end{lemma}
\begin{proof}
It follows directly from the definition that \(R_1\) is an additive group and closed under multiplication.
Write \(1=\sum_{g\in G} x_g\) with \(x_g\in R_g\) for all \(g\).
Then \(x_h=1\cdot x_h = \sum_{g\in G} x_g x_h \) with \(x_g x_h\in R_{gh}\) for all \(h\in G\). 
Since \(\{R_g\}_g\) is a decomposition of \(R\) we have \(x_1 x_h = x_h\) and thus \(x_1 \cdot 1 = 1\).
Hence \(1=x_1\in R_1\), and \(R_1\) is a subring of \(R\).

As before write \(y^{-1}=\sum_{g\in G} z_g\) with \(z_g\in R_g\) for all \(g\).
As \(R_1 \ni 1 = y y^{-1} = \sum_{g\in G} y z_g\) it follows that \(yz_{h^{-1}} = 1\) and thus \(y^{-1}=z_{h^{-1}}\in R_{h^{-1}}\).
Finally \(yR_1\subseteq R_h =yy^{-1} R_h \subseteq y R_1\), so that \(yR_1=R_h\).
\end{proof}

\begin{lemma}\label{lem:torsion_group}
Suppose \(R\subseteq\overline{\Z}\) is a subring and \((G,\{R_g\}_{g})\) is a grading of \(R\). 
Then \(H=\{ g\in G\,|\, R_g\neq 0\}\) is a countable torsion subgroup of \(G\).
Moreover, if \((G,\{R_g\}_g)\) is universal, then \(G=H\).
\end{lemma}
\begin{proof}
Since \(0\) is the only zero-divisor in \(\overline{\Z}\), we have for \(g,h\in H\) that \(0\subsetneq R_g R_h \subseteq R_{gh}\), so \(gh\in H\).
For \(g\in H\) and \(x\in R_g\) non-zero we have \(x^n=\sum_{i=0}^{n-1} a_i x^i\) for some \(n\in\Z_{\geq 1}\) and \(a_i\in\Z\), so \(0\neq x^n \in R_{g^n} \cap \sum_{i=0}^{n-1} R_{g^i}\). 
Hence \(g^n=g^i\) for some \(0\leq i < n\), so the order of \(g\) is finite, and \(H\) is a torsion group. That it is countable follows from the countability of \(R\).

Suppose now that \((G,\{R_g\}_g)\) is universal.
By what we have shown before \((H,\{R_h\}_h)\) is a grading of \(R\), and the inclusion \(i:H\to G\) is a morphism of gradings \((H,\{R_h\}_h)\to(G,\{R_g\}_g)\) of \(R\).
By universality there exists a map \(u:G\to H\) of gradings in the opposite direction. 
The composition \(i\circ u\) is an endomorphism of \((G,\{R_g\}_g)\), which by universality is unique and thus equal to the identity.
Hence \(i\) is surjective. It follows that \(G = H\), as was to be shown.
\end{proof}

\begin{definition}
Let \(n\in\Z_{\geq 1}\).
We write \(\mu_n\subseteq\overline{\Z}\) for the group of \(n\)-th roots of unity.
We equip  \(\overline{\Z}\tensor_\Z\Z[\mu_n]\) with a Hilbert lattice structure with the inner product given by
\[ \langle x,y\rangle_n = \frac{1}{[L:\Q]} \sum_{\sigma\in\X(L\tensor\Z[\mu_n])} \sigma(x)\cdot \overline{\sigma(y)}\]
for \(x,y\in \overline{\Z}\tensor_\Z\Z[\mu_n]\) and any \(\Q\subseteq L \subseteq \overline{\Q}\) such that \([L:\Q]<\infty\) and \(x,y\in L\tensor \Z[\mu_n]\). 
\end{definition}

The proof that this is a Hilbert lattice is completely analogous to the proof for \(\overline{\Z}\).
One can show that for \(x,y\in\overline{\Z}\) we have \(\langle x,y\rangle_n = \varphi(n) \langle x,y\rangle\), where \(\varphi\) is the Euler totient function, using that we have a natural bijection \(\X(L)\times \X(\Z[\mu_n]) \to \X(L\tensor\Z[\mu_n])\).

\begin{lemma}\label{lem:pontryagin_orthogonal}
Suppose \(R\subseteq\overline{\Z}\) is a subring and \((G,\{R_g\}_{g})\) is a grading of \(R\). 
Then for all \(g,h\in G\) distinct we have \(\langle R_g,R_h\rangle = 0\).
\end{lemma}
\begin{proof}
Let \(H\) be as in Lemma~\ref{lem:torsion_group}.
If \(R_g=0\) or \(R_h=0\) we are done, so suppose \(g,h\in H\) and let \(n\in\Z_{\geq 1}\) such that \(g^n=1=h^n\).
As \(\varphi(n)\langle R_g,R_h\rangle=\langle R_g,R_h\rangle_n\) it suffices to show that \(\langle R_g,R_h\rangle_n=0\).
Since \(g\neq h\) there exists by Pontryagin duality some morphism \(\chi\in\Hom(H,\mu_n)\) such that \(\chi(g)\neq \chi(h)\).
One then verifies that the induced \(\Z[\mu_n]\)-algebra automorphism \(\chi\) of \(R\tensor\Z[\mu_n]\) that maps \(x\tensor 1\in R_i\tensor\Z[\mu_n]\) to \( x \tensor \chi(i)\) for all \(i\in H\) is an isometry. 
Hence for \(x\in R_g\) and \(y\in R_h\) we have
\[ \langle x, y \rangle_n = \langle \chi x, \chi y\rangle_n = \langle \chi(g) x, \chi(h) y\rangle_n = \langle x,\chi(g)^{-1}\chi(h) y\rangle_n.  \]
Because \(\chi(g^{-1}h)\) is a non-trivial \(n\)-th root of unity, \(1-\chi(g^{-1}h)\) divides \(n\).
Hence
\[0=\langle R_g,(1-\chi(g^{-1}h))R_h\rangle_n \supseteq \langle R_g,nR_h\rangle_n = n \langle R_g, R_h\rangle_n,\]
and we are done.
\end{proof}

\noindent\textbf{Theorem~\ref{thm:universal_grading}. }{\em Every subring of \(\overline{\Z}\) has a universal grading. }

\begin{proof}
Let \(R\) be a subring of \(\overline{\Z}\), which is also a sublattice of \(\overline{\Z}\).
Let \(\textup{Y}=\{\textup{Y}_i\}_{i\in I}\) be a universal decomposition of the lattice \(R\), which exists by Theorem~\ref{thm:eichler}.
We obtain this decomposition by starting with the graph \(\mathcal{G}\) on the vertex set \(\textup{indec}(R)\) with edges between \(x,y\in\textup{indec}(R)\) if and only if \(\langle x,y\rangle \neq 0\), then taking \(I\) to be the set of connected components of \(\mathcal{G}\) and \(\textup{Y}_i\) the group generated by \(i\in I\).
For \(x=\sum_{i}\upsilon_i \in R\) with \(\upsilon_i\in \textup{Y}_i\) write \(\textup{supp}(x)=\{i\in I\,|\, \upsilon_i\neq 0\}\).
Now consider the free abelian group \(\Z^{(I)}\) and let \(U\) be the group obtained from it by dividing out 
\[ N = \langle i+j-k \,|\, i,j\in I,\,k\in\textup{supp}(\textup{Y}_i\cdot\textup{Y}_j)\rangle. \]
We have an induced map \(f:I\to \Z^{(I)} \to U\) which induces a decomposition \(f(\textup{Y})\) of \(R\), which is also a grading.
That it is universal follows from the fact that \(\textup{Y}\) is a refinement of \(\mathcal{R}\) for any grading \((G,\mathcal{R})\) by Lemma~\ref{lem:pontryagin_orthogonal}.
\end{proof}

\begin{example}
From Theorem~\ref{thm:universal_grading} it follows in particular that \(\overline{\Z}\) has a universal grading.
It follows from Lemma~\ref{lem:pontryagin_orthogonal} that such a grading gives an orthogonal decomposition of the lattice.
By Theorem~\ref{thm:Zbar_indecomposable} such a decomposition is trivial, and thus \(\overline{\Z}\) has a trivial universal grading.
\end{example}

\begin{example}
Every countable torsion group occurs as the group of a universal grading of a subring of \(\overline{\Z}\).
Note that such a group is a subgroup of \(\Omega=\bigoplus_{\smash{p\in\mathcal{P}}} (\Q/\Z)\), where \(\mathcal{P}\) is some countably infinite set.
We choose \(\mathcal{P}\) to be the set of positive prime numbers.
Fixing some embedding \(\overline{\Z}\to\C\) we have a well-defined \(x\)-th power of \(p\) in \(\overline{\Q}\cap \R_{>0}\) for all \(x\in\Q\).
Let \([\cdot]:\Q/\Z\to[0,1)\cap\Q\) be the (bijective) map that assigns to each class its smallest non-negative representative.
It is then easy to verify that \(R = \Z[p^x \,|\, p\in\mathcal{P},\,x\in\Q_{\geq 0}]\subseteq\overline{\Z}\) has a grading \((\Omega,(R_{(x_p)_p})_{(x_p)_p})\) with 
\[R_{(x_p)_p} = \Big(\prod_{p\in \mathcal{P}} p^{[x_p]}\Big) \cdot \Z.\]
In turn any subgroup \(G\subseteq\Omega\) gives a grading \((G,(R_g)_g)\) of the subring \(\bigoplus_{g\in G} R_g\subseteq R\), and this grading must be universal because a universal grading exists and all \(R_g\) are of rank 1.
\end{example}

\begin{lemma}\label{lem:homogenous_intersection}
Suppose \(R\subseteq\overline{\Z}\) is a subring and \((G,\{R_g\}_{g})\) is a grading of \(R\). 
If \(K=\Q(A)\) for some subset \(A\subseteq \bigcup_{g\in G} R_g\), then \((G,\{R_g\cap K\}_g)\) is a grading of \(R\cap K\).
\end{lemma}
\begin{proof}
It is clear that \((G,\{R_g\cap K\}_g)\) is a grading of \(R\cap K\) once we show \(\bigoplus_{g} (R_g\cap K)=R\cap K\). 
For this it remains to show that \(R\cap K \subseteq \sum_g (R_g\cap K)\). 
Let \(x\in R\cap K\). 
As \(x\in R\) we may uniquely write \(x=\sum_{g} x_g\) for some \(x_g\in R_g\). 
Without loss of generality \(A\) is closed under multiplication, so that \(A\) generates \(K\) as a \(\Q\)-vector space.
Then we may write \(x=\sum_{a\in A} r_a a\) for some \(r_a\in \Q\) which are almost all equal to zero.
Hence a positive integer multiple \(\lambda x\) of \(x\) satisfies \(\sum_g \lambda x_g=\lambda x = \sum_{a\in A} \lambda r_a a\) with \(\lambda r_a\in\Z\) for all \(a\) and thus \(\lambda r_a a\in R_{g_a}\) for some \(g_a\in G\). 
It follows from uniqueness of the decomposition that \(\lambda x_g = \sum_{a\in A,\, g_a=g} \lambda r_a a\) and thus \(x_g\in K\).
We conclude that \(x_g\in R_g\cap K\) and thus \(x\in \sum_g (R_g\cap K)\), as was to be shown.
\end{proof}

We may also lift Theorem~1.4 of \citep{Lenstra2018} to our more general setting.

\vspace{6pt}

\noindent\textbf{Theorem~\ref{thm:integral_universal_grading}. }{\em Every integrally closed subring of \(\overline{\Z}\) has a universal grading with a subgroup of \(\Q/\Z\).}

\begin{proof}
Let \(R\) be an integrally closed subring of \(\overline{\Z}\) and let \((G,\{R_g\}_g)\) be a universal grading, which exists by Theorem~\ref{thm:universal_grading}.
It suffices to show that every finitely generated subgroup \(H\) of \(G\) is cyclic.

Let \(H\subseteq G\) be finitely generated and thus finite by Lemma~\ref{lem:torsion_group}.
Moreover, by Lemma~\ref{lem:torsion_group} we have \(R_h\neq 0\) for all \(h\in H\), so we may choose some non-zero \(a_h\in R_h\).
Let \(A=\{a_h\mid h\in H\}\) and \(K=\Q(A)\).
Then by Lemma~\ref{lem:homogenous_intersection} we get a grading \((H,\{R_h\cap K\})\) of \(S=R\cap K\).
Since \(K\) is a field and \(R\) is integrally closed, the ring \(S\) is integrally closed.
As \(K\) is finite over \(\Q\), the same holds for the field of fractions of \(S\). 
Hence we may apply Theorem~1.4 from \citep{Lenstra2018} to conclude that the universal grading of \(S\) has a cyclic grading group \(Y\). 
By universality we get a morphism of gradings and thus a morphism of groups \(Y\to H\). 
The latter is surjective since \(0\neq R_h\cap K \ni a_h\) for all \(h\in H\).
Thus \(H\) is cyclic, as was to be shown.
\end{proof}

\begin{lemma}\label{lem:sufficient_universal}
Suppose \(R\subseteq\overline{\Z}\) is a subring and \((G,\{R_g\}_{g})\) is a grading of \(R\). 
If the universal grading of \(R_1\) is trivial and \(R_g\neq 0\) for all \(g\in G\), then \((G,\{R_g\}_g)\) is universal.
\end{lemma}
\begin{proof}
Suppose \((H,\{S_h\}_h)\) is a universal grading of \(R\), which exists by Theorem~\ref{thm:universal_grading}, and let \(f:(H,\{S_h\}_h)\to (G,\{R_g\}_g)\) be the map given by universality.
Then \(R_1=\bigoplus_{h\in\ker(f)} S_h\), which is a grading of \(R_1\).
Since the universal grading of \(R_1\) is trivial, it follows that \(R_1=S_1\).
By Lemma~\ref{lem:torsion_group} we have \(S_h\neq 0\) for all \(h\in\ker(f)\), so it follows that \(\ker(f)=1\) and that \(f\) is injective.
From the fact that \(R_g\neq 0\) for all \(g\in G\) it follows that \(f\) must be surjective.
Thus \(f\) is an isomorphism of gradings and \((G,\{R_g\}_g)\) is universal.
\end{proof}

\begin{example}
We will show that every subgroup of \(\Q/\Z\) occurs as the group of a universal grading of an integrally closed subring of \(\overline{\Z}\).

Recall from Definition~\ref{def:roots_of_unity} the notation \(\mu_\infty\) for the group of roots of unity in \(\overline{\Z}\).
For a prime number \(p\) write \(\mu_{p^\infty}=\{ \zeta\in\mu_\infty \,|\, (\exists\, n \in \Z_{\geq 0})\ \zeta^{p^n} = 1 \}\) and \(\mu_p=\{\zeta\in\mu_\infty\,|\,\zeta^p=1\}\).
The map \(\zeta\mapsto \zeta^p\) gives an isomorphism \(\mu_{p^\infty}/\mu_p\to \mu_{p^\infty}\).
Taking the direct sum over all \(p\) we get an isomorphism \(\mu_\infty/\mu_0\to\mu_\infty\), where \(\mu_0=\{\zeta\in\mu_\infty\,|\, (\exists\,n\text{ square-free})\ \zeta^n=1\}\).
Thus it suffices to show that for every \(\mu_0\subseteq M \subseteq \mu_\infty\) the group \(G=M/\mu_0\) occurs as a universal grading group.

Consider \(R=\Z[M]\), the smallest subring of \(\overline\Z\) containing \(M\), which is integrally closed. 
Define \(R_{\zeta\cdot\mu_0} = \zeta \cdot \Z[\mu_0]\) for all \(\zeta\cdot\mu_0\in M/\mu_0\) and note that this gives a grading \((G,\{R_g\}_{g})\) of \(R\).
To prove this is a universal grading it suffices by Lemma~\ref{lem:sufficient_universal} to show that the universal grading of \(\Z[\mu_0]\) is trivial, or in turn, by Lemma~\ref{lem:pontryagin_orthogonal}, that \(\Z[\mu_0]\) is indecomposable.
The elements of \(\mu_0\) are indecomposable in \(\Z[\mu_0]\) because they are so in \(\overline{\Z}\), and they generate \(\Z[\mu_0]\) as an additive group.
From Proposition~\ref{prop:orthogonal_roots_of_unity} we may conclude that no pair \(\zeta,\xi\in\mu_0\) is orthogonal, so from Theorem~\ref{thm:eichler} it follows that \(\Z[\mu_0]\) is indecomposable. Hence the grading is universal.
\end{example}

\label{sec:end-application}
\section{Enumeration of indecomposable algebraic integers of degree 2}
\label{sec:enum2}

The indecomposables of \(\overline{\Z}\) of degree 1 are \(1\) and \(-1\). 
In this section we compute the indecomposables of degree 2.
The fields of degree 2 over \(\Q\) are \(\Q(\sqrt{d})\) for \(d\in\Z\setminus\{1\}\) square-free.
The following lemma is easily verified by separating the cases \(d\) negative and positive.

\begin{lemma}
Let \(d\in\Z\setminus\{1\}\) be square-free and let \(a,b\in\Q\). Then \(q(a+b\sqrt{d})=a^2+|d|\cdot b^2\).\qed
\end{lemma}

\begin{lemma}\label{lem:help_deg_2_dec}
Let \(\alpha\in\overline{\Z}\) and suppose one of the following holds:
\begin{enumerate}[nosep,label=\textup{(\roman*)}]
\item the real part of \(\alpha^2\) is at least \(2\) under every embedding \(\Q(\alpha)\to\C\);
\item the real part of \(\alpha^2\) is at most \(-2\) under every embedding \(\Q(\alpha)\to\C\);
\item \(\alpha=(1+\sqrt{d})/2\) with \(d\in\Z\) square-free such that \(9\leq d \leq 25\).
\end{enumerate}
Then \(\alpha\) has a non-trivial decomposition in a degree \(2\) extension of \(\Q(\alpha)\).
\end{lemma}
\begin{proof}
Let \(K=\Q(\alpha)\) and \(\gamma=\alpha^2/4\). 
Let \(f=X^2-\alpha X+1 \in K[X]\), let \(\beta\in\overline{\Z}\) be a root of \(f\) and write \(L=K(\beta)\).
Then \((\beta-\alpha/2)^2=\beta^2-\alpha\beta+\alpha^2/4=\alpha^2/4-1=\gamma-1\).
For (i) and (iii) we will show \(q(\beta-\alpha/2)\leq q(\alpha/2)\).
Then \((\beta,\alpha-\beta)\) is a decomposition of \(\alpha\) by Lemma~\ref{lem:eq_dec_def}, and since neither \(0\) nor \(\alpha\) is a root of \(f\) we conclude that this decomposition is non-trivial.

(i) 
Let \(\sigma\in\X(L)\). 
For \(\delta\in L\) write \(\textup{Re}_\sigma(\delta)\) and \(\textup{Im}_\sigma(\delta)\) for the real respectively imaginary part of \(\sigma(\delta)\).
By assumption \(\textup{Re}_\sigma(\gamma)\geq 1/2\).
Thus \(\textup{Re}_\sigma(\gamma-1)^2 = (\textup{Re}_\sigma(\gamma)-1)^2 \leq \textup{Re}_\sigma(\gamma)^2\). 
As \(\textup{Im}_\sigma(\gamma-1)^2=\textup{Im}_\sigma(\gamma)^2\) we may conclude that \(|\sigma(\gamma-1)|\leq |\sigma(\gamma)|\).
Then
\begin{align*}
q(\beta-\alpha/2) 
%= \frac{1}{[L:\Q]}\sum_{\sigma\in \X(L)} |\sigma(\beta-\alpha/2)^2| 
= \frac{1}{[L:\Q]} \sum_{\sigma\in \X(L)} |\sigma(\gamma-1)| \leq \frac{1}{[L:\Q]} \sum_{\sigma\in \X(L)} |\sigma(\gamma)| 
= q(\alpha/2),
\end{align*}
as was to be shown.

(ii) 
Let \(\i\in\overline{\Q}\) be a primitive fourth root of unity. Then \(\i \alpha\) satisfies the conditions to (1), hence it has a non-trivial decomposition \((\beta,\i\alpha-\beta)\), where \(\beta\) is a root of \(X^2-\i\alpha X+1\).
In turn, \((-\i\beta,\alpha+\i\beta)\) is a non-trivial decomposition of \(\alpha\), where \(-\i\beta\) is a root of \(X^2-\alpha X-1\). In particular \(-\i\beta\) is of degree at most \(2\) over \(\Q(\alpha)\).

(iii) Since \(d>0\) the field \(K\) is totally real.
Let \(\gamma_1,\gamma_2\in\R\) be the images of \(\gamma\) under \(\X(K)\) such that \(\gamma_1<\gamma_2\). Because \(9\leq d \leq 25\) we have \(\gamma_1 = (\sqrt{d}-1)^2/16 \leq 1\) and  \(\gamma_2 = (\sqrt{d}+1)^2/16 \geq 1\).
Hence
\[q(\beta-\alpha/2)=\frac{1}{2}\big( |\gamma_1-1|+|\gamma_2-1| \big) = \frac{1}{2}\big( (1-\gamma_1) + (\gamma_2-1) \big) = \frac{2\sqrt{d}}{16} \leq \frac{1+d}{16}=q(\alpha/2), \]
as was to be shown.
\end{proof}

\begin{theorem}\label{thm:deg2_indec}
The indecomposable elements of \(\overline{\Z}\) of degree $2$ up to conjugacy and sign are \(\sqrt{-1}\), \(\frac{1+\sqrt{-7}}{2}\), \(\frac{1+\sqrt{-3}}{2}\) and \(\frac{1+\sqrt{5}}{2}\), for a total of $14$ indecomposables.
\end{theorem}
\begin{proof}
First note that the \(4\) listed elements indeed are indecomposable: 
We treated \((1+\sqrt{-7})/2\) in Example~\ref{ex:proof-7}, and the remaining \(3\) have square-norm less than \(2\), so Proposition~\ref{prop:q_is_2} applies.
Since conjugation and multiplication by \(-1\) are isometries by Lemma~\ref{lem:isometries}, all \(14\) are indecomposable.

Let \(\alpha\in\overline\Z\) be of degree \(2\) over \(\Q\).
It remains to show that \(\alpha\), up to conjugation and sign, admits a non-trivial decomposition or is one of the $4$ listed indecomposables.
Since \(\alpha\) is of degree 2 over \(\Q\) it is an element of \(\Q(\sqrt{d})\) for some square-free \(d\in\Z\setminus\{1\}\). 
Then we may write \(\alpha=(a+b\sqrt{d})/2\) for some \(a,b\in\Z\) with \(a+b\in 2\Z\) and by conjugating and changing sign we may assume \(a,b\geq 0\).
If \(a\geq 2\) we have
\[ q(\alpha/2-1) = \Big(\frac{a}{4}-1\Big)^2+|d|\Big(\frac{b}{4}\Big)^2 = \Big( \Big(\frac{a}{4}\Big)^2+|d|\Big(\frac{b}{4}\Big)^2 \Big) + \Big(1-\frac{a}{2}\Big) \leq q(\alpha/2), \]
so \((1,\alpha-1)\) is a decomposition of \(\alpha\). 
Since \(\alpha\neq 1\), this decomposition is non-trivial.
Similarly we get a decomposition \((\sqrt{d},\alpha-\sqrt{d})\) if \(b\geq 2\), so either this decomposition is non-trivial or \(\alpha=\sqrt{d}\).

First suppose \(\alpha=\sqrt{d}\). 
If \(|d|<2\) we have \(d=-1\) and \(\sqrt{-1}\) is listed.
Otherwise \(\alpha^2=d\) satisfies the conditions to Lemma~\ref{lem:help_deg_2_dec}.i or ii, so \(\sqrt{d}\) is not indecomposable.
For \(\alpha\neq \sqrt{d}\) the remaining cases are \(\alpha=(1+\sqrt{d})/2\), which is integral only if \(d\equiv 1 \bmod 4\). 
If \(d\leq-9\) the real part of \(\alpha^2\) is \((1+d)/4\leq -2\) under either embedding, so \(\alpha\) satisfies the conditions to Lemma~\ref{lem:help_deg_2_dec}.ii.
If \(-9<d<9\) we have \(d\in\{-7,-3,5\}\) and thus \(\alpha=(1+\sqrt{d})/2\) is listed. For \(9\leq d \leq 25\) we may apply Lemma~\ref{lem:help_deg_2_dec}.iii. 
The remaining case is \(25< d\), where we have that \(\sigma(\alpha^2)=[(1\pm\sqrt{d})/2]^2 \geq (\sqrt{d}-1)^2/4\geq 2\) for all \(\sigma\in\X(\Q(\sqrt{d}))\), so Lemma~\ref{lem:help_deg_2_dec}.i applies.
Hence \(\alpha\) is either listed or not indecomposable.
\end{proof}

Interesting to note is that all non-trivial decompositions of \(\alpha\in\overline{\Z}\) of degree 2 over \(\Q\) that are produced in Theorem~\ref{thm:deg2_indec} live in a field extension of degree at most 2 over \(\Q(\alpha)\).

\section{Geometry of numbers}
\label{sec:begin-capacity-large}

In this section we gather some known results about the geometry of numbers.

\begin{definition}
Let \(K\) be a number field. We write \(K_\R=\R\tensor_\Q K\) and \(K_\C=\C\tensor_\Q K\).
\end{definition}

Recall for a number field \(K\) the definition of \(\X(K)\), the set of ring homomorphisms from \(K\) to \(\C\).

\begin{lemma}\label{lem:KR_def}
We have an isomorphism of \(\C\)-algebras \(\Phi_K: K_\C\to \C^{\X(K)}\) given by
\[ \Phi_K( z\tensor\alpha ) = ( z \cdot \sigma(\alpha)  )_{\sigma\in \X(K)}. \]
We have a natural inclusion \(K_\R\to K_\C\to \C^{\X(K)}\), and its image is given by the subspace of elements invariant under the involution \((x_\sigma)_{\sigma}\mapsto (\overline{x_{\overline{\sigma}}})_\sigma\). 
This inclusion induces an isomorphism of \(\R\)-algebras \(K_\R\cong \R^r\times\C^s\) for integers \(r,s\geq 0\) such that \(r=\#\{\sigma\in \X(K)\,|\, \sigma[K] \subseteq\R\}\) and \(r+2s=[K:\Q]\).
\qed
\end{lemma}

\begin{definition}\label{def:KR_inner_product}
We equip \(K_\C\) with the inner product induced by the standard hermitian inner product on \(\C^{\X(K)}\) and \(K_\R\) with its restriction, turning \(K_\R\) into a real inner product space. Since \(K_\R\) is an inner product space we have an induced measure on \(K_\R\) we denote \(\vol\). 
\end{definition}

\begin{remark}
For a number field \(K\) and \(\alpha\in K\) we have \(q(\alpha)=\frac{1}{[K:\Q]}\|\Phi_K(\alpha)\|^2\). 
Note that the norm on \(K_\R\) is not the `standard' norm on \(\R^r\times \C^s\). In terms of the latter vector space it is given by 
\[(x_1,\dotsc, x_r,z_1,\dotsc,z_s ) \mapsto \sqrt{ |x_1|^2 +\dotsm + |x_r|^2 + 2|z_1|^2 +\dotsm + 2|z_s|^2  }.\]
\end{remark}

\begin{theorem}[Proposition~4.26 in \citep{milneANT}]\label{thm:det_lattice}
Let \(R\) be an order in a number field \(K\). Then \(\Phi_K[R]\) is a full rank lattice in \(K_\R\) with determinant \(|\Delta(R)|^{1/2}\), where \(\Delta(R)\) is the discriminant of \(R\). \qed
\end{theorem}

\begin{definition}\label{def:bounded_degree_poly}
For a commutative ring \(R\) and \(d\in\Z_{\geq 0}\) we write \(R[X]_d=\{ f \in R[X] \,|\, \deg(f) < d \}\).
\end{definition}

\begin{lemma}\label{lem:product_polyring}
For a number field \(K\) we have an isomorphism of real vector spaces \(K_\R[X]_d\cong (\R[X]_d)^r \times (\C[X]_d)^s\) for all \(d\in\Z_{\geq 0}\) induced by the isomorphism \(K_\R\cong \R^r\times \C^s\) of Lemma~\ref{lem:KR_def}. \qed
\end{lemma}

\begin{theorem}[Minkowski, Theorem~4.19 in \citep{milneANT}]\label{thm:minkowski}
Let \(n\in\Z_{\geq 0}\), let \(\Lambda\subseteq\R^n\) be a full rank lattice and let \(S\subseteq \R^n\) be a symmetric convex body.
If \(\vol(S) > 2^n\det(\Lambda)\), then there exists a non-zero element in \(\Lambda\cap S\). \qed
\end{theorem}

\section{Szeg\H{o} capacity theory}\label{sec:szego}

In this section we will give a proof of Theorem~\ref{thm:large}, which is due to T. Chinburg.

\begin{definition}
We define for \(\overline{\Q}\) and in turn any subfield of \(\overline{\Q}\) the {\em max-norm}
\[|\alpha|_\infty = \max_{\sigma\in \X(\overline{\Q})} |\sigma(\alpha)|.\]
\end{definition}
\noindent One easily shows that the max-norm satisfies the following.
\begin{lemma}\label{lem:ineq_q_max}
For all \(\alpha\in\overline{\Q}\) we have \(q(\alpha)\leq|\alpha|_\infty^2\) and for all \(n\in\Z_{\geq 1}\) also \(|\alpha^n|_\infty = |\alpha|^n_\infty\). \qed
\end{lemma}
In this section we will identify a number field \(K\) with its image in \(K_\R\). 
We extend \(\sigma\in \X(K)\) to \(K_\C\) by \(z\tensor\alpha\mapsto z\cdot\sigma(\alpha)\) or equivalently projecting \(\C^{\X(K)}\) onto the component at \(\sigma\).
We similarly equip \(K_\C\) with a max-norm
\[\Kab{x}=\max_{\sigma\in \X(K) } |\sigma(x)|.\]
and \(K_\R\) with the induced norm. 
Note that the max-norm \(|\,\cdot\,|_\infty\) of \(K\) as subfield of \(\overline{\Q}\) and the max-norm \(\Kab{\,\cdot\,}\) of \(K\) as subfield of \(K_\R\) agree.

\begin{definition}
Let \(X\) be a metric space with metric \(d\) and let \(S\subseteq X\) be a subset.
A {\em rounding function from \(X\) to \(S\)} is a map \(\lfloor\cdot\rceil:X\to S\) for which there exists some constant \(c\in\R_{\geq 0}\) such that for all \(x\in X\) we have \(d(x,\lfloor x \rceil)\leq c\).
We call such a \(c\) an {\em error constant} for \(\lfloor\cdot\rceil\).
\end{definition}

For \(\Z\) in \(\Q\) with the metric induced by the usual absolute value we may round to a nearest integer, giving a rounding function with error constant \(1/2\).
For a naive rounding map for an arbitrary order \(R\) with basis \((\alpha_i)_i\) of a number field \(K\) with metric induced by \(q\) we may simply send \(\sum_{i} x_i \alpha_i\in K\) with \(x_i\in\Q\) to \(\sum_i \lfloor x_i \rceil \alpha_i \in R\).
An error constant for this rounding function is for example \(\frac{1}{2}\sum_{i=1}^n |\alpha_i|_\infty\).
However, other rounding functions exist, for example Babai's nearest plane algorithm, which try to achieve beter error constants.

\begin{lemma}\label{lem:polynomial_rounding}
Let \(K\) be a \(\Z\)-algebra with a norm \(|\cdot|\) and let \(R\) be a subring of \(K\). 
Let \(\alpha\in K\) and suppose there is a rounding function \(K\to R\) with respect to this norm and with error constant \(c\).
Consider \(K[X]\) and let \(Y=X-\alpha\).
Then \(K[X]=K[Y]\) equipped with the max-norm \(\| \sum_i f_i Y^i \| \mapsto \max_i |f_i|\) has a rounding function \(\lfloor\cdot\rceil:K[X]\to R[X]\) with error constant \(c\) and the following additional property:
If for some \(n\geq 0\) and \(f=\sum_{i} f_i X^i \in K[X]\) we have \(f_k\in R\) for all \(k\geq n\), then \(f-\lfloor f \rceil\) has degree less than \(n\).
\end{lemma}
\begin{proof}
Let \(\lfloor\cdot\rceil_1:K\to R\) be the rounding function with error constant \(c\).
We may assume without loss of generality that \(\lfloor \cdot \rceil_1\) restricted to \(R\) is the identity.
With the notation as in Definition~\ref{def:bounded_degree_poly} we identify \(K\) with \(K[X]_1\) and \(R\) with \(R[X]_1\) and inductively define for \(n\in\Z_{\geq 1}\) a rounding function \(\lfloor \cdot\rceil_n : K[X]_n \to R[X]_n \).
For \(f\in K[Y]_{n+1}\) write \(f= aY^{n} + g\) with \(a\in K\) and \(g\in K[Y]_{n}\).
Now define \(\lfloor f \rceil_{n+1} = \lfloor a \rceil_1 X^n + \lfloor g-\lfloor a\rceil_1(X^n-Y^n) \rceil_n\).
Note that this is well-defined.
Suppose \(\lfloor \cdot \rceil_{n}\) has error constant \(c\). Then 
\[\|f-\lfloor f\rceil_{n+1}\|=\max\Big\{ \big|a-\lfloor a\rceil_1\big|,\, \big\| (g-\lfloor a \rceil_1(X^n-Y^n))-\lfloor g-\lfloor a \rceil_1(X^n-Y^n)\rceil_n \big\|\Big\} \leq c,\]
so \(\lfloor \cdot \rceil_{n+1}\) has error constant \(c\).
Note that \(\lfloor f \rceil_{m} = \lfloor f \rceil_n\) for \(m\geq n\) and \(f\in K[Y]_n\).
Hence this defines a rounding function \(K[X]\to R[X]\) with error constant \(c\) as required.
With \(f\) as before of degree \(n\), if \(a\in R\) then \(a-\lfloor a \rceil_1 =0\) and \(f-\lfloor f \rceil\) has degree less than \(n\).
One then inductively proves the final property of \(\lfloor\cdot\rceil\).
\end{proof}

\begin{lemma}\label{lem:help_binom}
Let \(d,k\in\Z_{>0}\). Then there exist \(m\in\Z_{>0}\) such that \(d \mid \binom{am}{k}\) for all \(a\in\Z_{>0}\). 
\end{lemma}
\begin{proof}
Note that \(\binom{X}{k}=X(X-1)\dotsm (X-k+1)/ (k!) \in \Q[X]\) is a polynomial for fixed \(k\). For \(X\) a multiple of \(d\cdot (k!)\) we have that \(d\mid \binom{X}{k}\).
\end{proof}

The following is a special case of a theorem of Szeg\H{o}, adapted from \cite{capacity}.

\begin{theorem}[Szeg\H{o}]\label{thm:szego_poly}
Let \(R\) be an order of a number field \(K\) and let \(r>1\).
Then for each \(\alpha\in K\) there exists a monic non-constant \(g\in R[X]\) such that for all \(z\in K_\C\) satisfying \(\Kab{g(z)}< r\) we have \(\Kab{z-\alpha}<r\).
\end{theorem}

\begin{proof}
We will write \(Y=X-\alpha\) and \(Y^n=\sum_{i=0}^n f_{n,i} X^i\) for \(n\in\Z_{\geq 0}\).
Let \(d_0=1\), and for \(k\in\Z_{>0}\) let \(d_k\in\Z_{>0}\) be such that \(d_k\alpha^k\in R\). 
For \(k\in\Z_{\geq 0}\) let \(m_k\in\Z_{>0}\) be such that for all \(a\in\Z\) we have \(d_k \mid \binom{am_k}{k}\), which exists by Lemma~\ref{lem:help_binom}. 
Then \(f_{n,n-k}=\pm\binom{n}{k}\alpha^k \in R\) for all \(k\in\Z_{\geq 0}\) and positive multiples \(n\) of \(m_k\) and such that \(k \leq n\).
Write \(M_b = \text{lcm}(m_1,\dotsc,m_b)\).
Then for all \(b\) and \(n\) such that \(M_b \mid n\) we have \(f_{n,n-k}\in R\) for all \(0\leq k\leq b\).

Let \(\lfloor \cdot \rceil_1:K\to R\) be a rounding function, which exists because \(R\) is a full rank lattice in \(K_\R\) by Theorem~\ref{thm:det_lattice}. 
Let \(\lfloor \cdot\rceil:K[Y]\to R[X]\) be the rounding function given by Lemma~\ref{lem:polynomial_rounding} and let \(c\) be its error constant.
For \(n\geq 0\) let \(g_n=\lfloor Y^n \rceil\) and \(e_n = \sum_{i=0}^n e_{n,i} Y^i = Y^n - g_n\).
Moreover, we may assume that if \(M_b \mid n\) for some \(0\leq b \leq n\) then \(\deg(e_n)<n-b\), as \(f_{n,i}\in R\) for \(i\geq n-b\).
Since \(r>1\) we may choose \(b\in\Z_{> 0}\) such that \(2cr^{-b}\leq r-1\). 
Next we may choose \(a\in\Z_{>0}\) such that \(n=aM_b\) satisfies \(n>b\) and \(r^{n-1}\geq 2\). 
We claim \(g=g_n\) satisfies the requirements. 

Clearly \(g_n\) is monic by the bound on the degree of \(e_n\) and non-constant as \(n>0\).
Let \(\X(K)\) act on \(K[X]\) coefficient-wise and fix \(\sigma\in\X(K)\). 
Let \(z\in \C\) such that \(|\sigma(Y)(z)|\geq r\) and write \(s=|\sigma(Y)(z)|=|z-\sigma(\alpha)|\).
Then
\[\left| \frac{\sigma(e_{n})(z)}{\sigma(Y^n)(z)} \right| \leq s^{-n} \sum_{i=0}^{n-b-1} |\sigma(e_{n,i})|\cdot s^i \leq s^{-n}\sum_{i=0}^{n-b-1} c s^i \leq \frac{cs^{-b}}{s-1} \leq \frac{cr^{-b}}{r-1}\leq \frac{1}{2}.\]
It follows that
\[ \left|\frac{\sigma(g_n)(z)}{\sigma(Y^n)(z)}\right| = \left| 1-\frac{\sigma(e_n)(z)}{\sigma(Y^n)(z)} \right| \geq \frac{1}{2} \quad\text{and}\quad |\sigma(g_n)(z)|\geq\frac{|\sigma(Y^n)(z)|}{2} \geq \frac{r^n}{2} \geq r. \]
Thus, if \(|\sigma(g_n)(z)|<r\), then \(|z-\sigma(\alpha)|<r\). Taking the maximum over all \(\sigma\in\X(K)\) proves the theorem.
\end{proof}

\begin{theorem}[Szeg\H{o}]\label{thm:szego}
Let \(\alpha\in\overline{\Q}\) and \(r\in\R_{>1}\).
Then there exist infinitely many \(\beta\in\overline{\Z}\) such that \(|\alpha-\beta|_\infty< r\).
\end{theorem}
\begin{proof}
Consider \(K=\Q(\alpha)\) and let \(R\subseteq K\) be some order of \(K\).
By Theorem~\ref{thm:szego_poly} there exists some monic non-constant \(g\in R[X]\) such that for all \(z\in K_\C\) satisfying \(\Kab{g(z)} < r\) we have \(\Kab{z-\alpha} < r\).
Now \(g^n-1\in R[X]\) is monic and non-constant for all \(n\in\Z_{\geq1}\), so \(S=\{\beta\in\overline{\Z}\,|\, (\exists\, n\in\Z_{\geq 1})\ g(\beta)^n = 1 \}\) is infinite.
Let \(\beta\in S\) and \(L=K(\beta)\). 
It now suffices to show that \(|\alpha-\beta|_\infty\leq r^2\).

Write \(\textup{P}=\prod_{\sigma\in\X(K)} \X_\sigma(L)\).
Let \(\rho=(\rho_\sigma)_\sigma\in\textup{P}\) and consider \(\beta_\rho=(\rho_\sigma(\beta))_\sigma\in\C^{\X(K)}=K_\C\).
Then
\[\Kab{g(\beta_\rho)}^n = \Kab{g^n(\beta_\rho)} = \max_{\sigma\in\X(K)} | \rho_\sigma(g^n(\beta)) | = \max_{\sigma\in\X(K)}|\rho_\sigma(1)| = 1,  \]
so \(\Kab{g(\beta_\rho)}=1 < r\).
Thus by definition of \(g\) we have \(\Kab{\alpha-\beta_\rho} < r\).
As this holds for all \(\rho\in\textup{P}\) we have \(|\alpha-\beta|_\infty=\max_{\rho\in\textup{P}}\, \Kab{\alpha-\beta_\rho} < r\), as was to be shown.
\end{proof}

\noindent\textbf{Theorem~\ref{thm:large}. }{\em If \(\gamma\in \overline{\Z}\) satisfies \(4<q(\gamma) \), then \(\gamma\) has infinitely many decompositions in \(\overline{\Z}\).}

\begin{proof}
Let \(\alpha=\gamma/2\).
By Theorem~\ref{thm:szego} there are infinitely many \(\beta\in\overline{\Z}\) such that \(|\alpha-\beta|_\infty \leq \sqrt{q(\alpha)}\) as \(q(\alpha)=q(\gamma)/4>1\).
By Lemma~\ref{lem:ineq_q_max} each such \(\beta\) satisfies \(q(\alpha-\beta)\leq q(\alpha)\) and thus \((\beta,\gamma-\beta)\) is a decomposition of \(\gamma\) by Lemma~\ref{lem:eq_dec_def}.
\end{proof}

It follows from this theorem, as we will show later in the form of Proposition~\ref{prop:counting_candidates}, that there are only finitely many indecomposables in \(\overline{\Z}\) of a given degree.
Recall the definition of \(r\) from Definition~\ref{def:radius}.

\begin{proposition}
The Hilbert lattice \(\overline{\Z}\) has a covering radius \(2^{-1/4}\leq r(\overline{\Z})\leq 1\).
\end{proposition}
\begin{proof}
By Theorem~\ref{thm:largest_indec} we have \(2^{3/2}\leq \sup\{q(\alpha)\,|\,\alpha\in\textup{indec}(\overline{\Z})\}\) and consequently we get the lower bound \(2^{-1/4}\leq\sup\{\|\alpha/2\|\,|\,\alpha\in\textup{indec}(\overline{\Z})\}\).
For any \(\alpha\in\textup{indec}(\overline\Z)\) we have \(\|\alpha/2\|\leq\|\alpha/2-x\|\) for all \(x\in\overline{\Z}\) by Lemma~\ref{lem:eq_dec_def}, and thus \(\alpha/2\in\vorc(\overline{\Z})\) by Corollary~\ref{cor:vorc}. 
Therefore \(2^{-1/4}\leq r(\overline{\Z})\) by Proposition~\ref{prop:vor_incl}.
For all \(r>1\) and \(\alpha\in\overline{\Q}\) there exist \(\beta\in\overline{\Z}\) such that \(\|\alpha-\beta\|\leq r\) by Theorem~\ref{thm:szego}.
Taking the limit of \(r\) down to \(1\) and noting that \(\overline{\Q}=\Q\cdot\overline{\Z}\) is dense in the Hilbert space of \(\overline{\Z}\) proves the proposition.
\end{proof}

\section{Computing decompositions of algebraic integers of large square-norm}

The proof of Theorem~\ref{thm:szego} and in turn Theorem~\ref{thm:large} is constructive and rather explicit as well. 
It is not too difficult to design an algorithm that given \(n\in\Z_{\geq 0}\), an order \(R\) and an \(\alpha\in R\) such that \(q(\alpha)>4\), computes \(n\) distinct non-trivial decompositions of \(\alpha\).
However, we will have to formalize how we encode the input and output.
Moreover, since \(q\) maps to \(\R\) we also have to deal with non-exact arithmetic.

\begin{convention}
We presuppose that we have decided on a way of encoding {\em integers} \(n\in\Z\) and {\em pairs} \((a,b)\) such that the length \(\textup{len}(n)\) of the encoding of \(n\) grows at a rate \(O(\log |n|)\) and \(\textup{len}((a,b))=O(\textup{len}(a)+\textup{len}(b))\). We then encode a {\em rational} as a pair of integers, a numerator and denominator.
We encode a finite sequence of integers, a {\em list}, as nested pairs.
An order is a commutative ring which is free and finitely generated as \(\Z\)-module.
We encode an {\em order} as a non-negative integer \(n\), its rank, together with integers \((a_{ijk})_{1\leq i,j,k\leq n}\) such that \(e_i\cdot e_j = \sum_{k=1}^n a_{ijk}e_k\), where \(e_i\) is the \(i\)-th standard basis vector of \(\Z^n\).
We represent a {\em number field} \(K\) with an order \(R\subseteq K\) such that \(\Q R = K\).
A {\em homomorphism of modules} out of a free \(\Z\) or \(\Q\)-module is encoded by a list of images of the basis elements. 
Given an encoding for a commutative ring \(R\) we encode a {\em polynomial} in \(R[X]\) as an integer \(d\), its degree (with \(\deg(0)=-1\)), and a list of its \(d+1\) coefficients, i.e.\ a dense encoding.
\end{convention}

\begin{theorem}\label{alg:all}
There exist {\em polynomial-time} algorithms for the following problems: Take as input a number field \(K\) and
\begin{enumerate}[nosep,label=\textup{(\roman*)}]
\item an irreducible polynomial \(f\in K[X]\), and compute the number field \(K[X]/(f)\) and the natural inclusion \(K\to K[X]/(f)\);
\item an element \(\alpha\in K\), and compute the minimal polynomial of \(\alpha\) over \(\Q\);
\item a non-zero polynomial \(f\in K[X]\), and compute the monic irreducible factors of \(f\) over \(K\).
\end{enumerate}
There also exists an algorithm that takes as input a number field \(K\) and
\begin{enumerate}[resume,nosep,label=\textup{(\roman*)}]
\item a polynomial \(f\in K[X]\), and computes a splitting field of \(f\) over \(K\).
\end{enumerate}
\end{theorem}
\begin{proof}
Part (i) and (ii) are basic linear algebra. Part (iii) is a result due to A. K. Lenstra \citep{factoring}. Part (iv) is a repeated application of (iii) and (i).
\end{proof}

\begin{definition}
An {\em approximation} of a real number \(x\in\R\) is an algorithm that takes as input a positive integer \(n\) written in unary and computes a rational number \(x(n)\) such that \(|x-x(n)|\leq 2^{-n}\).
We define the analogous notion for complex numbers as well by considering the real and imaginary parts separately.
A number for which there exists an approximation is called {\em computable}.
An {\em algebraic real number} \(x\) is encoded as a minimal polynomial \(f\) over \(\Q\) of \(x\) together with an approximation of \(x\).  
\end{definition}

Since there are only countably many algorithms and uncountably many real numbers, not every real number is computable.

\begin{lemma}[For example \cite{complex_roots}]\label{lem:compute_roots}
There is an algorithm that takes as input a non-zero polynomial \(f\in\Q[X]\) and computes approximations for all complex roots of \(f\) with multiplicities. \qed
\end{lemma}

As a direct consequence of Lemma~\ref{lem:compute_roots} and Theorem~\ref{alg:all}.ii we have the following proposition.

\begin{proposition}\label{prop:comp_q}
There is an algorithm that takes as input an algebraic number \(\alpha\in\overline{\Q}\) represented by its minimal polynomial and computes an approximation of \(q(\alpha)\), \(N(\alpha)\) and \(|\alpha|_{\infty}\). \qed
\end{proposition}

\begin{proposition}\label{prop:decide_q}
There is an algorithm that takes as input an algebraic real number \(x\) and algebraic numbers \(\alpha,\beta\in\overline{\Q}\) and decides whether \(\langle\alpha,\beta\rangle=x\).
\end{proposition}
\begin{proof}
Compute the normal closure \(K\) of \(\Q(\alpha,\beta)\) using Theorem~\ref{alg:all}.iv and compute the set \(\textup{A}\) of conjugates of \(\alpha\) in \(K\) using Theorem~\ref{alg:all}.iii and similarly for \(\beta\).  
Choose an embedding \(\rho\in\X(K)\).
For all \(\sigma\in\X(K)\) there exist \(\alpha_\sigma\in \textup{A}\) and  \(\beta_\sigma\in \textup{B}\) such that \(\sigma(\alpha)=\rho(\alpha_\sigma)\) and \(\overline{\sigma(\beta)}=\rho(\beta_\sigma)\) because \(K\) is Galois, and we may actually compute them since there are only finitely many candidates.
Factor the minimal polynomial \(f\) of \(x\) over \(K\) using Theorem~\ref{alg:all}.
Find a \(y\) among the roots of \(f\) in \(K\) such that \(\rho(y)=x\).
If it does not exist, then \(\langle\alpha,\beta\rangle\neq x\).
Otherwise we may decide \(\tfrac{1}{[K:\Q]}\sum_{\sigma\in\X(K)} \alpha_\sigma \beta_\sigma = y\) using exact arithmetic in \(K\).
\end{proof}

\begin{corollary}\label{cor:decide_decomp}
There is an algorithm that takes as input algebraic integers \(\alpha\), given by a minimal polynomial over \(\Q\), and \(\beta\), given by a minimal polynomial over \(\Q(\alpha)\), and decides whether \((\beta,\alpha-\beta)\in\textup{dec}(\alpha)\).
\end{corollary}
\begin{proof}
Using Proposition~\ref{prop:decide_q} we can decide whether \(\langle\beta,\alpha-\beta\rangle=0\). 
If not, then after finitely many steps approximating \(\langle\beta,\alpha-\beta\rangle\) using Proposition~\ref{prop:comp_q} we may decide whether \(\langle\beta,\alpha-\beta\rangle\geq 0\).
\end{proof}

\begin{lemma}\label{lem:computable_rounding_map}
For every order \(R\) in a number field \(K\) and \(\alpha\in K\) there exists a rounding function \(\lfloor \cdot \rceil: K[X-\alpha] \to R[X]\) as in Lemma~\ref{lem:polynomial_rounding}, which is computable and has a computable error constant.
\end{lemma}
\begin{proof}
Note that the naive rounding function \(\lfloor\cdot\rceil_1:K\to R\) with respect to a basis \(\alpha_1,\dotsc,\alpha_n\) of \(R\) is computable and has error constant \(\tfrac{1}{2}\sum_{i=1}^n |\alpha_i|_\infty\), which is computable by Proposition~\ref{prop:comp_q}. Finally note that the error function from Lemma~\ref{lem:polynomial_rounding} is trivially computable given that \(\lfloor\cdot\rceil_1\) is computable, and it has the same computable error constant.
\end{proof}

\begin{theorem}\label{alg:szego}
There is an algorithm that takes as input an order \(R\) in a number field \(K\), a rational number \(r>1\) and an \(\alpha\in K\), and computes a monic non-constant polynomial \(g\in R[X]\) such that for all \(z\in K_\C\) satisfying \(\Kab{g(z)}< r\) we have \(\Kab{z-\alpha}< r\).
\end{theorem}
\begin{proof}
Let \(\lfloor \cdot \rceil:K[X-\alpha]\to R[X]\) be the rounding map from Lemma~\ref{lem:computable_rounding_map}.
Then
\begin{enumerate}[nosep]
\item Compute some rational error constant \(c\) for \(\lfloor\cdot\rceil\);
\item Compute some \(d\in\Z_{\geq1}\) such that \(d\alpha\in R\) by clearing denominators;
\item Compute some \(b\in\Z_{\geq1}\) such that \(2cr^{-b}\leq r-1\);
\item Compute an \(n\in\Z\cdot d^b\cdot b!\) such that \(n>b\) and \(r^{n-1}\geq 2\);
\item Compute and return \(\lfloor (X-\alpha)^n\rceil\).
\end{enumerate}
The algorithm is correct by the proof of Theorem~\ref{thm:szego_poly}.
\end{proof}

\begin{theorem}
There is an algorithm that takes as input an \(n\in\Z_{\geq 0}\) and element \(\alpha\in\overline{\Z}\) given by its minimal polynomial, and decides whether \(q(\alpha)>4\) and if so computes \(n\) non-trivial \((\beta,\gamma)\in\textup{dec}(\alpha)\) each represented by the minimal polynomial of \(\beta\) over \(\Z[\alpha]\).  
\end{theorem}
\begin{proof}
To decide \(q(\alpha)>4\), first decide whether \(q(\alpha)=4\) using Proposition~\ref{prop:decide_q}. If not, approximating \(q(\alpha)\) will yield either \(q(\alpha)<4\) or \(q(\alpha)>4\) after finitely many steps.
We then follow the proof of Theorem~\ref{thm:szego} and Theorem~\ref{thm:large}.

Since \(q(\alpha)>4\) by assumption we may compute using Proposition~\ref{prop:comp_q} a rational \(r\) such that \(1<r\leq \sqrt{q(\alpha/2)}\).
Compute using Theorem~\ref{alg:all}.i the number field \(K=\Q(\alpha)\) and its order \(R=\Z[\alpha]\).
Compute using Theorem~\ref{alg:szego} a monic non-constant \(g\in R[X]\) such that for all \(z\in K_\C\) satisfying \(\Kab{g(z)}< r\) we have \(\Kab{z-\alpha}<r\). Compute the monic irreducible factors of \(g^k-1\) over \(K\) using Theorem~\ref{alg:all}.iii for \(k\) ranging from \(1\) to infinity until you have found among all of them \(n\) distinct integral polynomials that are not equal to \(X\) or \(X-\alpha\). 
By the proof of Theorem~\ref{thm:szego} this process terminates and produces non-trivial decompositions.
\end{proof}

\section{Bounds on indecomposable algebraic integers}\label{sec:analysis}

In this section we will prove an effective upper bound on the total number of indecomposable algebraic integers of a given degree.
In particular, we will show that this number is finite.
We do this by constructing a complete list of candidates for indecomposability among all algebraic integers of given degree.
We alo give a lower bound on the number of indecomposables.

\begin{proposition}\label{prop:classifying_candidates}
Suppose \(\alpha\in\textup{indec}(\overline{\Z})\) has minimal polynomial \(f=\sum_{k=0}^n f_{n-k} X^k\in\Z[X]\) of degree \(n\in\Z_{\geq 1}\). Then \(|f_k|\leq \binom{n}{k} 2^{k}\) for all \(0\leq k\leq n\).
\end{proposition}
\begin{proof}
Let \(\alpha_1,\dotsc,\alpha_n\in\C^\times\) be the roots of \(f\). We have Maclaurin's inequalities (Theorem~11.2 in \citep{Cvetkovski2012})
\[ s_1 \geq s_2^{1/2} \geq s_3^{1/3} \geq \dotsm \geq s_n^{1/n}, \quad\text{where }  s_k = \binom{n}{k}^{-1}\cdot\sum_{\substack{I \subseteq\{1,\dotsc,n\}\\|I|=k} } \prod_{i\in I} |\alpha_i|. \]
By Theorem~\ref{thm:large} we have that \(q(\alpha)\leq 4\).
Then by Lemma~\ref{lem:holder} we have 
\[s_1 = \frac{1}{n}\sum_i |\alpha_i| \leq \Big( \frac{1}{n}\sum_i |\alpha_i|^2 \Big)^{1/2} = \sqrt{q(\alpha)} \leq 2. \]
Then \(|f_k| \leq \binom{n}{k} s_k \leq \binom{n}{k} s_1^k\leq \binom{n}{k}2^k\) for all \(k\), as was to be shown.
\end{proof}

\begin{corollary}\label{cor:classifying_canditates}
Suppose \(\alpha\in\textup{indec}(\overline{\Z})\) has degree at most \(m\). Then there exists a monic polynomial \(g=\sum_{k=0}^m g_{m-k} X^k \in\Z[X]\) of degree \(m\) such that \(g(\alpha)=0\) and \(|g_k|\leq \binom{m}{k} 2^k\) for all \(0\leq k\leq m\).
\end{corollary}
\begin{proof}
Let \(f\) as in Proposition~\ref{prop:classifying_candidates} and \(g=X^{m-n}\cdot f\). Then \(|g_k|=|f_k|\leq \binom{n}{k}2^k\leq \binom{m}{k}2^k\).
\end{proof}

\begin{proposition}\label{prop:stirlings}
Considered as functions of \(n\in\Z_{\geq 1}\), the following hold:
\begin{align*}
\log( n! ) &= n \log n - n + O(\log n); \tag{i} \\
\log\Big( \prod_{k=1}^n k^k \Big) &= \tfrac{1}{2}n^2 \log n-\tfrac{1}{4}n^2+O(n\log n); \tag{ii} \\
\log\Big( \prod_{k=0}^n (k!)  \Big) &=  \tfrac{1}{2} n^2 \log n - \tfrac{3}{4} n^2 + O(n\log n); \tag{iii} \\
\log\Big( \prod_{k=0}^n \binom{n}{k} \Big) &= \tfrac{1}{2} n^2 + O(n\log n). \tag{iv}
\end{align*}
\end{proposition}
\begin{proof}
(i) This is Stirling's approximation, which is classical.

(ii) Note that \(f(x)=x\log x\) is an increasing function on \(\R_{\geq 1}\). Hence
\[ \log\Big( \prod_{k=1}^n k^k \Big) = \sum_{k=1}^n f(k) \leq \int_1^{n+1} f(x) \dif x = \Big[ \tfrac{1}{2} x^2 \log(x) - \tfrac{1}{4}x^2  \Big]_{x=1}^{n+1} = \tfrac{1}{2} n^2 \log(n) - \tfrac{1}{4}n^2 + O(n\log(n)). \]
We analogously get the same estimate for a lower bound by considering \(\int_1^n f(x) \dif x \).

(iii) From (i) and (ii) we get
\begin{align*}
\log\Big( \prod_{k=0}^n (k!)  \Big) = \sum_{k=1}^n\Big( k\log(k) -k + O(\log(k))  \Big) = \big(\tfrac{1}{2} n^2 \log(n) - \tfrac{1}{4}n^2\big) - \tfrac{1}{2}n^2 + O(n\log(n)).
\end{align*}

(iv) We first rewrite the binomials in terms of factorials and then apply (i) and (iii), so that
\begin{align*} 
\log\Big( \prod_{k=0}^n \binom{n}{k} \Big) 
&= \log\bigg(\frac{(n!)^n}{\big(\prod_{k=0}^n (k!) \big)^2} \bigg) 
= n \log(n!) - 2 \log\Big( \prod_{k=0}^n (k!)  \Big) \\
&= (n^2 \log(n) - n^2) - 2( \tfrac{1}{2}n^2 \log n - \tfrac{3}{4}n^2 ) + O(n\log n) = \tfrac{1}{2}n^2 + O(n\log n),
\end{align*}
as was to be shown.
\end{proof}

\begin{proposition}\label{prop:counting_candidates}
Let \(n\in\Z_{\geq 1}\). There are at most 
\[ n\prod_{k=1}^n \Big( 2\binom{n}{k} 2^k + 1 \Big) = \exp\Big( \tfrac{\log(2)+1}{2}\, n^2 +O(n\log(n))\Big)\]
indecomposable elements in \(\overline{\Z}\) of degree up to \(n\).
\end{proposition}
\begin{proof}
By Corollary~\ref{cor:classifying_canditates} every indecomposable of degree at most \(n\) is the root of a monic polynomial \(f=\sum_{k=0}^n f_{n-k} X^k\) such that  \(|f_k|\leq \binom{n}{k} 2^{k}\) for all \(0\leq k\leq n\). Hence every such polynomial corresponds to at most \(n\) indecomposables.
For every \(0<k\leq n\) there are \(2\binom{n}{k}2^k+1\) choices for \(f_k\), and \(f_0=1\), proving the first upper bound.
We may bound \(2\binom{n}{k}2^k+1\leq 3\binom{n}{k}2^k\), so that by Proposition~\ref{prop:stirlings}.iv we get
\[n\prod_{k=1}^n \Big( 2\binom{n}{k} 2^k + 1 \Big) \leq  n \cdot 3^n \cdot 2^{\binom{n+1}{2}}\cdot \prod_{k=0}^n \binom{n}{k} = \exp\Big( \tfrac{\log(2)+1}{2}\, n^2 +O(n\log(n))\Big),\]
as was to be shown.
\end{proof}

For \(f\in\Q[X]\) monic write \(q(f)\) for the average of the square length of the roots of \(f\) in \(\C\), such that for all \(\alpha\in\overline{\Z}\) with minimal polynomial \(f_\alpha\in\Q[X]\) we get \(q(\alpha)=q(f_\alpha)\).
Note that \(f=(X+2)^n\), although it is not irreducible, has \(q(f)=4\) and attains the bounds of Proposition~\ref{prop:classifying_candidates}.
However, that does not imply that Proposition~\ref{prop:counting_candidates} cannot be improved, as it is not clear that all combinations of coefficients occur for polynomials \(f\) with \(q(f)\leq 4\).
Some small degree numerical results might suggest improvements can be made.

\begin{center}
\begin{tabular}{r|rrrr}
degree & $1$ & $2$ & $3$ & $4$ \\ \hline
\# monic \(f\in\Z[X]\) satisfying the conclusion to Proposition~\ref{prop:classifying_candidates} & 5 & 81 & 5525 & 1786785 \\
\# monic \(f\in\Z[X]\) satisfying \(q(f)\leq 4\) & 5 & 49 & 989 & 48422 \\
\# \(\alpha\in\overline{\Z}\) satisfying \(q(\alpha)\leq 4\) & 5 & 39 & 739 & 40354 \\
\end{tabular}
\end{center}
\vspace{16pt}

We also have the following lower bound.

\begin{proposition}\label{prop:lower_bound}
Let \(n\in\Z_{\geq 1}\). There are at least 
\[ \exp\Big( \frac{\log 2}{4} n^2+O(n\log n) \Big)\]
indecomposable algebraic integers of degree \(n\).
\end{proposition}
\begin{proof}
Let \(n\in\Z_{\geq 1}\) and recall the definition of \(\Z[X]_n\) from Definition~\ref{def:bounded_degree_poly}. Consider the set
\[S_n=\left\{ f = \sum_{k=0}^{n-1} f_kX^k\in 2X\Z[X]_{n-1}+2 \,\middle|\, (\forall k)\ |f_k| (\sqrt{2})^{k}n \leq (\sqrt{2})^{n-1} \right\}.\]
For \(f\in S_n\) consider \(g=X^n-f\) and note that \(g\) is irreducible by Eisenstein's criterion.
Consider the ball \(D\subseteq\C\) of radius \(r=(\sqrt{2})^{1-1/n}<\sqrt{2}\) around \(0\).
For all \(z\) on the boundary of \(D\) we have
\[|f(z)|\leq \sum_{k=0}^{n-1} |f_k| |z|^k \leq \sum_{k=0}^{n-1} |f_k| (\sqrt{2})^k \stackrel{\textup{(i)}}{\leq} \sum_{k=0}^{n-1} \frac{(\sqrt{2})^{n-1}}{n} = (\sqrt{2})^{n-1} = |z|^n,\]
where (i) is strict for \(n\) sufficiently large due to \(|f_0|n=2n<(\sqrt{2})^{n-1}\).
Hence by Rouch\'e's Theorem the polynomials \(X^n\) and \(g\) have the same number of roots in \(D\).
It follows that all roots of \(g\) in \(\C\) have length less than \(\sqrt{2}\).
Thus \(q(\alpha)<2\) for all roots \(\alpha\in\overline{\Z}\) of \(g\), so \(\alpha\) is indecomposable by Proposition~\ref{prop:q_is_2}.

We conclude that for \(n\) sufficiently large there are at least \(n\cdot\# S_n\) indecomposable algebraic integers of degree \(n\), so it remains to prove a lower bound on \(\#S_n\).
Note that the coefficients of \(f\in S_n\) satisfy independent inequalities, so we may simply give a lower bound per coefficient.
Let \(B=n-3\log_2(n)-2\), which is positive for \(n\) sufficiently large.
For \(k > B\) we consider only \(f_k=0\) and get a lower bound of \(1\) for this coefficient.
For \(0<k\leq B\) we have
\[ 2\left\lfloor\frac{(\sqrt{2})^{n-k-1}}{2n}\right\rfloor+1 \geq 2\Big(\frac{(\sqrt{2})^{n-k-1}}{2n}-1\Big)+1 = \frac{(\sqrt{2})^{n-k-1}}{n}-1 = \textup{(ii)}\]
choices for \(f_k\).
Then for \(n\) sufficiently large we have
\[ \frac{n}{(\sqrt{2})^{n-k-2}} \leq \frac{n}{n^{3/2}} \leq \sqrt{2}-1, \quad\text{so that}\quad \textup{(ii)} \geq \frac{(\sqrt{2})^{n-k-2}}{n}.\]
Hence \(S_n\) contains, for \(n\) sufficiently large, at least
\[ \prod_{k=1}^B \frac{(\sqrt{2})^{n-k-2}}{n} = \exp\Big( \frac{\log 2}{2} \sum_{k=1}^B (n-k-2)   - B\log n \Big) = \exp\Big(\frac{\log 2}{4} n^2 + O(n\log n) \Big) \]
elements, from which the proposition follows.
\end{proof}

\begin{corollary}
Let \(n\in\Z_{\geq 1}\). There are at least 
\[ \exp\Big( \frac{\log 2}{4} n^2+O(n\log n) \Big)\]
indecomposable algebraic integers of degree up to \(n\). \qed
\end{corollary}

From the upper and lower bound we may now conclude the following. 
\vspace{.6em}

\noindent\textbf{Theorem~\ref{thm:counting_asymptotic}. }{\em There are least \(\exp( \frac{1}{4}(\log 2) n^2 + O(n\log n))\) and at most \(\exp(\frac{1}{2}(1+\log 2)n^2+O(n\log n))\) indecomposable algebraic integers of degree up to \(n\). \qed}

\label{sec:end-capacity-large}
\section{Fekete capacity theory}\label{sec:classical_fekete}
\label{sec:begin-capacity-small}

In this section we present a proof of a special case of Fekete's theorem using Minkowski's convex body theorem. 
Fekete's theorem can be thought of as a partial converse to Theorem~\ref{thm:szego} of Szeg\H{o}.
Although this does not give us a converse to Theorem~\ref{thm:large}, using similar techniques as in this section we will later prove Theorem~\ref{thm:small} in Section~\ref{sec:fekete}.
The goal of this section is to showcase the proof technique we will use to prove Theorem~\ref{thm:small} so that we may later improve clarity by brevity.
Recall for \(\alpha\in\overline{\Q}\) the definition of the norm \(|\alpha|_\infty=\max_{\sigma\in\X(\overline{\Q})}|\sigma(\alpha)|\) from Section~\ref{sec:szego}.

\begin{theorem}[Fekete]\label{thm:classical_fekete_sans_polynomial}
For each \(\alpha\in\overline{\Q}\) and \(0<r<1\) there exist only finitely many \(\beta\in\overline{\Z}\) such that \(|\beta-\alpha|_\infty\leq r\).
\end{theorem}

Just like for Szeg\H{o}'s theorem, it is possible to derive an algorithmic counterpart to Fekete's theorem.

Combining Theorem~\ref{thm:classical_fekete_sans_polynomial} and Theorem~\ref{thm:szego}, the point \(r=1\) is still a singularity.
For \(\alpha\in\overline{\Z}\) and \(r=1\) clearly all \(\beta\in\alpha+\mu_\infty\) satisfy \(|\beta-\alpha|_\infty\leq r\).
However, when \(\alpha\not\in\overline{\Z}\) we do not know what happens in general.
We start with a volume computation.

\begin{definition}\label{def:poly_inner_product}
Let \(A\) be an \(\R\)-algebra equipped with a real inner product. 
We equip \(A[Y]\) with an inner product
\[ \Big\langle \sum_{k=0}^\infty f_k Y^k, \sum_{k=0}^\infty g_k Y^k\Big\rangle = \sum_{k=0}^\infty \langle f_k, g_k \rangle, \]
which is the `standard' inner product when we naturally identify \(A[Y]\) with \(A^{(\Z_{\geq 0})}\).
For \(n\in\Z_{\geq 0}\) we equip \(A[Y]_n\), as defined in Definition~\ref{def:bounded_degree_poly}, with the restriction of this inner product.
\end{definition}

\begin{remark}
Obviously \(\R\) is an \(\R\)-algebra with a real inner product.
We identify \(\C\) with \(\R^2\) by choosing \(\R\)-basis \(\{1,\textup{i}\}\), and equip \(\C\) with the inner product induced by the natural inner product on \(\R^2\).
Note that this inner product in turn induces the standard norm of \(\C\).
For a number field \(K\) we remark that \(K_\R\) has a real inner product as in Definition~\ref{def:KR_inner_product}.
\end{remark}

\begin{lemma}\label{lem:linear_transformations}
Let \(A\) be an \(\R\)-algebra of dimension \(d<\infty\).
For all \(a,b\in\R\), \(c\in A\) and \(n\in\Z_{\geq 0}\) we have an \(\R\)-linear transformation \(\phi\) on \(A[X]_n\) given by \(f \mapsto b f(a(X-c))\) with \(\det \phi= ( a^{n(n-1)/2} \cdot b^n )^d\).
\end{lemma}
\begin{proof}
Note that \(\phi\) is trivially an \(\R\)-linear transformation.
Choose an \(\R\)-basis \(\{e_1,\dotsc,e_d\}\) for \(A\).
Writing \(\phi\) as a matrix with respect to the basis \(\{e_i X^j\,|\,1\leq i \leq d,\, 0\leq j < n\}\) for \(A[X]_n\) we note that \(\phi\) is a lower triangular matrix with diagonal entries \(b,ba,ba^2,\dotsc,ba^{n-1}\), each occurring with multiplicity \(d\). The determinant of \(\phi\) is then simply the product of the diagonal.
\end{proof}

\begin{lemma}\label{lem:classical_fekete_volume}
Let \(\F\) be either \(\R\) or \(\C\) and let \(r\in\R_{>0}\). 
For \(n\in\Z_{\geq0}\) consider
\[ S_n(r) = \{  f\in \F[Y]_n\,|\, (\forall\,z\in\C)\ |z| \leq r \Rightarrow |f(z)|\leq r \}. \]
Then as function of \(n\) we have \( \log\vol(S_n(r)) \geq - \tfrac{1}{2} n^2 \cdot [\F:\R]\cdot \log r  + O(n\log n)\).
\end{lemma}
\begin{proof}
Write \(S_n=S_n(1)\).
By applying the transformation \(f\mapsto rf(r^{-1}Y)\) to \(\F[Y]_n\) we bijectively map \(S_n\) to \(S_n(r)\).
From Lemma~\ref{lem:linear_transformations} it follows that \(\log\vol(S_n(r)) = -\tfrac{1}{2}n^2\cdot [\F:\R] \log(r) + \log\vol(S_n) + O(n\log n)\).
It remains to prove \(\log\vol S_n \geq O(n\log n)\).

First suppose \(\F=\R\).
Consider the set
\[ T_n = \Big\{ \sum_{k=0}^{n-1} f_k Y^k\in \R[Y]_n \,\Big|\, \sum_{k=0}^{n-1}|f_k| \leq 1 \Big\}. \]
Note that for all \(f\in T_n\) and \(z\in\C\) such that \(|z|\leq 1\) we have \(|f(z)|\leq \sum_{k=0}^{n-1} |f_k| \leq 1\), so \(f\in S_n\). 
Hence \(T_n\subseteq S_n\) and \(\vol(T_n) \leq \vol(S_n)\). 
With Proposition~\ref{prop:stirlings} we compute \(\log\vol(T_n)=\log(2^n/n!) = O(n\log n)\), from which the lemma follows for \(\F=\R\). 

For \(\F=\C\), note that we have an isometry \(\R[X]_n^2\to \C[X]_n\) given by \((f,g)\mapsto f+\textup{i}\cdot g\).
For \(f,g\in \tfrac{1}{2}T_n\) and \(z\in\C\) such that \(|z|\leq 1\) we have \(|f(z)+\textup{i}\cdot g(z)|\leq |f(z)| + |g(z)| \leq \tfrac{1}{2}+\tfrac{1}{2}=1\), so \(f+\textup{i}\cdot g\in S_n\). 
Hence \(\log\vol(S_n)\geq \log( \vol(\tfrac{1}{2}T_n)^2) = 2 \log\vol(T_n)-2n\log2 = O(n\log n)\), from which the lemma follows for \(\F=\C\).
\end{proof}

\begin{theorem}\label{thm:classical_fekete}
Let \(R\) be an order of a number field \(K\) and let \(0<r<1\).
For all \(\alpha\in K\) there exists a non-zero \(g\in R[X]\) such that for all \(z\in K_\R\), if \(\Kab{z-\alpha}\leq r\) then \(\Kab{g(z)}\leq r\).
\end{theorem}

\begin{proof}
Write \(d=[K:\Q]\).
Let \(n\in\Z_{\geq 0}\) and consider the lattice \(\Lambda_n = R[X]_n\) in the inner product space \(K_\R[Y]_n\), where \(Y=X-\alpha\).
Note that \(\dim_\R K_\R[Y]_n = dn\) and that \(\Lambda_n\) is a full-rank lattice in \(K_\R[Y]_n\) with \(\det(\Lambda_n)=|\Delta(R)|^{n/2}\) by Lemma~\ref{lem:linear_transformations} and Theorem~\ref{thm:det_lattice}.
Consider
\[ S_n = \big\{ f\in K_\R[Y]_n \,\big|\, (\forall\, \sigma\in\X(K))\, (\forall\,z\in\C)\ |z| \leq r \Rightarrow |\sigma(f)(z)|\leq r \big\} \]
and note that it is both symmetric and convex.
Moreover, it follows from Lemma~\ref{lem:classical_fekete_volume} that \(\log\vol(S_n)\geq-\tfrac{1}{2}n^2d\log r + O(n\log n)\). Hence
\[ \log\Big(\frac{\vol(S_n)}{2^{dn}\cdot\det(\Lambda_n)}\Big) \geq -\tfrac{1}{2}n^2d\log r + O(n\log n). \]
Because \(-\tfrac{1}{2}d\log r > 0\) there exists some \(n\) sufficiently large such that \(\vol(S_n) > 2^{dn} \det(\Lambda_n)\).
By Theorem~\ref{thm:minkowski} there then exists some non-zero \(g\in \Lambda_n \cap S_n\) which as polynomial in \(X\) satisfies the requirements.
\end{proof}

\begin{proof}[Proof of Theorem~\ref{thm:classical_fekete_sans_polynomial}]
Let \(K=\Q(\alpha)\) and let \(R\subseteq K\) be some order of \(K\).
Then by Theorem~\ref{thm:classical_fekete} there exists some non-zero \(g\in R[X]\) such that for all \(z\in K_\R\), if \(\Kab{z-\alpha}\leq r\) then \(\Kab{g(z)}\leq r\).
Suppose \(\beta\in\overline{\Z}\) satisfies \(|\beta-\alpha|_\infty \leq r\) and let \(L=K(\beta)\).
Let \((\rho_\sigma)_\sigma\in \textup{P}=\prod_{\sigma\in\X(K)}\X_\sigma(L)\) such that \(\overline{\rho_\sigma}=\rho_{\overline{\sigma}}\) for all \(\sigma\in\X(K)\), and consider \(z=(\rho_\sigma(\beta))_\sigma\in K_\R\).
Then \(\Kab{z-\alpha}\leq |\beta-\alpha|_\infty \leq r\), so \(\Kab{g(z)}\leq r\).
Consequently we have \(|\rho(g(\beta))|\leq r\) for all \(\rho\in\X(L)\).
Hence 
\[|N_{L/\Q}(g(\beta))| = \prod_{\rho\in\X(L)}|\rho(g(\beta))| \leq r^{[L:\Q]} < 1. \]
As \(g(\beta)\in\overline{\Z}\), we must then have \(g(\beta)=0\).
As \(\beta\) must be a root of \(g\) and \(g\) is non-zero, there can only be finitely many \(\beta\).
\end{proof}

\section{Reduction to exponentially bounded polynomials}\label{sec:exp_bounded}

We now prepare to prove Theorem~\ref{thm:small}.
If there are only finitely many decompositions of an algebraic integer \(\alpha\), then certainly there exists a non-zero polynomial \(f\in\Z[X]\) such that \(f(\beta)=0\) for all decompositions \((\beta,\alpha-\beta)\) of \(\alpha\). 
The goal is to exhibit such a polynomial when \(q(\alpha)\) is small using a lattice argument, similarly to the proof of Theorem~\ref{thm:classical_fekete_sans_polynomial}.
In this section we derive an analytic sufficient condition for a polynomial \(f\) to have this property.

\begin{definition}
Let \(K\) be a number field. 
We define \(\mathcal{S}(K) = \X(K)\times \C\), the coproduct (i.e.\ disjoint union) of measurable spaces of \(\#\X(K)\) copies of \(\C\), where \(\C\) has the standard Lebesgue measurable space structure. 
We write \(\mathcal{M}(K)\) for the set of probability measures \(\mu\) on \(\mathcal{S}(K)\), i.e.\ all \(\mu\) such that \(\mu(\mathcal{S}(K))=1\).
\end{definition}

\begin{definition}\label{def:exponentially_bounded}
Let \(K\) be a number field and \(r\in\R_{>0}\). 
For \(f\in K_\R[Y]\) we say \(f\) is {\em exponentially bounded at radius \(r\)} if for all \(\mu\in \mathcal{M}(K)\) satisfying \(\int |z|^2\dif\mu(\sigma,z)<r^2\) it holds that \(\int \log|\sigma(f)(z)|\dif\mu(\sigma,z) < 0\).
\end{definition}

\begin{proposition}\label{prop:rgood_implies_dec}
Let \(\alpha\in\overline{\Z}\), \(K=\Q(\alpha)\) and \(r>\sqrt{q(\alpha/2)}\).
If \(f\in \mathcal{O}_K[X]\) is exponentially bounded at radius \(r\in\R_{>0}\) as polynomial in the variable \(Y=X-\alpha/2\), then for all \((\beta,\gamma)\in\textup{dec}(\alpha)\) we have \(f(\beta)=0\). 
\end{proposition}
\begin{proof}
Suppose \((\beta,\gamma)\in\textup{dec}(\alpha)\). 
Then \(q(\beta-\alpha/2)\leq q(\alpha/2)<r^2\) by Lemma~\ref{lem:eq_dec_def}.
Let \(L=K(\beta)\) and
\[B=\{(\sigma,\rho(\beta-\alpha/2))\in\mathcal{S}(K)\,|\, \rho\in \X_\sigma(L) \},\]
which has \(\# B = [L:\Q]\) and \(\#(B\cap(\{\sigma\}\times\C))=[L:K]\) for all \(\sigma\in\X(K)\).
Let \(\mu\in\mathcal{M}(K)\) be the uniform probability measure on \(B\) and write \(f_Y\) for \(f\) as a polynomial in the variable \(Y\).
Because \(\int|x|^2\dif\mu(\sigma,x)=q(\beta-\alpha/2)<r^2\) and \(f_Y\) is exponentially bounded at radius \(r\) we get
\begin{align*} 
\log \big( N(f(\beta))^{[L:\Q]} \big) 
&= \log \prod_{\rho\in \X(L)} |\rho(f(\beta))|
= \sum_{\rho\in \X(L)} \log | \rho(f_Y(\beta-\alpha/2)) | \\
&= [L:\Q] \cdot \int\log|\sigma(f_Y)(x)|\dif\mu(\sigma,x) < 0.
\end{align*}
We conclude that \(N(f(\beta))<1\). Since \(f(\beta)\) is integral we have \(f(\beta)=0\), as was to be shown.
\end{proof}

\begin{example}\label{ex:not_convex_exp_bounded}
The set of polynomials of \(K_\R\) exponentially bounded at radius \(r\) is closed under multiplication and symmetric.
However, we will show that it is not convex. 

Let \(r=1\) and \(K=\Q\).
For all \(c\in(-1,1)\) the constant polynomial \(c\) is trivially exponentially bounded at any positive radius, in particular at radius \(1\).
Also the polynomial \(Y^2\) is exponentially bounded: 
For any \(\mu\in\mathcal{M}(K)\) such that \(\int|z|^2\dif\mu(\sigma,z)<1\) we have
\[ \int\log|z^2|\dif\mu(\sigma,z) \leq \log \int|z|^2\dif\mu(\sigma,z)<\log 1 = 0. \]
Here the first inequality is Jensen's inequality for integrals. When \(\mu\) has finite support, this comes down to Lemma~\ref{lem:amgm}.
For \(c\in(-1,1)\) and \(k\in\Z_{\geq 0}\) the product \(cY^{2k}\) of exponentially bounded polynomials at radius \(1\) is exponentially bounded at radius \(1\).
We claim that \(\tfrac{1}{4}(1+Y^{2k})\) for \(k\) sufficiently large, which is a convex combination of \(\tfrac{1}{2}\) and \(\tfrac{1}{2}Y^{2k}\), is not exponentially bounded at radius \(1\). 
Taking \(\mu\in\mathcal{M}(\Q)\) with weight \(\tfrac{1}{5}\) at \(2\) and remaining weight at \(0\) we have \(\int|z|^2\dif\mu(\sigma,z)=\tfrac{4}{5}<1\), yet \(\int\log|\tfrac{1}{4}(1+Y^{2k})|\dif\mu(\sigma,z)=\tfrac{1}{5}\log(1+2^{2k})-\log 4\to \infty\) as \(k\to\infty\).
We conclude that the set of exponentially bounded polynomials at radius \(1\) is not convex.
A similar argument works for all radii and number fields.
\end{example}

\begin{lemma}\label{lem:slope}
Let \(D\subseteq\C\) be a convex subset and let \(f:D\to\C\) be analytic.
Then for distinct \(x,y\in D\) we have
\[ \bigg|\frac{f(x)-f(y)}{x-y}\bigg| \leq \sup_{z\in D} |f'(z)|. \]
\end{lemma}
\begin{proof}
Let \(\gamma:[0,1]\to D\) be the parametrization of the straight line connecting \(x\) and \(y\), which is well-defined since \(D\) is convex. First note that
\begin{align*}
\int_0^1 f'(\gamma(t))\dif t = \int_0^1 f'(tx+(1-t)y)\dif t = \frac{1}{x-y} \Big[ f(tx+(1-t)y)\Big]_{t=0}^1 = \frac{f(x)-f(y)}{x-y}.
\end{align*}
Then
\begin{align*}
\bigg| \frac{f(x)-f(y)}{x-y}\bigg| =\bigg|\int_0^1 f'(\gamma(t)) \dif t \bigg| \leq \int_0^1 \big|f'(\gamma(t))\big| \dif t \leq \int_0^1 \big(\sup_{z\in D} |f'(z)|\big) \dif t = \sup_{z\in D} |f'(z)|,
\end{align*}
as was to be shown.
\end{proof}

We will now translate the measure theoretic property of Definition~\ref{def:exponentially_bounded} to an analytic one.
Our results in the coming sections only depend on the `only if' part of the following equivalence.

\begin{theorem}\label{thm:exp_bounded}
Let \(K\) be a number field, \(0<r<1\) and \(f\in K_\R[Y]\). 
Then \(f\) is exponentially bounded at radius \(r\) if and only if there exist \(a\in\R_{>0}\) such that for all \(\sigma\in\X(K)\) and \(z\in\C\) we have \(|\sigma(f)(z)|\leq \exp(a(|z|^2-r^2))\).
\end{theorem}
\begin{proof}
(\(\Leftarrow\)) Suppose such \(a\) exists. 
Let \(\mu\in\mathcal{M}(K)\) such that \(\int |z|^2 \dif\mu(\sigma,z)<r^2\). 
Then 
\[ \int\log|\sigma(f)(z)| \dif \mu(\sigma,z) \leq \int a(|z|^2-r^2) \dif\mu(\sigma,z) < ar^2-ar^2 = 0, \]
so \(f\) is exponentially bounded at radius \(r\).

(\(\Rightarrow\)) 
Let \(D_0=\{z\in\C\,|\,|z| < r\}\) and \(D_\infty=\{z\in\C\,|\,|z| > r\}\). For \(c\in\{0,\infty\}\) let 
\[A_c=\big\{a\in\R\,\big|\, (\forall\,\rho\in\X(K))\,(\forall\,z\in D_c)\ |\rho(f)(z)|\leq\exp(a(|z|^2-r^2))\big\}.\]

Firstly, we show that \(A_0\) is non-empty.
Let \(\rho\in\X(K)\) and \(z_0\in D_0\), and let \(\mu\in\mathcal{M}(K)\) be the measure with weight \(1\) at \((\rho,z_0)\).
Then \(\int |x|^2 \dif\mu(\sigma,x)=|z_0|^2 < r^2\), so by exponential boundedness
\[|\rho(f)(z_0)|=\exp\Big( \int \log|\sigma(f)(x)|\dif\mu(\sigma,x)\Big) < 1.\] 
It follows that \(0\in A_0\), and even \((-\infty,0]\subseteq A_0\). This argument also shows that \(\rho(f)\) is bounded by \(1\) on the boundary of \(D_0\), the circle of radius \(r\).

Secondly, we show that \(A_\infty\) is non-empty.
Since \(\exp(|z|^2-r^2)\) grows faster than any polynomial, there exists some \(b>r\) such that \(|\rho(f)(z)|\leq \exp(|z|^2-r^2)\) for all \(|z|\geq b\).
Write \(B=\{x\in \C\,|\,|x|\leq b\}\).
Let \(\rho\in\X(K)\) and \(z\in B\cap D_\infty\), and write \(g=\rho(f)\) and \(\theta = z/|z|\).
As remarked at the end of the previous paragraph we have \(|g(r\theta)|\leq 1\), so that
\begin{align*}
\log|g(z)| &\leq \log(1+|g(z)-g(r\theta)|) \leq |g(z)-g(r\theta)| = \frac{|z|^2-r^2}{|z|+r} \cdot \bigg| \frac{g(z)-g(r\theta)}{z-r\theta} \bigg| \\
&\stackrel{*}{\leq} (|z|^2-r^2) \cdot \frac{\sup_{x\in B} |g'(x)|}{2r} \leq a (|z|^2-r^2),
\end{align*}
where \(*\) follows from Lemma~\ref{lem:slope} and \(a\) is the maximum of \(1\) and all \((2r)^{-1} \sup_{x\in B} |\rho(f)'(x)|\) for \(\rho\in\X(K)\). Thus \(a\in A_\infty\).

Thirdly, we show that \(A_0\cap A_\infty\) is non-empty.
Suppose for the sake of contradiction that \(A_0\cap A_\infty\) is empty.
Clearly \(A_0\) and \(A_\infty\) are closed. 
Hence there exist reals that are neither in \(A_0\) nor \(A_\infty\), and let \(a\) be such a real number.
It follows that \(a>0\).
In turn, there exist \(z_0\in D_0\) and \(z_\infty\in D_\infty\) with \(\rho_0,\rho_\infty\in\X(K)\) such that \(|\rho_0(f)(z_0)|>\exp(a(|z_0|^2-r^2))\) and \(|\rho_\infty(f)(z_\infty)|>\exp(a(|z_\infty|^2-r^2))\).
Choose \(t\in(0,1)\) such that \((1-t)|z_0|^2+t|z_\infty|^2 < r^2\) and let \(\mu\) be the measure that assigns weight \(1-t\) to \((\rho_0,z_0)\) and weight \(t\) to \((\rho_\infty,z_\infty)\). Then \(\int |z|^2 \dif\mu(\sigma,z)<r^2\) and thus
\begin{align*}
0 &> \int \log|\sigma(f)(z)|\dif\mu(\sigma,z) = (1-t)\log|\rho_0(f)(z_0)|+t\log|\rho_\infty(f)(z_\infty)|.
\end{align*}
Taking the limit of \(t\) up to \(s\in\R\) such that \((1-s)|z_0|^2+s|z_\infty|^2=r^2\) we get
\[ 0 \geq (1-s)\log|\rho_0(f)(z_0)|+s\log|\rho_\infty(f)(z_\infty)| > a((1-s)|z_0|^2+s|z_\infty|^2-r^2) = 0, \]
a contradiction. Hence \(A_0\cap A_\infty\neq\emptyset\), as was to be shown.

Note that \(D_0\cup D_\infty\) is dense in \(\C\), so any positive \(a\in A_0\cap A_\infty\) gives the inequality we set out to prove.
Suppose \(a\in A_0\cap A_\infty\) is such that \(a \leq 0\).
Thus \(|\rho(f)(z)|\leq \exp(a(|z|^2-r^2))\leq 1\) for all \(z\in D_\infty\) and \(\rho\in\X(K)\), so \(\rho(f)\) is a constant function.
However, as \(|\rho(f)(z)|<1\) for \(z\in D_0\) as shown before, this constant is strictly less than \(1\).
Let \(c\in(0,1)\) be a constant that bounds \(\rho(f)\) for all \(\rho\in\X(K)\). Then \(-r^{-2}\log c \in A_0 \cap A_\infty\) is positive.
Hence \(A_0\cap A_\infty\) always contains a positive element.
\end{proof}

\section{Volume computation}

The next step is to compute the volume of a symmetric convex set of exponentially bounded polynomials.
As in Lemma~\ref{lem:classical_fekete_volume} it suffices for the sake of volume computation to consider the case where the radius is \(1\) and the base field is \(\R\).
In view of Theorem~\ref{thm:exp_bounded}, we consider the unit-ball of the following norm.

\begin{definition}
Let \(\F\) be either \(\R\) or \(\C\). We equip \(\F[Y]\) with the norm
\[\|f\|_\ee = \max_{z\in\C}  \frac{|f(z)|}{\exp(|z|^2)}.\]
\end{definition}

\begin{lemma}\label{lem:phi}
Consider the map \(\phi:\Z_{\geq 0}\to\R_{\geq 0}\) given by
\[ \phi(n) = \begin{cases}\big(\frac{n}{2}\big)! &\textup{if \(n\) is even} \\ \big(\frac{n-1}{2}\big)!\cdot\sqrt\frac{n+1}{2} &\textup{if \(n\) is odd} \end{cases}. \]
Then we have
\[ \log\Big( \prod_{k=0}^n \phi(k) \Big) = \tfrac{1}{4}n^2 \log n - (\tfrac{3}{8}+\tfrac{1}{4}\log 2) n^2 + O(n\log n)\]
and for all \(x\in\R_{\geq 0}\) and \(m\in\Z_{\geq 0}\) we have
\[ \frac{x^{2m+1}}{\phi(2m+1)} \leq  \frac{1}{2}\bigg( \frac{x^{2m}}{\phi(2m)} + \frac{x^{2m+2}}{\phi(2m+1)} \bigg). \]
\end{lemma}
\begin{proof}
Writing out the product we have
\[ \prod_{k=0}^n \phi(k)  = \Bigg( \prod_{m=0}^{\lfloor n/2\rfloor} \phi(2m) \Bigg) \Bigg( \prod_{m=0}^{\lfloor (n-1)/2\rfloor} \phi(2m+1) \Bigg) = \Bigg( \prod_{m=0}^{\lfloor n/2\rfloor} m! \Bigg) \Bigg( \prod_{m=0}^{\lfloor (n-1)/2\rfloor} m! \Bigg) \Bigg( \prod_{m=0}^{\lfloor (n-1)/2\rfloor} \sqrt{m+1} \Bigg). \]
We then apply Proposition~\ref{prop:stirlings} to compute
\begin{align*} 
\smash{\log\Big(  \prod_{k=0}^n \phi(k) \Big)}\vphantom{\scalebox{2}{1}} 
&= \lfloor\tfrac{n}{2}\rfloor^2 \big( \tfrac{1}{2} \log\lfloor\tfrac{n}{2}\rfloor-\tfrac{3}{4} \big) + \lfloor\tfrac{n-1}{2}\rfloor^2 \big( \tfrac{1}{2} \log\lfloor\tfrac{n-1}{2}\rfloor-\tfrac{3}{4} \big) + \lfloor\tfrac{n-1}{2}\rfloor \big( \log\lfloor\tfrac{n-1}{2}\rfloor-1 \big) + O(n\log n) \\
&= \big(\tfrac{n}{2}\big)^2\big(\tfrac{1}{2}\log\tfrac{n}{2}-\tfrac{3}{4}\big) + \big(\tfrac{n}{2}\big)^2\big(\tfrac{1}{2}\log\tfrac{n}{2}-\tfrac{3}{4}\big) + 0 + O(n\log n) \\
&= \tfrac{1}{4}n^2 \log n - (\tfrac{3}{8}+\tfrac{1}{4}\log 2) n^2 + O(n\log n),
\end{align*}
proving the first part. For the second, let \(m\in\Z_{\geq 0}\) and \(x\in\R_{\geq 0}\). Then
\[\frac{1}{\phi(2m)} + \frac{x^2}{\phi(2m+2)} = \frac{1}{m!} \Big( \Big(1-\frac{x}{\sqrt{m+1}} \Big)^2+\frac{2x}{\sqrt{m+1}} \Big) \geq \frac{1}{m!} \frac{2x}{\sqrt{m+1}} = 2 \cdot \frac{x}{\phi(2m+1)},\]
from which the second part follows.
\end{proof}

Recall the notation \(R[X]_n\) from Definition~\ref{def:bounded_degree_poly}, the subset of \(R[X]\) of polynomials of degree strictly less than \(n\).

\begin{proposition}\label{prop:unscaled_volume}
Write \(S=\{f\in\R[Y]\,|\, \|f\|_\ee \leq 1\}\).
Then for \(n\in\Z_{\geq 0}\) we have 
\[\log \vol(S\cap \R[Y]_n) \geq -\tfrac{1}{4}n^2 \log n + (\tfrac{3}{8}+\tfrac{1}{4}\log 2) n^2 + O(n\log n). \]
\end{proposition}
\begin{proof}
Consider \(\phi\) as in Lemma~\ref{lem:phi} and define
\[ T = \Big\{ \sum_{i=0}^\infty f_i Y^i \in \R[Y] \,\Big|\, (\forall i)\ |f_i|\leq \frac{1}{2\phi(i)}  \Big\}. \]
Then for all \(f\in T\) and \(z\in\C\) we have using Lemma~\ref{lem:phi} that
\begin{align*} 
|f(z)| 
&\leq \sum_{i=0}^\infty |f_i|\cdot |z|^i 
\leq \sum_{i=0}^\infty \frac{|z|^i}{2\phi(i)} 
= \frac{1}{2}\Bigg[\sum_{k=0}^\infty \frac{|z|^{2k}}{k!} + \sum_{k=0}^\infty \frac{|z|^{2k+1}}{\phi(2k+1)}\Bigg] \\
&\leq \frac{1}{2} \Bigg[\sum_{k=0}^\infty\frac{|z|^{2k}}{k!} + \sum_{k=0}^\infty\frac{1}{2}\bigg(\frac{|z|^{2k}}{k!} + \frac{|z|^{2k+2}}{(k+1)!} \bigg) \Bigg]
\leq \sum_{k=0}^\infty\frac{|z|^{2k}}{k!} = \exp|z^2|,
\end{align*}
so \(f\in S\) and \(T\subseteq S\). Then by Lemma~\ref{lem:phi} we have
\[\log\vol(T\cap \R[Y]_n) = -n\log 2 - \log\Big( \prod_{k=0}^{n-1} \phi(k)  \Big) = -\tfrac{1}{4}n^2 \log n + (\tfrac{3}{8}+\tfrac{1}{4}\log 2) n^2 + O(n\log n),\]
from which the proposition follows.
\end{proof}

Proposition~\ref{prop:unscaled_volume} is sufficient for our purposes. 
It may interest a reader attempting to improve our results that the lower bound of Proposition~\ref{prop:unscaled_volume} is actually an equality, which we will show in the remainder of this section.

\begin{theorem}[Brunn--Minkowski inequality, Theorem~2.1 in \citep{BrunnMinkowski}]\label{thm:brunn}
Let \(n\in\Z_{\geq 1}\) and let \(A,B\subseteq\R^n\) be bounded non-empty measurable sets. 
Then for all \(t\in[0,1]\) such that 
\[(1-t)A+tB=\{(1-t)a+tb\,|\,a\in A,\,b\in B\}\] 
is measurable we have the inequality
\[\textup{vol}((1-t)A+tB)^{1/n} \geq (1-t)\textup{vol}(A)^{1/n} + t\textup{vol}(B)^{1/n}. \qed\]
\end{theorem}

We will only apply this theorem to compact subsets of \(\R^n\), which are indeed measurable and bounded.
Moreover, for \(A,B\subseteq\R^n\) compact and \(t\in(0,1)\) also the set \((1-t)A+tB\) is compact, hence measurable.

\begin{corollary}\label{cor:brunn}
Let \(n\in\Z_{\geq 0}\), let \(V\subseteq\R^n\) be a subspace and let \(S\subseteq\R^n\) be a symmetric convex body. 
Then the map that sends \(x\in \R^n\) to \(\vol_{V}(V\cap (S-x))\) takes a maximum at \(0\).
\end{corollary}
\begin{proof}
Let \(x\in \R^n\) and write \(m=\dim V\) and \(H_x=V\cap (S-x)\).
If \(m=0\) the corollary holds trivially, so suppose \(m>0\).
Note that \(H_{-x}=-H_x\) since \(S\) and \(V\) are symmetric. 
Because \(S\) and \(V\) are convex we have \(\tfrac{1}{2}H_x + \tfrac{1}{2}H_{-x} \subseteq H_0\).
Hence by Theorem~\ref{thm:brunn} we have
\[\textup{vol}(H_0)^{1/m} \geq \textup{vol}(\tfrac{1}{2} H_x + \tfrac{1}{2} H_{-x})^{1/m} \geq \tfrac{1}{2}\textup{vol}(H_x)^{1/m}+\tfrac{1}{2}\textup{vol}(H_{-x})^{1/m} = \textup{vol}(H_x)^{1/m}, \]
from which the corollary follows.
\end{proof}

Recall the definition of \(\operp\) from Definition~\ref{def:orth_sum}.

\begin{corollary}\label{cor:brunnbound}
Let \(n\in\Z_{\geq 0}\), let \(U,V\subseteq \R^n\) be subspaces such that \(U\operp V = \R^n\) and write \(\pi\) for the projection \(U\operp V\to U\). 
If \(S\subseteq\R^n\) is a symmetric convex body, then \( \vol_{\R^n}(S) \leq \vol_U(\pi S) \cdot \vol_V(S\cap V)\).
\end{corollary}
\begin{proof}
By Corollary~\ref{cor:brunn} we have
\[\vol(S) = \int_{\pi S} \vol_V( V\cap(S-x) )\dif x \leq \int_{\pi S} \vol_V(V\cap S) \dif x = \vol_U(\pi S) \cdot \vol_V(V\cap S).\]
\end{proof}

\begin{theorem}\label{thm:unscaled_volume}
Write \(S=\{f\in\R[Y]\,|\, \|f\|_\ee \leq 1\}\).
Then for \(n\in\Z_{\geq 0}\) we have 
\[\log \vol(S\cap \R[Y]_n) = -\tfrac{1}{4}n^2 \log n + (\tfrac{3}{8}+\tfrac{1}{4}\log 2) n^2 + O(n\log n). \]
\end{theorem}
\begin{proof}
We already proved a lower bound in Proposition~\ref{prop:unscaled_volume}, so it remains to prove an upper bound.
We will inductively show that 
\[\vol(S\cap\R[Y]_n)\leq 2^n\prod_{k=1}^{n-1} \Big(\frac{2\ee}{k}\Big)^{k/2}.\]
It then follows from Proposition~\ref{prop:stirlings} that
\[\log\vol(S\cap\R[Y]_n) \leq n\log2+ \frac{1}{2}\sum_{k=1}^{n-1}k\big(\log (2\ee) -\log k) = -\tfrac{1}{4}n^2 \log n + (\tfrac{3}{8}+\tfrac{1}{4}\log 2) n^2 + O(n\log n).  \]
For \(n=0\) and \(n=1\) the inequality certainly holds.
Now suppose the inequality holds for \(n\geq 1\).
Write \(\R[Y]_{n+1}=(\R Y^n) \operp \R[Y]_n\) and let \(\pi:\R[Y]_{n+1}\to \R Y^n\) be the projection map.
By Corollary~\ref{cor:brunnbound} it suffices to show for \(n>0\) that \(\vol(\pi(S\cap\R[Y]_{n+1})) \leq 2(\frac{n}{2\ee})^{-n/2}\). 
We do this by proving
\[ \pi(S\cap\R[Y]_{n+1}) \stackrel{\textup{(i)}}{\subseteq} S\cap (\R Y^n) \stackrel{\textup{(ii)}}{\subseteq} [-1,+1] \Big(\frac{2\ee}{n}\Big)^{n/2} Y^n.  \]

(i) Suppose \(f\in \R[Y]_{n+1}\) and \(\|f\|_\ee<\|\pi(f)\|_\ee\). 
Since \(\pi(f)\) is a monomial, the function \(z\mapsto|\pi(f)(z)|\exp(-|z|^2)\) takes its maximum on a circle of radius say \(r\).
Then for all \(z\) on this circle we have \(|f(z)|\leq \|f\|_\ee \exp(r^2) < \|\pi(f)\|_\ee\exp(r^2) = |\pi(f)(z)|\).
Hence by Rouch\'e's theorem, the polynomial \(f-\pi(f)\) has as many roots as \(\pi(f)\) in the disk \(\{z\in\C\,|\,|z|\leq r\}\), counting multiplicities.
However, since \(f-\pi(f)\) has degree at most \(n-1\) and \(\pi(f)\) has \(n\) such roots, this is a contradiction.
Hence \(\|\pi(f)\|_\ee\leq\|f\|_\ee\), from which (i) follows.

(ii) Consider the map \(g:\R_{\geq 0}\to\R_{\geq 0}\) given by \(x\mapsto x^n\exp(-x^2)\).
Then
\[\frac{\dif g}{\dif x} = x^{n-1}(n-2x^2)\exp(-x^2) = 0 \ \Longleftrightarrow\ x = 0 \ \lor\ x=\sqrt{n/2}.   \]
Hence \(g\) takes a maximum at \((n/2)^{1/2}\). 
We conclude that \(\|Y^n\|_\ee=g((n/2)^{1/2})=(\tfrac{n}{2\ee})^{n/2}\).
Thus \(\max\{c\in\R\,|\,cY^n\in S\}=(\tfrac{2\ee}{n})^{n/2}\), as was to be shown.

The theorem now follows by induction.
\end{proof}

\section{Proof of the main theorem}\label{sec:fekete}

We are now ready to give a proof of Theorem~\ref{thm:small}.

\begin{proposition}\label{prop:agood_exists}
Let \(R\) be an order of a number field \(K\), let \(\alpha\in K\) and \(0<r^2<\frac{1}{2}\exp(\frac{1}{2})\).
Then there exists some non-zero \(f\in R[X]\) such that \(f(X-\alpha)\) is exponentially bounded at radius \(r\).
\end{proposition}
\begin{proof}
Let \(n\in\Z_{\geq 1}\) and \(d=[K:\Q]\).
Write \(Y=X-\alpha\) and consider the real vector space \(K_\R[Y]_n\), which we equip with an inner product as in Definition~\ref{def:poly_inner_product} with respect to the variable \(Y\).
By Theorem~\ref{thm:det_lattice} and Lemma~\ref{lem:linear_transformations} the lattice \(R[X]_n\) in \(K_\R[Y]_n\) is full rank and has determinant \(\det(R[Y]_n)=\left|\det(R)\right|^n=|\Delta(R)|^{n/2}\).
For \(b\in\R_{\geq0}\) consider 
\begin{align*} 
S_n &= \{f\in K_\R[Y]_n \,|\, (\forall\,\sigma\in\X(K))\,(\forall\,z\in\C)\ |\sigma(f)(z)|\leq\exp(bn(|z|^2-r^2))\} \\
&= \{f\in K_\R[Y]_n \,|\, (\forall\,\sigma\in\X(K))\ \|\exp(bnr^2)\cdot\sigma(f)((bn)^{-1/2}\cdot Y)\|_\ee\leq 1 \}.
\end{align*}
We have a natural orthogonal decomposition \(K_\R\cong \R^u\times \C^v\) for some \(u,v\in\Z_{\geq 0}\) which in turn gives an orthogonal decomposition \(K_\R[Y]_n= (\R[Y]_n)^u \times (\C[Y]_n)^v\).
Note that \(S_n\) is simply a product over \(\sigma\in\X(K)\) of 
\[S_n(\sigma)=\{f\in \F[Y]_n \,|\, \|\exp(bnr^2)\cdot\sigma(f)((bn)^{-1/2}\cdot Y)\|_\ee\leq 1\}\]
where \(\F=\R\) or \(\F=\C\) depending on whether \(\sigma(K)\subseteq\R\).
Then using Lemma~\ref{lem:linear_transformations} and Proposition~\ref{prop:unscaled_volume} we compute
\begin{align*}
\log\vol(S_n) &\geq d  \Big(\big(-\tfrac{1}{4}n^2 \log n + (\tfrac{3}{8}+\tfrac{1}{4}\log 2) n^2 \big) + \big( \tfrac{1}{4}n^2 \log(bn) - bn^2r^2 \big) \Big) + O(n\log n) \\
&= d n^2 \cdot \epsilon(b) + O(n\log n),
\end{align*}
with \(\epsilon(b)=\tfrac{1}{4}\log(2b)+\tfrac{3}{8}-r^2b\). Choosing \(b=(2r)^{-2}\) we get \(\epsilon(b)=\tfrac{1}{4}(\tfrac{1}{2}-\log(2r^2)) > 0\).
Hence 
\[ \log\Big(\frac{\vol(S_n)}{2^{d n}\cdot |\Delta(R)|^{n/2}}\Big) \geq d n^2 \cdot \epsilon(b) + O(n\log n) \to \infty \quad(\text{as }n\to\infty). \]
Thus by Minkowski's theorem there exists for \(n\) sufficiently large some non-zero \(g\in S_n\cap R[X]\).
Because \(g\in S_n\), this polynomial is exponentially bounded at radius \(r\) by Theorem~\ref{thm:exp_bounded}.
\end{proof}

\noindent\textbf{Theorem~\ref{thm:small}. }{\em If \(\gamma\in\overline{\Z}\) satisfies \(q(\gamma)<2\exp(\tfrac{1}{2})\approx 3.297\), then \(\gamma\) has only finitely many decompositions.}

\begin{proof}
Let \(\alpha=\gamma/2\), let \(K=\Q(\gamma)\) and let \(R\) be some order in \(K\).
Choose \(r\in\R_{>0}\) such that \(q(\alpha)<r^2<\tfrac{1}{2}\exp(\tfrac{1}{2})\).
Then by Proposition~\ref{prop:agood_exists} there exists some non-zero polynomial \(f\in R[X]\) which as polynomial in \(Y=X-\alpha\) is exponentially bounded at radius \(r\). Hence by Proposition~\ref{prop:rgood_implies_dec} all \((\beta,\gamma-\beta)\in\text{dec}(\gamma)\) satisfy \(f(\beta)=0\).
As \(f\) has only finitely many roots, the theorem follows.
\end{proof}

\section{Remarks on the proof of the main theorem} 
\label{sec:proof_remarks}

In this section we briefly discuss the proof of Theorem~\ref{thm:small} and make some practical remarks for explicit computation.

The proof of Theorem~\ref{thm:small} proceeds in the following steps:
\begin{enumerate}
\item We determine a sufficient condition for a polynomial to have all decompositions of \(\alpha\) as roots.
\item We translate this condition into an analytic one.
\item We determine the volume of a symmetric convex set of polynomials satisfying this condition.
\item We apply Minkowski's convex body theorem to find integral polynomials in this set.
\end{enumerate}
Theorem~\ref{thm:exp_bounded} suggests that step (2) can hardly be improved upon.
By Theorem~\ref{thm:unscaled_volume} we correctly computed the volume of our symmetric convex set in step (3).
However, in order to make it convex we fixed the constant \(a\) that comes out of Theorem~\ref{thm:exp_bounded}.
It is easy to verify that we indeed made an optimal choice of \(a\) in Proposition~\ref{prop:agood_exists}, although that does not guarantee we chose the best convex subset.
If the weakest link in the proof is step (4), we likely require a completely different approach. 
It should be noted however that Minkowski's convex body theorem is powerful enough to prove the classical Theorem~\ref{thm:classical_fekete}.

A piece of information we did not exploit in our proof of the main theorem is the following symmetry.

\begin{proposition}
For all \(\alpha\in\overline{\Z}\) we have an involution \(x\mapsto \alpha - x\) on \(\overline{\Z}\) which induces action on \(\textup{dec}(\alpha)\) given by \((\beta,\gamma)\mapsto(\gamma,\beta)\). \qed
\end{proposition}

\begin{corollary}\label{cor:dimension_reduction}
Let \(\alpha\in\overline{\Z}\) and let \(\Z[\alpha]\subseteq R\) be an order of \(\Q(\alpha)\).
For all \(f\in R[X]\) such that \(f(\beta)=0\) for all \((\beta,\gamma)\in\textup{dec}(\alpha)\), also \(g=f(\alpha-X)\in R[X]\) satisfies \(g(\beta)=0\) for all \((\beta,\gamma)\in\textup{dec}(\alpha)\). \qed
\end{corollary}

Corollary~\ref{cor:dimension_reduction} turns the involution on \(\overline{\Z}\) into an \(R\)-algebra automorphism on \(R[X]\).
We can incorporate this automorphism in our proof of Theorem~\ref{thm:small}.

\begin{proposition}\label{prop:dimension_reduction}
Let \(\alpha\in\overline{\Z}\), let \(\Z[\alpha]\subseteq R\) be an order of \(K=\Q(\alpha)\) and let \(r\in\R_{>0}\).
For all \(f\in R[X]\) such that \(f\) as a polynomial in \(Y=X-\alpha/2\) is exponentially bounded at radius \(r\), so is \(f(X)\cdot f(\alpha-X)\in K[Y^2]\).
\end{proposition}
\begin{proof}
Note that the involution \(X\mapsto\alpha-X\) is with respect to \(Y\) given by \(Y\mapsto -Y\). 
Hence if \(f\) as a polynomial in \(Y\) is exponentially bounded at radius \(r\), then so is \(f(\alpha-X)\). 
As noted in Example~\ref{ex:not_convex_exp_bounded}, the set of polynomials exponentially bounded at \(r\) is closed under multiplication. 
Hence \(g=f(X)\cdot f(\alpha-X)\) is exponentially bounded at radius \(r\). 
Now \(g\) is invariant under \(Y\mapsto -Y\), meaning all coefficients at odd degree monomials in \(Y\) are zero, i.e.\ \(g\in K[Y^2]\).
\end{proof}

An interesting question to ask is how dissimilar \(f\) and \(f(\alpha-X)\) can be for exponentially bounded \(f\).
Certainly both should have \(\beta\) as root for all \((\beta,\gamma)\in\textup{dec}(\alpha)\) by Corollary~\ref{cor:dimension_reduction}.
In the context of finding `small' \(f\) algorithmically it seems that often \(f\) and \(f(\alpha-X)\) are the same (up to sign).

As a consequence of Proposition~\ref{prop:dimension_reduction}, in the proof of Theorem~\ref{thm:small} we may look at the lattice \(R[X(\alpha-X)]\) in \(K_\R[Y^2]\) instead of \(R[X]\) in \(K_\R[Y]\).
The effect is two-fold.
Firstly, it simplifies the volume computation of Proposition~\ref{prop:unscaled_volume}, as we no longer require the ad-hoc function \(\phi\) from Lemma~\ref{lem:phi}.
Secondly, any integral polynomial in our symmetric body can be found in a lower dimensional lattice in Proposition~\ref{prop:agood_exists}.
This follows from the suggested changes to Proposition~\ref{prop:unscaled_volume}, but can heuristically be seen as follows.
If \(f\) is a solution in the original lattice \(R[X]\), then \(f(X)\cdot f(\alpha-X)\) exists in our new lattice \(R[X(\alpha-X)]\) at the same dimension.
However, as discussed before, \(f\) might already be an element of \(R[X(\alpha-X)]\) anyway, so we could have already found \(f\) at half the dimension. 
Neither of these changes have an effect on the quality of our theoretical results.
However, when we want to compute decompositions in practice, the latter is very useful.

\section{Computing decompositions of algebraic integers of small square-norm}\label{sec:compute_small}

From the proof of Theorem~\ref{thm:small} one easily derives the following result.

\begin{theorem}\label{thm:compute_small}
There exists an algorithm that takes as input an element \(\alpha\in\overline{\Z}\), and decides whether \(q(\alpha)<2\sqrt{\ee}\) and if so, computes all decompositions of \(\alpha\).
\end{theorem}
\begin{proof}
Clearly \(q(\alpha)\neq2\sqrt{\ee}\) as the latter is not algebraic. However, both are computable, and after finitely many steps of approximation we can decide whether \(q(\alpha)<2\sqrt{\ee}\).
If so, we can even compute a rational \(r\) such that \(q(\alpha)<r^2<2\sqrt{\ee}\).

We have an explicit formula for a lower bound on the volume of the set \(S\) as defined in the proof of Proposition~\ref{prop:unscaled_volume}.
Thus we can compute a sufficiently large \(n\) such that Minkowski's theorem, as in the proof of Proposition~\ref{prop:agood_exists}, guarantees the existence of a non-zero lattice point in \(S\). We may then simply enumerate all lattice points to eventually find a polynomial \(f\) as in Proposition~\ref{prop:agood_exists}. We determine using Theorem~\ref{alg:all} the monic irreducible factors of \(f\) and decide using Corollary~\ref{cor:decide_decomp} for each whether it gives rise to a decomposition of \(\alpha\). 
\end{proof}

\section{Working out an example}\label{sec:example}

We will now work out an example proving an algebraic integer \(\alpha\) is indecomposable, proceeding along the lines of our lattice algorithm.
Of course, this result will be trivial when \(q(\alpha)\leq 2\) as we have seen in Proposition~\ref{prop:q_is_2}, so we will choose \(\alpha\) such that \(q(\alpha)>2\).
On the other hand, the algorithms from Section~\ref{sec:compute_small} terminate faster the smaller \(q(\alpha)\) is, so for this example we will consider \(\alpha=\sqrt[3]{3}\) with \(q(\alpha)=3^{2/3}\approx 2.080\). \\

\noindent\textbf{Setup.} Let \(\alpha=\sqrt[3]{3}\) and let \(r^2=6/11\), so that \(q(\alpha/2) < r^2 < \tfrac{1}{2}\exp(\tfrac{1}{2})\). %(approximately \(.52 < .55 < .82\).
Write \(K=\Q(\alpha)\) and \(R=\Z[\alpha]\) and consider the ring \(R[X]\).
Writing \(Y=X-\alpha/2\), we are looking for a polynomial \(f\in R[X]\) such that \(f\) as polynomial in \(Y\) is exponentially bounded at radius \(r\).
However, writing \(Z=X(\alpha-X)\) we may instead look for such a polynomial in \(R[Z]\), as follows from Proposition~\ref{prop:dimension_reduction}. \\

\noindent\textbf{Finding a polynomial.} 
It is quite involved to systematically find short vectors in a lattice. 
Instead we will employ a more ad-hoc approach, more along the lines of Theorem~\ref{thm:szego}.
We guess that our polynomial \(f\) will be monic in \(Z\) of some degree \(n\).
We start with \(Z^n\) and then greedily subtract \(\Z[\alpha]\)-multiples of lower degree powers of \(Z\) such that the resulting polynomial in \(Y\) becomes `small', i.e.\ has small coefficients under every embedding \(K\to\C\) with lower degree terms weighing more heavily.
Efficitively, we are applying a rounding function in the sense of Lemma~\ref{lem:polynomial_rounding}.
Note that \(Z=-Y^2+\alpha^2/4\).
Similarly as in the proof of Theorem~\ref{thm:szego}, taking \(n=4\) the \(Y^6\) term becomes integral, which is useful.
Thus we will try \(n=4\). We compute:
{\renewcommand{\arraystretch}{1.2}%
\begin{align*}
\begin{tabu}{rrrrrrrr}
Z^4&&=& Y^8 & - \alpha^2 Y^6 & +\frac{9}{8} \alpha Y^4 & - \frac{9}{16} Y^2 & + \frac{9}{256}\alpha^2 \\
&\alpha^2 Z^3&=&   & -\alpha^2 Y^6 & +\frac{9}{4} \alpha Y^4 & - \frac{27}{16} Y^2 & + \frac{9}{64}\alpha^2 \\ \hline
Z^4&-\alpha^2 Z^3 &=& Y^8 &  & - \frac{9}{8}\alpha Y^4 & + \frac{9}{8}Y^2 & - \frac{27}{256}\alpha^2 \\
&&-\alpha Z^2 =& &  & - \alpha Y^4 & + \frac{3}{2}Y^2 & - \frac{3}{16}\alpha^2 \\ \hline
Z^4&-\alpha^2 Z^3 &+\alpha Z^2 =& Y^8 &  & - \frac{1}{8} \alpha Y^4 & -\frac{3}{8} Y^2 & + \frac{21}{256}\alpha^2 \\
\end{tabu}
\end{align*}
The remaining coefficients with respect to \(Y\) look pretty small in every embedding \(K\to\C\), so we guess 
\[f(Y)=Y^8  - \tfrac{1}{8} \alpha Y^4  -\tfrac{3}{8} Y^2  + \tfrac{21}{256}\alpha^2 = Z^4-\alpha^2 Z^3+\alpha Z^2 \in R[Z].\] 
is going to be exponentially bounded at radius \(r\) as polynomial in \(Y\). \\

\noindent\textbf{Proving exponentially boundedness.} If we take \(b:\R_{\geq 0}\to \R_{\geq 0}\) given by
\[ b(w) = w^4+\tfrac{1}{8}\cdot3^{1/3}\cdot w^2+\tfrac{3}{8}\cdot w+\tfrac{21}{256}\cdot 3^{2/3},\]
then for all \(\sigma\in \X(K)\) and \(z\in\C\) we have \(|\sigma(f)(z)|\leq b(|z|^2)\).
To prove that \(f\) is exponentially bounded at radius \(r\) it suffices to find \(a\in\R\) such that \(b(w)\leq \exp(a(w-r^2))\) for all \(w\in\R_{\geq 0}\).
Because \(b(0)=\frac{21}{256} 3^{2/3}\), we must have \(a\leq -\log(b(0))/r^2\approx 3.4\).
We will try \(a=3\) for simplicity.
Consider the function \(B(w)=b(w)\cdot\exp(-a(w-r^2))\), for which we want to show \(B(w)\leq 1\) for all \(w\). 
Then
\[ 0 = B'(w) = \exp(-a(w-r^2))( b'(w)-ab(w) ) \]
if and only if \(ab(w)-b'(w)=0\). 
Since the latter is simply a polynomial equation we will find using standard techniques that it has no positive real roots.
We compute:
\begin{align*} 
ab(w)-b'(w)&= 3w^4 - 4w^3 + \tfrac{3}{8}3^{1/3}w^2 + (\tfrac{9}{8}-\tfrac{1}{4}3^{1/3})w + (\tfrac{63}{256}3^{2/3} - \tfrac{3}{8}) \\
& > 3w^4 - 4w^3 + \tfrac{3}{5}w^2+\tfrac{3}{4}w+\tfrac{1}{4}.
\end{align*}
For \(1\leq w\) we get \(ab(w)-b'(w)>3w^4-4w^3+\tfrac{3}{2}=w^2(3^{1/2}w-2\cdot3^{-1/2})^2+(\tfrac{3}{2}-\tfrac{4}{3}w^2)\geq 0\) and for \(0<w\leq 1\) we get 
\(ab(w)-b'(w)>3w^4-4w^3+\tfrac{3}{2}w^2=3w^2(w^2-\tfrac{4}{3}w+\tfrac{1}{2})\geq3w^2(w-2^{-1/2})^2\geq 0\).
Hence \(B\) has no local maxima besides possibly at \(0\), and because \(B(w)\to 0\) as \(w\to\infty\) we conclude that \(B\) takes a maximum at \(0\).
Therefore \(b\) is bounded by \(w\mapsto\exp(a(w-r^2))\) and thus \(f\) is exponentially bounded at radius \(r\). \\

\noindent\textbf{Finding decompositions.}
Writing \(f\) as a polynomial in \(X\) we get
\begin{align*}
f&=X^8 - 4\alpha X^7 + 7\alpha^2 X^6 - 21 X^5 + 13\alpha X^4 - 5\alpha^2X^3 + 3X^2 \\
&=X^2 \cdot (\alpha- X)^2 \cdot (X^4 - 2\alpha X^3 + 2\alpha^2 X^2 - 3 X + \alpha).
\end{align*}
By Proposition~\ref{prop:rgood_implies_dec} all decompositions of \(\alpha\) can be found among the roots of \(f\).
The factors \(X\) and \(\alpha-X\) correspond to the trivial decompositions \((0,\alpha)\) and \((\alpha,0)\) of \(\alpha\).
The polynomial \(h=X^4 - 2\alpha X^3 + 2\alpha^2 X^2 - 3 X + \alpha\) is irreducible as it is Eisenstein at the prime \((\alpha)\).
Let \(\beta\in\overline{\Z}\) be a root of \(h\).
By Lemma~\ref{lem:Nq_ineq} we have \(q(\beta)\geq N(\beta)^2=N(h(0))^2=3^{2/3}=q(\alpha)\). 
We can only have \(q(\beta)+q(\alpha-\beta)\leq q(\alpha)\) if \(q(\alpha-\beta)=0\), i.e.\ \(\alpha=\beta\), which is impossible.
Hence \(\alpha\) is indecomposable by Lemma~\ref{lem:eq_dec_def}.

%\newpage
\label{sec:end-capacity-small}

\section{Enumeration of indecomposable algebraic integers of degree 3}\label{sec:enum3}

In this section we discuss our attempt to compute the indecomposable algebraic integers of degree 3 and derive Theorem~\ref{thm:counting}.
These results are obtained by a computer program written in Sage.

We will write \(f_\alpha\in\Z[x]\) for the minimal polynomial of \(\alpha\in\overline{\Z}\).
We save ourselves some work by only considering \(\alpha\) up to `trivial isometries' of \(\overline{\Z}\), namely those of \(\mu_\infty\rtimes\textup{Gal}(\overline{\Q})\) as in Lemma~\ref{lem:isometries}.
Using Proposition~\ref{prop:classifying_candidates} we can compute a set of \(5525\) polynomials among which we can find all minimal polynomials of the indecomposable algebraic integers of degree 3.
Among those \(5525\) polynomials \(f\) only \(700\) are in fact irreducible with \(q(\alpha)\leq 4\) for all roots \(\alpha\) of \(f\).
We already eliminated the Galois action by considering minimal polynomials instead of elements, and by choosing only one element of each \(\mu_\infty\)-orbit \(\{f,-f(-x)\}\) we eliminate the action of \(\mu_\infty\), and end up with `only' \(350\) polynomials to check. 
Of those \(350\), there are \(27\) polynomials \(f_\alpha\) such that \(q(\alpha)<2\), so that \(\alpha\) is indecomposable by Proposition~\ref{prop:q_is_2}.

{\em Small degree decompositions.}
For \(95\) polynomials \(f\), the roots \(\alpha\) of \(f\) have a non-trivial decomposition in the ring of integers of \(\Q(\alpha)\).
For \(116\) of the remaining polynomials \(f_\alpha\) we can find a non-trivial decomposition \((\beta,\alpha-\beta)\) of \(\alpha\) with \(\beta\) in the ring of integers of a degree 2 extension of \(\Q(\alpha)\).
Of those \(116\) there are \(84\) for which the minimal polynomial \(g_\beta\) of \(\beta\) over \(\Q(\alpha)\) is of the form \(x^2-\alpha x\pm1\), a polynomial we encountered in the proof of Lemma~\ref{lem:help_deg_2_dec}.
The remaining \(32\) polynomials and corresponding decompositions can be found in Table~\ref{tab:deg2dec}. 
We are now left with \(112\) polynomials to check.

\begin{table}[H]
\resizebox{\columnwidth}{!}{
\begin{tabular}{|l|l|l|}\hline $q(\alpha)$ & $f_\alpha$ & $g_\beta$ \\\hline
$2.2844$ & $x^{3} - x^{2} + 3 x + 1$ & $x^{2} - \alpha x + \frac{1}{2} (\alpha^{2} + 1) $ \\
$2.5198$ & $x^{3} + 4$ & $x^{2} - \alpha x + \frac{1}{2} \alpha^{2}$ \\
$2.6178$ & $x^{3} + 2 x^{2} + 4 x + 4$ & $x^{2} - \alpha x + \frac{1}{2} \alpha^{2}$ \\
$2.7246$ & $x^{3} + 2 x + 4$ & $x^{2} - \alpha x + \frac{1}{2} \alpha^{2}$ \\
$2.7850$ & $x^{3} - 2 x^{2} + 4$ & $x^{2} - \alpha x + \frac{1}{2} \alpha^{2}$ \\
$2.8582$ & $x^{3} + 2 x^{2} - x + 2$ & $x^{2} - \alpha x + \frac{1}{2} (\alpha^{2} + \alpha)$ \\
$2.8657$ & $x^{3} - 2 x^{2} + 3 x + 2$ & $x^{2} - \alpha x + \frac{1}{2} (\alpha^{2} - \alpha)$ \\
$2.9073$ & $x^{3} + 2 x^{2} - 2 x + 1$ & $x^{2} - \alpha x - \alpha$ \\
$2.9379$ & $x^{3} + 2 x^{2} + x + 4$ & $x^{2} - \alpha x + \frac{1}{2} (\alpha^{2} + \alpha)$ \\
$2.9380$ & $x^{3} - x^{2} - x + 5$ & $x^{2} - \alpha x + \frac{1}{2} (\alpha^{2} - 1)$ \\
$2.9516$ & $x^{3} + x^{2} + x + 5$ & $x^{2} - \alpha x + \frac{1}{2} (\alpha^{2} + 1)$ \\
$3.0000$ & $x^{3} + x^{2} - 4 x + 1$ & $x^{2} - \left(\alpha + 1\right) x + 1$ \\
$3.1102$ & $x^{3} + x^{2} - 2 x + 4$ & $x^{2} - \alpha x + \frac{1}{2} (\alpha^{2} + \alpha)$ \\
$3.2714$ & $x^{3} + 2 x^{2} + 4$ & $x^{2} - \alpha x + \frac{1}{2} \alpha^{2}$ \\
$3.2849$ & $x^{3} - 2 x^{2} + 2 x + 4$ & $x^{2} - \alpha x + \frac{1}{2} \alpha^{2}$ \\
$3.2981$ & $x^{3} + 2 x^{2} - 2 x + 2$ & $x^{2} - \left(\alpha - 1\right) x - (\alpha - 1)$ \\
\hline
\end{tabular}
\begin{tabular}{|l|l|l|}\hline $q(\alpha)$ & $f_\alpha$ & $g_\beta$ \\\hline
$3.3321$ & $x^{3} + 3 x + 5$ & $x^{2} - \alpha x + \frac{1}{3} (\alpha^{2} + \alpha + 1)$ \\
$3.3333$ & $x^{3} - 2 x^{2} - 3 x + 1$ & $x^{2} - \left(\alpha - 1\right) x + 1$ \\
$3.3378$ & $x^{3} + x^{2} - 4 x + 3$ & $x^{2} - \left(\alpha - 1\right) x - (\alpha - 1)$ \\
$3.3853$ & $x^{3} + 4 x + 4$ & $x^{2} - \alpha x + \frac{1}{2} \alpha^{2} + 1$ \\
$3.4436$ & $x^{3} - 2 x^{2} - x + 6$ & $x^{2} - \alpha x + \frac{1}{2} (\alpha^{2} - \alpha)$ \\
$3.5716$ & $x^{3} + 2 x^{2} + 6 x + 4$ & $x^{2} - \alpha x + \frac{1}{2} \alpha^{2} + 1$ \\
$3.6432$ & $x^{3} + 2 x^{2} + 6 x + 2$ & $x^{2} - \left(\alpha + 1\right) x - 1$ \\
$3.6830$ & $x^{3} + x^{2} - 4 x + 4$ & $x^{2} - \alpha x + \frac{1}{2} \alpha^{2} + \frac{1}{2} \alpha - 1$ \\
$3.6839$ & $x^{3} - x^{2} - 3 x + 7$ & $x^{2} - \alpha x + \frac{1}{2} (\alpha^{2} - 1)$ \\
$3.7159$ & $x^{3} + x^{2} + 3 x + 7$ & $x^{2} - \alpha x + \frac{1}{2} (\alpha^{2} + 1)$ \\
$3.7362$ & $x^{3} + 2 x^{2} - 3 x + 2$ & $x^{2} - \alpha x - \alpha$ \\
$3.7962$ & $x^{3} + 2 x^{2} + 6 x + 1$ & $x^{2} - \left(\alpha - 1\right) x - \alpha - 1$ \\
$3.8854$ & $x^{3} - 3 x + 7$ & $x^{2} - \alpha x + \frac{1}{3} (\alpha^{2} - \alpha + 1)$ \\
$3.9111$ & $x^{3} - 2 x^{2} + 5 x + 2$ & $x^{2} - \left(\alpha - 1\right) x - 1$ \\
$3.9938$ & $x^{3} + x^{2} - 5 x + 4$ & $x^{2} - \left(\alpha + 1\right) x + 1$ \\
$4.0000$ & $x^{3} - 2 x^{2} - 4 x + 1$ & $x^{2} - \alpha x + \alpha$ \\
\hline
\end{tabular}}
\captionsetup{width=.9\linewidth}
\caption{\label{tab:deg2dec}
A table of minimal polynomials \(f_\alpha\) together with a minimal polynomial \(g_\beta\) of degree 2 over \(\Q(\alpha)\) such that \((\beta,\alpha-\beta)\) is a non-trivial decomposition of \(\alpha\).}
\end{table}

{\em Large degree decompositions.}
To find decompositions in higher degree extensions we implemented a lattice algorithm.
Since we are interested in finding only one decomposition instead of all of them, and since verifying whether something is a decomposition is computationally easy, we can get away with a lot of heuristics.
For \((\beta,\alpha-\beta)\in\textup{dec}(\alpha)\) we have, on average of squares over all embeddings of \(\Q(\alpha,\beta)\) in \(\C\), that \(|\beta-\alpha/2|\leq \sqrt{q(\alpha/2)}=r\) by Lemma~\ref{lem:eq_dec_def}. Hence if we write \(g_\beta = \sum_i c_i(x-\alpha/2)^i\) we have that \(\sum_i |c_i| r^i\) should be small. 
It is useful for our lattice algorithm to instead consider the 2-norm \(( \sum_i r^i \sum_{\sigma} |\sigma(c_i)|^2)^{1/2}\) and hope this does not affect the quality of our results for the worse.
We enumerate small polynomials \(\epsilon\in\Q(\alpha)[x]\) of degree less than \(d\in\Z_{>0}\) such that \((x-\alpha/2)^d-x^d + \epsilon\) is in the lattice of integral polynomials, and thus \((x-\alpha/2)^d+\epsilon\) is monic, integral and small.
We then verify for each of those whether they induce a decomposition of \(\alpha\).
The \(41\) polynomials \(f_\alpha\) for which this method has found a non-trivial decomposition \((\beta,\alpha-\beta)\) of \(\alpha\) with \(g_\beta\) of degree greater than \(2\) are listed in Table~\ref{tab:largedec} together with the polynomial \(g_\beta\) found.
This leaves \(71\) polynomials to check and gives an upper bound of \(6\cdot98=588\) on the number of indecomposable algebraic integers of degree 3.

\begin{table}
\resizebox{\columnwidth}{!}{
\begin{tabular}{|l|l|p{14cm}|} \hline $q(\alpha)$ & $f_\alpha$ & $g_\beta$ \\\hline
$2.9240$ & $x^{3} + 5$ & $x^{4} - 2 \alpha x^{3} + 2 \alpha^{2} x^{2} + 5 x - \alpha$ \\
$3.0103$ & $x^{3} - x^{2} + 5$ & $x^{6} - 3 \alpha x^{5} + (4 \alpha^{2} + 1) x^{4} + (-3 \alpha^{2} - 2 \alpha + 15) x^{3} + (3 \alpha^{2} - 7 \alpha - 6) x^{2} + (\alpha^{2} + \alpha + 5) x - \alpha + 1$ \\
$3.2595$ & $x^{3} + 2 x^{2} - x + 3$ & $x^{4} - 2 \alpha x^{3} + (2 \alpha^{2} + \alpha - 1) x^{2} + (\alpha^{2} + 3) x - \alpha + 1$ \\
$3.2624$ & $x^{3} - 2 x^{2} + 3 x + 3$ & $x^{4} - 2 \alpha x^{3} + (2 \alpha^{2} - \alpha + 1) x^{2} + (-\alpha^{2} + 2 \alpha + 3) x - \alpha$ \\
$3.3019$ & $x^{3} + 6$ & $x^{6} - 3 \alpha x^{5} + 4 \alpha^{2} x^{4} + 18 x^{3} - 8 \alpha x^{2} + 2 \alpha^{2} x + 1$ \\
$3.3378$ & $x^{3} + x + 6$ & $x^{4} - 2 \alpha x^{3} + (\frac{3}{2} \alpha^{2} - \frac{1}{2} \alpha) x^{2} + (\frac{1}{2} \alpha^{2} + \frac{1}{2} \alpha + 3) x + 1$ \\
$3.3656$ & $x^{3} - 2 x^{2} + 4 x + 2$ & $x^{4} - 2 \alpha x^{3} + (2 \alpha^{2} - \alpha + 2) x^{2} + (-\alpha^{2} + 2 \alpha + 2) x - \alpha$ \\
$3.3709$ & $x^{3} - x^{2} + 2 x + 5$ & $x^{6} - 3 \alpha x^{5} + 4 \alpha^{2} x^{4} + (-3 \alpha^{2} + 6 \alpha + 15) x^{3} + (-\alpha^{2} - 10 \alpha - 5) x^{2} + 3 \alpha^{2} x + 3$ \\
$3.3863$ & $x^{3} + x^{2} - 3 x + 4$ & $x^{8} - 4 \alpha x^{7} + 7 \alpha^{2} x^{6} + (7 \alpha^{2} - 21 \alpha + 28) x^{5} + (18 \alpha^{2} - 30 \alpha + 17) x^{4} + (19 \alpha^{2} - 30 \alpha + 32) x^{3} + (13 \alpha^{2} - 24 \alpha + 19) x^{2} + (6 \alpha^{2} - 9 \alpha + 8) x + \alpha^{2} - 2 \alpha + 2$ \\
$3.3895$ & $x^{3} - x^{2} + 6$ & $x^{6} - 3 \alpha x^{5} + (4 \alpha^{2} + 1) x^{4} + (-3 \alpha^{2} - 2 \alpha + 18) x^{3} + (3 \alpha^{2} - 8 \alpha - 7) x^{2} + (\alpha^{2} + \alpha + 6) x - \alpha + 1$ \\
$3.4190$ & $x^{3} - 2 x + 6$ & $x^{6} - 3 \alpha x^{5} + (4 \alpha^{2} - 1) x^{4} + (-4 \alpha + 18) x^{3} + (\alpha^{2} - 8 \alpha) x^{2} + 2 \alpha^{2} x + 1$ \\
$3.4338$ & $x^{3} - x^{2} + 4 x + 3$ & $x^{8} - 4 \alpha x^{7} + (7 \alpha^{2} - 1) x^{6} + (-7 \alpha^{2} + 31 \alpha + 21) x^{5} + (-17 \alpha^{2} - 30 \alpha - 13) x^{4} + (19 \alpha^{2} - 27 \alpha - 24) x^{3} + (4 \alpha^{2} + 26 \alpha + 13) x^{2} + (-5 \alpha^{2} + 3 \alpha + 3) x - 2 \alpha - 1$ \\
$3.4438$ & $x^{3} + x^{2} + 6$ & $x^{10} - 5 \alpha x^{9} + 11 \alpha^{2} x^{8} + (14 \alpha^{2} + 84) x^{7} + (11 \alpha^{2} - 69 \alpha + 68) x^{6} + (44 \alpha^{2} - 36 \alpha + 30) x^{5} + (29 \alpha^{2} - 5 \alpha + 99) x^{4} + (5 \alpha^{2} - 24 \alpha + 42) x^{3} + (3 \alpha^{2} - 6 \alpha) x^{2} - 1$ \\
$3.4537$ & $x^{3} + 2 x + 6$ & $x^{6} - 3 \alpha x^{5} + (4 \alpha^{2} + 1) x^{4} + (4 \alpha + 18) x^{3} + (-\alpha^{2} - 8 \alpha) x^{2} + 2 \alpha^{2} x + 1$ \\
$3.4767$ & $x^{3} + x^{2} - 2 x + 5$ & $x^{14} - 7 \alpha x^{13} + 23 \alpha^{2} x^{12} + (47 \alpha^{2} - 94 \alpha + 235) x^{11} + (200 \alpha^{2} - 467 \alpha + 333) x^{10} + (695 \alpha^{2} - 761 \alpha + 1040) x^{9} + (1143 \alpha^{2} - 1912 \alpha + 2739) x^{8} + (1840 \alpha^{2} - 3030 \alpha + 3435) x^{7} + (2256 \alpha^{2} - 3290 \alpha + 4258) x^{6} + (1960 \alpha^{2} - 3100 \alpha + 3990) x^{5} + (1343 \alpha^{2} - 2101 \alpha + 2601) x^{4} + (667 \alpha^{2} - 1022 \alpha + 1300) x^{3} + (226 \alpha^{2} - 354 \alpha + 447) x^{2} + (49 \alpha^{2} - 75 \alpha + 95) x + 5 \alpha^{2} - 8 \alpha + 10$ \\
$3.4858$ & $x^{3} + 2 x^{2} + 2 x + 6$ & $x^{6} - 3 \alpha x^{5} + (4 \alpha^{2} + 1) x^{4} + (6 \alpha^{2} + 4 \alpha + 18) x^{3} + (4 \alpha^{2} - 3 \alpha + 16) x^{2} + (3 \alpha^{2} - 2 \alpha + 6) x + \alpha^{2} + 2$ \\
$3.5833$ & $x^{3} - 2 x^{2} + 6$ & $x^{8} - 4 \alpha x^{7} + (7 \alpha^{2} + \alpha) x^{6} + (-17 \alpha^{2} + 42) x^{5} + (25 \alpha^{2} - 26 \alpha - 75) x^{4} + (-14 \alpha^{2} + 36 \alpha + 72) x^{3} + (-\alpha^{2} - 22 \alpha - 29) x^{2} + (4 \alpha^{2} + 5 \alpha) x - \alpha^{2} + 2$ \\
$3.6133$ & $x^{3} + x^{2} - x + 6$ & $x^{6} - 3 \alpha x^{5} + (4 \alpha^{2} + \alpha) x^{4} + (\alpha^{2} - 3 \alpha + 18) x^{3} + (\alpha^{2} - 8 \alpha - 1) x^{2} + (2 \alpha^{2} + \alpha) x + 1$ \\
$3.6361$ & $x^{3} + 2 x^{2} - x + 4$ & $x^{6} - 3 \alpha x^{5} + (4 \alpha^{2} - 1) x^{4} + (6 \alpha^{2} - \alpha + 12) x^{3} + (5 \alpha^{2} - 8 \alpha + 11) x^{2} + (4 \alpha^{2} - 4 \alpha + 4) x + \alpha^{2} - \alpha + 2$ \\
$3.6370$ & $x^{3} - 2 x^{2} + 3 x + 4$ & $x^{6} - 3 \alpha x^{5} + (4 \alpha^{2} + 1) x^{4} + (-6 \alpha^{2} + 7 \alpha + 12) x^{3} + (3 \alpha^{2} - 14 \alpha - 10) x^{2} + (2 \alpha^{2} + 5 \alpha + 4) x - \alpha^{2} + 1$ \\
$3.6521$ & $x^{3} + 2 x^{2} + 5$ & $x^{6} - 3 \alpha x^{5} + 4 \alpha^{2} x^{4} + (6 \alpha^{2} + 15) x^{3} + (5 \alpha^{2} - 7 \alpha + 13) x^{2} + (4 \alpha^{2} - 3 \alpha + 5) x + \alpha^{2} + 2$ \\
$3.6593$ & $x^{3} + 7$ & $x^{3} - 2 \alpha x^{2} + \alpha^{2} x + 1$ \\
$3.6785$ & $x^{3} - x^{2} - x + 7$ & $x^{10} - 5 \alpha x^{9} + (11 \alpha^{2} + 1) x^{8} + (-14 \alpha^{2} - 18 \alpha + 98) x^{7} + (30 \alpha^{2} - 68 \alpha - 81) x^{6} + (16 \alpha^{2} + 27 \alpha + 140) x^{5} + (-13 \alpha^{2} - 70 \alpha + 18) x^{4} + (24 \alpha^{2} + 8 \alpha - 14) x^{3} + (-8 \alpha^{2} - 5 \alpha + 31) x^{2} + (2 \alpha^{2} - 2 \alpha - 7) x + 1$ \\
$3.6878$ & $x^{3} - x + 7$ & $x^{6} - 3 \alpha x^{5} + 4 \alpha^{2} x^{4} + (-3 \alpha + 21) x^{3} + (\alpha^{2} - 9 \alpha) x^{2} + 2 \alpha^{2} x + 1$ \\
$3.6915$ & $x^{3} + x + 7$ & $x^{6} - 3 \alpha x^{5} + 4 \alpha^{2} x^{4} + (3 \alpha + 21) x^{3} + (-\alpha^{2} - 9 \alpha) x^{2} + 2 \alpha^{2} x + 1$ \\
$3.6939$ & $x^{3} + x^{2} + x + 7$ & $x^{6} - 3 \alpha x^{5} + (4 \alpha^{2} + \alpha + 1) x^{4} + (\alpha^{2} + \alpha + 21) x^{3} + (-9 \alpha - 1) x^{2} + (2 \alpha^{2} + \alpha) x + 1$ \\
$3.7123$ & $x^{3} - 2 x^{2} + 4 x + 3$ & $x^{6} - 3 \alpha x^{5} + 4 \alpha^{2} x^{4} + (-6 \alpha^{2} + 12 \alpha + 9) x^{3} + (-15 \alpha - 7) x^{2} + (4 \alpha^{2} + \alpha) x - \alpha^{2} + 2 \alpha + 1$ \\
$3.7228$ & $x^{3} + 2 x^{2} + x + 6$ & $x^{6} - 3 \alpha x^{5} + 4 \alpha^{2} x^{4} + (6 \alpha^{2} + 3 \alpha + 18) x^{3} + (4 \alpha^{2} - 5 \alpha + 16) x^{2} + (3 \alpha^{2} - 3 \alpha + 6) x + \alpha^{2} + 2$ \\
$3.7238$ & $x^{3} - x^{2} + 2 x + 6$ & $x^{8} - 4 \alpha x^{7} + 7 \alpha^{2} x^{6} + (-7 \alpha^{2} + 14 \alpha + 42) x^{5} + (-4 \alpha^{2} - 35 \alpha - 26) x^{4} + (15 \alpha^{2} + 8 \alpha - 6) x^{3} + (-6 \alpha^{2} + 6 \alpha + 20) x^{2} + (\alpha^{2} - 4 \alpha - 6) x + \alpha + 1$ \\
$3.7345$ & $x^{3} + x^{2} - 3 x + 5$ & $x^{3} - 2 \alpha x^{2} + (\frac{3}{2} \alpha^{2} - \frac{1}{2}) x + \frac{1}{2} \alpha^{2} - \alpha + \frac{3}{2}$ \\
$3.7479$ & $x^{3} - x^{2} + 7$ & $x^{6} - 3 \alpha x^{5} + (4 \alpha^{2} + 1) x^{4} + (-3 \alpha^{2} - 2 \alpha + 21) x^{3} + (3 \alpha^{2} - 9 \alpha - 9) x^{2} + (\alpha^{2} + 2 \alpha + 7) x - \alpha$ \\
$3.7664$ & $x^{3} - 2 x + 7$ & $x^{6} - 3 \alpha x^{5} + (4 \alpha^{2} - 1) x^{4} + (-4 \alpha + 21) x^{3} + (\alpha^{2} - 9 \alpha) x^{2} + 2 \alpha^{2} x + 1$ \\
$3.7771$ & $x^{3} - 2 x^{2} + x + 6$ & $x^{6} - 3 \alpha x^{5} + 4 \alpha^{2} x^{4} + (-6 \alpha^{2} + 3 \alpha + 18) x^{3} + (4 \alpha^{2} - 10 \alpha - 15) x^{2} + (4 \alpha + 6) x - \alpha - 1$ \\
$3.7949$ & $x^{3} + 2 x + 7$ & $x^{6} - 3 \alpha x^{5} + (4 \alpha^{2} + 1) x^{4} + (4 \alpha + 21) x^{3} + (-\alpha^{2} - 9 \alpha) x^{2} + 2 \alpha^{2} x + 1$ \\
$3.7995$ & $x^{3} + x^{2} + 7$ & $x^{6} - 3 \alpha x^{5} + (4 \alpha^{2} + \alpha) x^{4} + (\alpha^{2} + 21) x^{3} + (-9 \alpha - 1) x^{2} + (2 \alpha^{2} + \alpha) x + 1$ \\
$3.8253$ & $x^{3} + x^{2} - 2 x + 6$ & $x^{6} - 3 \alpha x^{5} + 4 \alpha^{2} x^{4} + (3 \alpha^{2} - 6 \alpha + 18) x^{3} + (4 \alpha^{2} - 10 \alpha + 9) x^{2} + (3 \alpha^{2} - 5 \alpha + 6) x + \alpha^{2} - \alpha + 1$ \\
$3.8560$ & $x^{3} + 2 x^{2} + 2 x + 7$ & $x^{6} - 3 \alpha x^{5} + (4 \alpha^{2} + 1) x^{4} + (6 \alpha^{2} + 4 \alpha + 21) x^{3} + (4 \alpha^{2} - 4 \alpha + 19) x^{2} + (3 \alpha^{2} - 3 \alpha + 7) x + \alpha^{2} + 2$ \\
$3.8739$ & $x^{3} - x^{2} + x + 7$ & $x^{3} - 2 \alpha x^{2} + (\frac{3}{2} \alpha^{2} + \frac{1}{2}) x - \frac{1}{2} \alpha^{2} + \frac{5}{2}$ \\
$3.9466$ & $x^{3} - 2 x^{2} + 7$ & $x^{6} - 3 \alpha x^{5} + (4 \alpha^{2} - 1) x^{4} + (-6 \alpha^{2} + 2 \alpha + 21) x^{3} + (4 \alpha^{2} - 10 \alpha - 19) x^{2} + (\alpha^{2} + 5 \alpha + 7) x - \alpha^{2} - \alpha + 1$ \\
$3.9928$ & $x^{3} + 2 x^{2} - x + 5$ & $x^{6} - 3 \alpha x^{5} + (4 \alpha^{2} - 1) x^{4} + (6 \alpha^{2} - \alpha + 15) x^{3} + (5 \alpha^{2} - 9 \alpha + 13) x^{2} + (4 \alpha^{2} - 4 \alpha + 5) x + \alpha^{2} - \alpha + 1$ \\
$3.9948$ & $x^{3} - x^{2} + 3 x + 6$ & $x^{6} - 3 \alpha x^{5} + (4 \alpha^{2} + 1) x^{4} + (-3 \alpha^{2} + 7 \alpha + 18) x^{3} + (-\alpha^{2} - 12 \alpha - 9) x^{2} + (3 \alpha^{2} + 3 \alpha) x - \alpha^{2} + \alpha + 1$ \\ \hline
\end{tabular}
}
\captionsetup{width=.9\linewidth}
\caption{\label{tab:largedec}
A table of minimal polynomials \(f_\alpha\) together with a minimal polynomial \(g_\beta\) of degree greater than 2 over \(\Q(\alpha)\) such that \((\beta,\alpha-\beta)\) is a non-trivial decomposition of \(\alpha\).}
\end{table}

{\em Indecomposables.}
On the other hand, we want to prove that certain \(\alpha\) are indecomposable.
To this end, we implemented a lattice algorithm similar to that of Theorem~\ref{thm:compute_small}.
To hopefully speed up the algorithm we also apply the dimension reducing symmetry trick discussed in Section~\ref{sec:proof_remarks}.
Writing \(R\) for the ring of integers of \(\Q(\alpha)\) and \(z=x(\alpha-x)\), we enumerate short \(g\in R[z]\) for which we verify whether \(g\) as polynomial in \(y=x-\alpha/2\) is exponentially bounded. 
The \(32\) polynomials \(f_\alpha\) for which we found such a \(g\) proving indecomposability of \(\alpha\) are listed in Table~\ref{tab:indec}.
We present \(g\) in factored form for compactness. 
This leaves \(39\) polynomials undetermined and gives a lower bound of \(6\cdot 59=354\) on the number of indecomposable algebraic integers of degree \(3\).

\begin{table}
\resizebox{\columnwidth}{!}{
\begin{tabular}{|l|l|p{14cm}|}\hline $q(\alpha)$ & $f_\alpha$ & $g$ \\\hline
$2.0000$ & $x^{3} - 2 x^{2} - x + 1$ & $\left(\alpha^{2} - 2 \alpha\right) \cdot x \cdot (x - \alpha) \cdot (x^{2} - \alpha x + \alpha^{2} - \alpha - 1)$ \\
$2.0347$ & $x^{3} + x^{2} + 3 x + 2$ & $x^{2} \cdot (x - \alpha)^{2} \cdot (x^{2} - \alpha x - 1)$ \\
$2.0780$ & $x^{3} - 2 x^{2} + x + 2$ & $(x - 1) \cdot (x - \alpha + 1) \cdot x^{2} \cdot (x - \alpha)^{2}$ \\
$2.0801$ & $x^{3} + 3$ & $x^{2} \cdot (x - \alpha)^{2} \cdot (x^{4} - 2 \alpha x^{3} + 2 \alpha^{2} x^{2} - 3 x + \alpha)$ \\
$2.0826$ & $x^{3} + 2 x^{2} + 3 x + 3$ & $\left(\alpha^{2} + \alpha + 1\right) \cdot x \cdot (x - \alpha) \cdot (x^{2} - \alpha x + \frac{1}{3} \alpha^{2})$ \\
$2.0872$ & $x^{3} - x^{2} - x + 3$ & $\left(\alpha + 1\right) \cdot x \cdot (x - \alpha) \cdot (x^{4} - 2 \alpha x^{3} + \left(2 \alpha^{2} - \alpha + 1\right) x^{2} + \left(-2 \alpha + 3\right) x + \frac{1}{2} \alpha^{2} - \alpha + \frac{1}{2})$ \\
$2.0905$ & $x^{3} - x^{2} - 2 x + 3$ & $ \alpha  \cdot x \cdot (x - \alpha) \cdot (x^{4} - 2 \alpha x^{3} + \left(2 \alpha^{2} - 1\right) x^{2} + \left(-\alpha^{2} - \alpha + 3\right) x + \frac{1}{3} \alpha^{2} - \frac{1}{3} \alpha - \frac{2}{3})$ \\
$2.0967$ & $x^{3} + x^{2} + x + 3$ & $ \alpha  \cdot x^{2} \cdot (x - \alpha)^{2} \cdot (x^{4} - 2 \alpha x^{3} + \left(\frac{5}{3} \alpha^{2} - \frac{1}{3} \alpha - \frac{1}{3}\right) x^{2} + \left(\alpha^{2} + \alpha + 2\right) x + 1)$ \\
$2.1102$ & $x^{3} + x^{2} + 2 x + 3$ & $x \cdot (x - \alpha) \cdot (x^{4} - 2 \alpha x^{3} + 2 \alpha^{2} x^{2} + \left(\alpha^{2} + 2 \alpha + 3\right) x + 1)$ \\
$2.1279$ & $x^{3} - x + 3$ & $x^{2} \cdot (x - \alpha)^{2} \cdot (x^{4} - 2 \alpha x^{3} + 2 \alpha^{2} x^{2} + \left(-\alpha + 3\right) x - \alpha)$ \\
$2.1390$ & $x^{3} + x + 3$ & $2 \cdot x \cdot (x - \alpha) \cdot (x^{6} - 3 \alpha x^{5} + \left(4 \alpha^{2} + \frac{1}{2}\right) x^{4} + \left(2 \alpha + 9\right) x^{3} + \left(-\frac{1}{2} \alpha^{2} - 4 \alpha\right) x^{2} + \alpha^{2} x + \frac{1}{2})$ \\
$2.1626$ & $x^{3} - x^{2} + 3$ & $\left(\alpha - 1\right) \cdot x \cdot (x - \alpha) \cdot (x^{2} - \alpha x + \frac{1}{3} \alpha^{2})$ \\
$2.1815$ & $x^{3} - 2 x^{2} - x + 3$ & $x \cdot (x - \alpha) \cdot (x^{2} - \alpha x + 1)$ \\
$2.1904$ & $x^{3} - x^{2} + 2 x + 2$ & $\left(\alpha - 1\right) \cdot x \cdot (x - \alpha) \cdot (x^{4} - 2 \alpha x^{3} + \frac{3}{2} \alpha^{2} x^{2} + \left(-\frac{1}{2} \alpha^{2} + \alpha + 1\right) x - \frac{1}{4} \alpha^{2} - \frac{1}{2})$ \\
$2.1973$ & $x^{3} + 2 x^{2} + 2 x + 3$ & $\left(\alpha + 1\right) \cdot x^{2} \cdot (x - \alpha)^{2} \cdot (x^{4} - 2 \alpha x^{3} + 2 \alpha^{2} x^{2} + \left(2 \alpha^{2} + 2 \alpha + 3\right) x + \frac{1}{2} \alpha^{2} + \frac{1}{2} \alpha + \frac{3}{2})$ \\
$2.2230$ & $x^{3} + 2 x^{2} + 4 x + 2$ & $\left(\alpha^{2} + \alpha + 1\right) \cdot x \cdot (x - \alpha) \cdot (x^{2} - \alpha x + \frac{1}{3} \alpha^{2} + \frac{1}{3}) \cdot (x^{2} - \alpha x + \alpha^{2} + \alpha + 2)$ \\
$2.2309$ & $x^{3} + x^{2} + 3$ & $\left(\alpha + 1\right) \cdot x \cdot (x - \alpha) \cdot (x^{2} - \alpha x + \frac{1}{3} \alpha^{2})$ \\
$2.2512$ & $x^{3} - 2 x + 3$ & $ \alpha  \cdot x \cdot (x - \alpha) \cdot (x^{2} - \alpha x + \frac{1}{3} \alpha^{2} + \frac{1}{3}) \cdot (x^{4} - 2 \alpha x^{3} + \left(\alpha^{2} - \alpha\right) x^{2} + \alpha^{2} x + 1)$ \\
$2.3044$ & $x^{3} + x^{2} - 2 x + 2$ & $\left(\alpha + 2\right) \cdot x^{3} \cdot (x - \alpha)^{3} \cdot (x^{2} - \alpha x + \frac{1}{2} \alpha^{2} + \frac{1}{2} \alpha)$ \\
$2.3681$ & $x^{3} + 2 x^{2} + 4 x + 1$ & $\left(\alpha^{2} + \alpha + 4\right) \cdot x \cdot (x + 1) \cdot (x - \alpha - 1) \cdot (x - \alpha) \cdot (x^{2} - \alpha x + \frac{1}{2} \alpha^{2} + \frac{1}{2} \alpha + \frac{1}{2}) \cdot (x^{4} - 2 \alpha x^{3} + \left(\frac{7}{5} \alpha^{2} - \frac{2}{5} \alpha - \frac{1}{5}\right) x^{2} + \left(\frac{6}{5} \alpha^{2} + \frac{9}{5} \alpha + \frac{2}{5}\right) x + \frac{1}{5} \alpha^{2} + \frac{4}{5} \alpha + \frac{2}{5})$ \\
$2.4206$ & $x^{3} + 2 x^{2} + 2$ & $\left(\alpha + 1\right) \cdot x^{2} \cdot (x - \alpha)^{2} \cdot (x^{4} - 2 \alpha x^{3} + \left(\frac{4}{3} \alpha^{2} - \frac{2}{3} \alpha - \frac{1}{3}\right) x^{2} + \left(\frac{4}{3} \alpha^{2} + \frac{1}{3} \alpha + \frac{2}{3}\right) x + \frac{2}{3} \alpha^{2} + \frac{2}{3} \alpha + \frac{1}{3}) \cdot (x^{4} - 2 \alpha x^{3} + \left(2 \alpha^{2} + \alpha\right) x^{2} + \left(\alpha^{2} + 2\right) x - \alpha)$ \\
$2.4239$ & $x^{3} + 2 x^{2} - x + 1$ & $\left(\alpha^{2} + 2 \alpha - 2\right) \cdot x \cdot (x - \alpha) \cdot (x^{2} - \alpha x - \alpha) \cdot (x^{2} - \alpha x + \frac{1}{3} \alpha^{2} + \frac{1}{3} \alpha + \frac{1}{3}) \cdot (x^{4} - 2 \alpha x^{3} + \left(2 \alpha^{2} + \alpha - 1\right) x^{2} + \left(\alpha^{2} + 1\right) x - \alpha)$ \\
$2.4300$ & $x^{3} - 2 x^{2} + 2 x + 2$ & $\left(\alpha - 1\right) \cdot x \cdot (x - \alpha) \cdot (x^{2} - \alpha x + \frac{1}{3} \alpha^{2} - \frac{1}{3} \alpha + \frac{1}{3}) \cdot (x^{8} - 4 \alpha x^{7} + \left(7 \alpha^{2} + \alpha\right) x^{6} + \left(-17 \alpha^{2} + 14 \alpha + 14\right) x^{5} + \left(17 \alpha^{2} - 35 \alpha - 24\right) x^{4} + \left(-2 \alpha^{2} + 30 \alpha + 20\right) x^{3} + \left(-7 \alpha^{2} - 11 \alpha - 3\right) x^{2} + \left(5 \alpha^{2} - \alpha - 2\right) x - \alpha^{2} + \alpha + 1)$ \\
$2.4436$ & $x^{3} - 2 x^{2} + 3 x + 1$ & $x \cdot (x - \alpha) \cdot (x^{4} - 2 \alpha x^{3} + \left(2 \alpha^{2} - \alpha + 2\right) x^{2} + \left(-\alpha^{2} + \alpha + 1\right) x - 1) \cdot (x^{6} - 3 \alpha x^{5} + \left(4 \alpha^{2} - \alpha\right) x^{4} + \left(-4 \alpha^{2} + 9 \alpha + 3\right) x^{3} + \left(-2 \alpha^{2} - 4 \alpha - 1\right) x^{2} + \left(2 \alpha^{2} - 3 \alpha - 1\right) x + \alpha)$ \\
$2.4517$ & $x^{3} + x^{2} - x + 3$ & $\left(\alpha^{2} + 2 \alpha - 2\right) \cdot x^{2} \cdot (x - \alpha)^{2} \cdot (x^{2} - \alpha x + \frac{1}{2} \alpha^{2} - \frac{1}{2}) \cdot (x^{6} - 3 \alpha x^{5} + \left(\frac{179}{46} \alpha^{2} + \frac{6}{23} \alpha + \frac{1}{46}\right) x^{4} + \left(\frac{52}{23} \alpha^{2} - \frac{65}{23} \alpha + \frac{192}{23}\right) x^{3} + \left(\frac{44}{23} \alpha^{2} - \frac{101}{23} \alpha + \frac{51}{23}\right) x^{2} + \left(\frac{67}{46} \alpha^{2} - \frac{16}{23} \alpha + \frac{51}{46}\right) x + \frac{4}{23} \alpha^{2} - \frac{5}{23} \alpha + \frac{13}{23})$ \\
$2.4960$ & $x^{3} + 2 x^{2} + x + 3$ & $\left(3 \alpha^{2} + 9 \alpha + 5\right) \cdot x^{3} \cdot (x - \alpha)^{3} \cdot (x^{4} - 2 \alpha x^{3} + \left(\alpha^{2} - 1\right) x^{2} + \alpha x - \alpha^{2} - 1) \cdot (x^{4} - 2 \alpha x^{3} + \left(2 \alpha^{2} + \alpha\right) x^{2} + \left(\alpha^{2} + \alpha + 3\right) x - \alpha) \cdot (x^{4} - 2 \alpha x^{3} + \left(\frac{10}{7} \alpha^{2} - \frac{2}{7} \alpha - \frac{1}{7}\right) x^{2} + \left(\frac{8}{7} \alpha^{2} + \frac{4}{7} \alpha + \frac{9}{7}\right) x + \frac{2}{7} \alpha^{2} + \frac{1}{7} \alpha + \frac{4}{7})^{2}$ \\
$2.5248$ & $x^{3} - x^{2} - 2 x + 4$ & $\left(\frac{1}{2} \alpha^{2} - \frac{1}{2} \alpha + 1\right) \cdot x \cdot (x - \alpha) \cdot (x - \frac{1}{2} \alpha)^{2} \cdot (x^{2} + \left(-\alpha - 1\right) x + \frac{1}{2} \alpha^{2} + \frac{1}{2} \alpha) \cdot (x^{2} - \alpha x + 1) \cdot (x^{2} + \left(-\alpha + 1\right) x + \frac{1}{2} \alpha^{2} - \frac{1}{2} \alpha)$ \\
$2.5324$ & $x^{3} + x^{2} + 2 x + 4$ & $\left(-\alpha + 1\right) \cdot x \cdot (x - \alpha) \cdot (x - \frac{1}{2} \alpha)^{2} \cdot (x^{2} - \alpha x + \frac{1}{2} \alpha^{2}) \cdot (x^{4} - 2 \alpha x^{3} + \left(\frac{3}{2} \alpha^{2} - \frac{1}{2} \alpha\right) x^{2} + \left(\alpha^{2} + \alpha + 2\right) x + 1)$ \\
$2.5426$ & $x^{3} + x^{2} + x + 4$ & $ \alpha  \cdot x^{3} \cdot (x - \alpha)^{3} \cdot (x^{4} - 2 \alpha x^{3} + 2 \alpha^{2} x^{2} + \left(\alpha^{2} + \alpha + 4\right) x - \alpha + 1) \cdot (x^{4} - 2 \alpha x^{3} + \left(\frac{3}{2} \alpha^{2} - \frac{1}{2} \alpha - \frac{1}{2}\right) x^{2} + \left(\alpha^{2} + \alpha + 2\right) x + 1)^{2}$ \\
$2.5601$ & $x^{3} - x + 4$ & $\left(-\frac{1}{2} \alpha^{2} - \frac{3}{2} \alpha\right) \cdot x \cdot (x - \alpha) \cdot (x^{4} - 2 \alpha x^{3} + \left(\frac{3}{2} \alpha^{2} - \frac{1}{2} \alpha\right) x^{2} + \left(\frac{1}{2} \alpha^{2} - \frac{1}{2} \alpha + 2\right) x - \frac{1}{2} \alpha + \frac{1}{2}) \cdot (x^{4} - 2 \alpha x^{3} + \left(\frac{3}{2} \alpha^{2} + \frac{1}{2} \alpha\right) x^{2} + \left(-\frac{1}{2} \alpha^{2} - \frac{1}{2} \alpha + 2\right) x - 1) \cdot (x^{4} - 2 \alpha x^{3} + \left(\frac{8}{5} \alpha^{2} + \frac{1}{5} \alpha - \frac{1}{5}\right) x^{2} + \left(-\frac{1}{5} \alpha^{2} - \frac{2}{5} \alpha + \frac{12}{5}\right) x + \frac{1}{10} \alpha^{2} - \frac{3}{10} \alpha - \frac{1}{5})$ \\
$2.6335$ & $x^{3} + x^{2} - 3 x + 2$ & $\left(\alpha^{2} + 2 \alpha\right) \cdot x^{2} \cdot (x - \alpha)^{2} \cdot (x^{2} + \left(-\alpha - 1\right) x + 1) \cdot (x^{2} + \left(-\alpha + 1\right) x - \alpha + 1) \cdot (x^{2} - \alpha x + \frac{1}{2} \alpha^{2} + \frac{1}{2} \alpha - \frac{1}{2})^{2} \cdot (x^{4} - 2 \alpha x^{3} + \left(\frac{3}{2} \alpha^{2} - \frac{1}{2} \alpha - \frac{1}{2}\right) x^{2} + \left(\alpha^{2} - \alpha + 1\right) x + \frac{1}{2} \alpha^{2} - \frac{1}{2} \alpha - \frac{1}{2})$ \\
$2.6377$ & $x^{3} - 2 x^{2} - x + 4$ & $\left(-\alpha^{2} + \alpha\right) \cdot x^{2} \cdot (x - \alpha)^{2} \cdot (x^{2} - \alpha x + 1)^{2} \cdot (x^{2} - \alpha x + \frac{1}{2} \alpha^{2} - \frac{1}{2})^{2} \cdot (x^{8} - 4 \alpha x^{7} + \left(7 \alpha^{2} - 1\right) x^{6} + \left(-14 \alpha^{2} - 4 \alpha + 28\right) x^{5} + \left(18 \alpha^{2} - 9 \alpha - 35\right) x^{4} + \left(-8 \alpha^{2} + 8 \alpha + 24\right) x^{3} + \left(-\alpha^{2} - 3 \alpha - 3\right) x^{2} + \left(2 \alpha^{2} - 4\right) x - \frac{1}{2} \alpha^{2} + \frac{1}{2} \alpha + 1)$ \\
\hline\end{tabular}
}
\captionsetup{width=.9\linewidth}
\caption{\label{tab:indec}
A table of minimal polynomials \(f_\alpha\) together with a polynomial \(g\) over \(\Q(\alpha)\) such that \(g(x+\alpha/2)\) is exponentially bounded at radius \(\sqrt{q(\alpha/2)}\), proving \(\alpha\) is indecomposable.}
\end{table}

\vspace{6pt}

\noindent\textbf{Theorem~\ref{thm:counting}. }{\em There are exactly \(2\) indecomposable algebraic integers of degree \(1\), there are exactly \(14\) of degree \(2\), and there are at least \(354\) and at most \(588\) of degree \(3\).}

\begin{proof}
The degree \(1\) case is obvious: \(1\) and \(-1\) are the only indecomposable integers. The degree \(2\) case is Theorem~\ref{thm:deg2_indec}.
The bounds for degree \(3\) are the result of the computation in this section.
\end{proof}

\newpage
\section*{Acknowledgements}

First and foremost I would like to thank H.W. Lenstra, my PhD advisor, for his help with and for his suggestions and corrections to this document.
We would like to thank T. Chinberg for his proof of Theorem~\ref{thm:large} and his suggestion to frame exponentially bounded polynomials in terms of probability measures in Section~\ref{sec:exp_bounded}. We would like to thank O. Berrevoets and B. Kadets for their help in proving Theorem~\ref{thm:onno}.

\bibliographystyle{apa}
\bibliography{references}

\end{document}